\documentclass[11pt]{amsart}

\usepackage{amstext}
\usepackage{amsthm,euscript}
\usepackage{amscd}
\usepackage{amssymb}
\usepackage{amsmath}
\usepackage{mathtools}
\usepackage[all]{xy}
\usepackage{paralist}
\usepackage{enumitem}
\usepackage{moreenum}
\usepackage{scalerel,stackengine}
\usepackage[colorlinks=true, pdfstartview=FitV, linkcolor=blue,
citecolor=blue,urlcolor=blue,breaklinks=true]{hyperref}
\usepackage{subfiles} 

\usepackage{lmodern} 
\usepackage{xcolor}
\usepackage{tikz-cd}
\usepackage{arydshln}

\usepackage{forloop}
\newcounter{fonts}
\let\eeee\edef
\forloop{fonts}{1}{\value{fonts} < 27}{%
\expandafter\eeee\csname \Alph{fonts}\Alph{fonts}\endcsname{\noexpand\mathbb{\Alph{fonts}}} 
\expandafter\eeee\csname bb\Alph{fonts}\endcsname{\noexpand\mathbb{\Alph{fonts}}} 
\expandafter\eeee\csname b\Alph{fonts}\endcsname{\noexpand\mathbf{\Alph{fonts}}} 
\expandafter\eeee\csname u\Alph{fonts}\endcsname{\noexpand\mathbf{\Alph{fonts}}} 
\expandafter\eeee\csname bf\Alph{fonts}\endcsname{\noexpand\mathbf{\Alph{fonts}}} 
\expandafter\eeee\csname bf\alph{fonts}\endcsname{\noexpand\mathbf{\alph{fonts}}} 
\expandafter\eeee\csname u\Alph{fonts}\Alph{fonts}\endcsname{\noexpand\mathbf{\Alph{fonts}}_0} 
\expandafter\eeee\csname c\Alph{fonts}\endcsname{\noexpand\mathcal{\Alph{fonts}}} 
\expandafter\eeee\csname cal\Alph{fonts}\endcsname{\noexpand\mathcal{\Alph{fonts}}} 
\expandafter\eeee\csname fr\Alph{fonts}\endcsname{\noexpand\mathfrak{\Alph{fonts}}} 
\expandafter\eeee\csname fr\alph{fonts}\endcsname{\noexpand\mathfrak{\alph{fonts}}} 
\expandafter\eeee\csname s\Alph{fonts}\endcsname{\noexpand\mathscr{\Alph{fonts}}} 
\expandafter\eeee\csname eu\Alph{fonts}\endcsname{\noexpand\EuScript{\Alph{fonts}}} 
\expandafter\eeee\csname sf\Alph{fonts}\endcsname{{\noexpand\sf \Alph{fonts}}} 
\expandafter\eeee\csname sf\alph{fonts}\endcsname{{\noexpand\sf \alph{fonts}}} 
\expandafter\eeee\csname ul\Alph{fonts}\endcsname{\noexpand\underline{\noexpand\rm \Alph{fonts}}} 
}

\swapnumbers 

\newcommand{\arxiv}[1]{\href{http://arxiv.org/abs/#1}{\tt arXiv:\nolinkurl{#1}}}
\newtheorem{thm}[subsection]{Theorem}
\newtheorem*{thm*}{Theorem}

\newtheorem{prop}[subsection]{Proposition}
\newtheorem{lem}[subsection]{Lemma}

\newtheorem{cor}[subsection]{Corollary}

\theoremstyle{definition}
\newtheorem{defn}[subsection]{Definition}
\newtheorem*{defn*}{Definition}

\newtheorem{remark}[subsection]{Remark}

\newtheorem{remarks}[subsection]{Remarks}

\newtheorem{example}[subsection]{Example}

\newtheorem{examples}[subsection]{Examples}

\newtheorem{question}[subsection]{Question}

\numberwithin{equation}{subsection}

\newcommand\sm{\smallskip}

\newcommand{\lv}[1]{} 

\newcounter{question} 
\newcounter{suggestion}

\newcommand{\arXiv}[1]{\href{http://arxiv.org/abs/#1}{\tt arXiv:\nolinkurl{#1}}}

\newcommand{\simlgr}{\buildrel \sim \over \longrightarrow}

\newcommand{\simla}{\buildrel \sim \over \longleftarrow}

\newcommand\ch{^{\scriptscriptstyle\vee}}

\newcommand{\fppf}{_\mathrm{fppf}}
\newcommand{\et}{_\mathrm{\acute{e}t}}
\newcommand{\Zar}{_\mathrm{Zar}}

\newcommand{\ti}{^\times}

\def\co{\colon}
\def\ot{\otimes} 
\def\op{^{\rm op}}

\newcommand{\me}{^{-1}}
\def\dar[#1]{\ar@<2pt>[#1]\ar@<-2pt>[#1]}
\def\wtl{\widetilde}

\stackMath
\newcommand\reallywidehat[1]{%
\savestack{\tmpbox}{\stretchto{%
  \scaleto{%
    \scalerel*[\widthof{\ensuremath{#1}}]{\kern.1pt\mathchar"0362\kern.1pt}%
    {\rule{0ex}{\textheight}}
  }{\textheight}%
}{2.4ex}}%
\stackon[-6.9pt]{#1}{\tmpbox}%
}

\newcommand{\Alt}{{\operatorname{Alt}}}

\newcommand{\Aut}{\operatorname{Aut}}

\newcommand{\End}{\operatorname{End}}

\newcommand{\GL}{{\operatorname{GL}}}

\newcommand{\Hom}{\operatorname{Hom}}

\newcommand{\cHom}{\mathcal{H}\hspace{-0.4ex}\textit{o\hspace{-0.2ex}m}} 
\newcommand{\cEnd}{\mathcal{E}\hspace{-0.4ex}\textit{n\hspace{-0.2ex}d}} 
\newcommand{\cSym}{\mathcal{S}\hspace{-0.5ex}\textit{y\hspace{-0.3ex}m}}
\newcommand{\cSymd}{\mathcal{S}\hspace{-0.5ex}\textit{y\hspace{-0.3ex}m\hspace{-0.2ex}d}}
\newcommand{\cSkew}{\mathcal{S}\hspace{-0.4ex}\textit{k\hspace{-0.2ex}e\hspace{-0.3ex}w}}
\newcommand{\cAlt}{\mathcal{A}\hspace{-0.1ex}\textit{$\ell$\hspace{-0.3ex}t}}

\newcommand{\Id}{\operatorname{Id}}
\newcommand{\Ima}{\operatorname{Im}}

\newcommand{\Ker}{\operatorname{Ker}}

\newcommand{\Mat}{{\operatorname{M}}}

\newcommand{\PGL}{{\operatorname{PGL}}}
\newcommand{\Pic}{{\operatorname{Pic}}}

\newcommand{\rank}{\operatorname{rank}}

\newcommand{\Spec}{\operatorname{Spec}}

\newcommand{\Skew}{{\operatorname{Skew}}}

\newcommand{\Span}{\operatorname{Span}}

\newcommand{\Symm}{{\operatorname{Sym}}}
\newcommand{\Symd}{{\operatorname{Symd}}}

\newcommand{\Trd}{\operatorname{Trd}}

\newcommand{\cOS}{\calO_S}

\newcommand{\uGL}{\mathbf{GL}}

\newcommand{\m}{\mathfrak m}

\newcommand{\Sch}{\mathfrak{Sch}} 

\newcommand{\Sets}{\mathfrak{Sets}} 
\newcommand{\Rings}{\mathfrak{Rings}} 
\newcommand{\Ab}{\mathfrak{Ab}} 
\newcommand{\Grp}{\mathfrak{Grp}} 
\newcommand{\fMod}{\mathfrak{Mod}} 
\newcommand{\QCoh}{\mathfrak{QCoh}}

\newcommand\veps{\varepsilon}

 \newcommand\vphi{\varphi}

\newcommand\si{\sigma}

 \newcommand\Om{\Omega}

\newcommand{\bmu}{\boldsymbol{\mu}}

\newcommand{\und}{\underline{\hspace{2ex}}} 
\DeclareMathOperator*{\colim}{colim} 

\DeclareMathOperator{\Sand}{Sand}
\newcommand{\Strong}{\Omega}
\newcommand{\Weak}{\omega}
\newcommand{\Inn}{\mathrm{Inn}} 

\newcommand{\iso}{\overset{\sim}{\longrightarrow}}

\begin{document}
\title[Quadratic Pairs]{Azumaya Algebras and Obstructions to Quadratic Pairs over a Scheme}
\author[P. Gille]{Philippe Gille}
\address{UMR 5208 du CNRS -
Institut Camille Jordan - Universit\'e Claude Bernard Lyon 1, 43 boulevard du
11 novembre 1918, 69622 Villeurbanne cedex - France }
\email{gille@math.univ-lyon1.fr}

\author[E. Neher]{Erhard Neher}
\address{Department of Mathematics and Statistics, University of Ottawa, 150 Louis-Pasteur Private,
Ottawa, Ontario, Canada, K1N 9A7}
\email{Erhard.Neher@uottawa.ca}

\author[C. Ruether]{Cameron Ruether}
\address{Department of Mathematics and Statistics, Memorial University of Newfoundland, St. John's, NL, Canada, A1C 5S7}
\email{cameronruether@gmail.ca}

\thanks{The first author was supported by the Labex Milyon (ANR-10-LABX-0070) of Universit\'e de Lyon, within the program ``Investissements d'Avenir'' (ANR-11-IDEX- 0007) operated by the French National Research Agency (ANR). The research of the second author was partially supported by an NSERC grant. The research of the third author was partially supported by the NSERC grants of the second author and of Kirill Zainoulline at the University of Ottawa (2022) and partially supported by the NSERC grants of Mikhail Kotchetov and Yorck Sommerh\"auser at the Memorial University of Newfoundland (2023).}

\date{January 27, 2025}

\maketitle

\noindent{\bf Abstract:} We investigate quadratic pairs for Azumaya algebras with involutions over a base scheme $S$
as defined by Calm\`es and Fasel, generalizing the case of quadratic pairs on central simple algebras over a field (Knus, Merkurjev, Rost, Tignol).
We describe a cohomological obstruction for an Azumaya algebra over $S$ with orthogonal involution to admit a quadratic pair and provide a classification of all quadratic pairs it admits. When $S$ is affine this obstruction vanishes, however it is non-trivial in general. In particular, we construct explicit examples with non-trivial obstructions.
\medskip

\bigskip
\section*{Introduction}
{%
\renewcommand{\thesubsection}{\Alph{subsection}}
In this paper we study orthogonal involutions on Azumaya algebras over a arbitrary base scheme, in particular we use characteristic independent methods focusing on quadratic pairs, defined below. In this setting, all simply connected (resp. adjoint) semisimple groups of type $D_n$, $n\geq 5$, are, up to isomorphism, the spin (resp. projective orthogonal group) of an Azumaya algebra with a quadratic pair, \cite[8.4.0.62, 8.4.0.63]{CF}.
\sm

Our work includes the commonly excluded case of fields of characteristic $2$, and more: it also includes the cases when $2$ is neither invertible nor $0$. To illustrate some complications which this causes, let us consider the special case $\cA = \End_R(M)$, where $M$ is a locally free $R$--module of constant even rank,  equipped with a regular quadratic form $q$, and let $\si$ be the associated adjoint involution, see \ref{bcbf} for explanation. The Dickson homomorphism associates with $(M,q)$ a quadratic \'etale $R$--algebra. This assignment is highly non-injective.
Hence, one  cannot expect the theory of regular quadratic forms to be any simpler than that of quadratic \'etale $R$--algebras. But already the theory of quadratic \'etale algebras over arbitrary $R$ is quite involved. They are classified by the \'etale cohomology set $H^1\et(R,\ZZ/2\ZZ)$. If $2$ is invertible, one can use Kummer theory to describe this group, while if $2=0$ one needs Artin-Schreier theory to describe it. Both cases are subsumed by Waterhouse's intricate theory for general $R$, developed in \cite{Waterhouse}. Another instance covered by our setting is that of quadratic forms over the integers.
 \smallskip

We will use the following conventions as in \cite{CF} which differ from those in \cite{KMRT}. An involution on a central simple algebra is \emph{orthogonal} if after extension to a splitting field it is adjoint to regular symmetric bilinear form, it is called \emph{weakly-symplectic} if the bilinear form is skew-symmetric, and called \emph{symplectic} if the bilinear form is alternating. Note that in characteristic $2$ any involution is simultaneously orthogonal and weakly-symplectic.

Previously, working over fields of arbitrary characteristic, all semisimple groups of type $D_n$ with $n \neq 4$ have been constructed by Tits in \cite{Tits} using a more general notion of quadratic form. His approach was generalized by the authors of The Book of Involutions \cite{KMRT} who introduced the more flexible concept of a quadratic pair.

Let $A$ be a central simple $\FF$--algebra. A \emph{quadratic pair} on $A$ is a pair $(\si,f)$ where $\si$ is an orthogonal involution and $f$ is a linear function $f\colon \Symm(A,\si) \to \FF$, called a \emph{semi-trace}, on the symmetric elements of $A$ (those $a\in A$ satisfying $\si(a)=a$), such that $f(a+\si(a))=\Trd_A(a)$ for all $a\in A$. Here $\Trd_A$ denotes the reduced trace of $A$, a notation we will also use later for Azumaya algebras over a scheme. For example, any $\ell \in A$ satisfying $\ell+\sigma(\ell) = 1$ gives rise to a quadratic pair $(\sigma,f)$ by putting $f=\Trd_A(\ell \und)|_{\Symm(A,\sigma)}$. Conversely, any quadratic  pair $(\sigma,f)$ is of this type by \cite[5.7]{KMRT}. Note that if one works in a context where $2$ is invertible, then such an $f$ exists and is unique, in particular $f=\frac{1}{2}\Trd_A(\und)$.

The notion of quadratic pairs was extended by Calm\`es and Fasel \cite{CF} from the setting of central simple algebras over a field, to the setting of Azumaya algebras over a scheme $S$. Their notion of a quadratic pair, and its resulting properties, is parallel to the concept over fields. In this generalized setting, objects such as algebras, groups, bilinear forms, etc., are sheaves or maps of sheaves on the (big) fppf site $\Sch_S$. Modules and algebras will be with respect to the global sections functor $\cO\colon \Sch_S \to \Ab$ sending $T\in \Sch_S$ to $\Gamma(T,\cO_T)$.

In this work, we address the following natural questions which were not asked in \cite{CF}.
\begin{enumerate}[label={\rm(\roman*)}]
\item \label{question_i} Given an Azumaya $\cO$--algebra with orthogonal involution $(\cA,\sigma)$ over the base scheme $S$, does there exists a quadratic pair $(\sigma,f)$ on $\cA$?
\item \label{question_ii} If there is such a pair involving $\sigma$, how many are there? In particular, is there a classification of such pairs?
\end{enumerate}
The answer to these questions goes as follows. We denote by $\cSymd_{\cA,\sigma}$ the sheaf-theoretic image of $\Id +\sigma \colon \cA \to \cA$ and we denote by $\cAlt_{\cA,\sigma}$ the sheaf-theoretic image of $\Id-\sigma \colon \cA \to \cA$. The map $\Id+\sigma$ then descends to a well defined map $\xi \colon \cA/\cAlt_{\cA,\sigma} \to \cSymd_{\cA,\sigma}$. Lemma \ref{lem_1_in_Symd} shows that a necessary condition for $\sigma$ to be in a quadratic pair is that $1_{\cA} \in \cSymd_{\cA,\sigma}(S)$. We call $\sigma$ a \emph{locally quadratic} involution when this is satisfied. We then give the following classification in Theorem \ref{thm_classification}.\\
\begin{thm}\label{thm_A}
Let $(\cA,\sigma)$ be an Azumaya $\cO$--algebra with locally quadratic involution. Then, there is a bijection
\[
\left\{\begin{array}{c} f \text{ such that } (\sigma,f) \text{ is} \\ \text{a quadratic pair on } \cA \end{array}\right\} \leftrightarrow \xi(S)^{-1}(1_{\cA}).
\]
\end{thm}
This answers question \ref{question_ii}. Question \ref{question_i} we answer with cohomological obstructions. We consider the diagram with exact rows
\[
\begin{tikzcd}[ampersand replacement=\&]
0 \arrow{r} \& \cSkew_{\cA,\sigma} \arrow[hookrightarrow]{r} \arrow{d} \& \cA \arrow{r}{\Id+\sigma} \arrow{d} \& \cSymd_{\cA,\sigma} \arrow[equals]{d} \arrow{r} \& 0 \\
0 \arrow{r} \& \cSkew_{\cA,\sigma}/\cAlt_{\cA,\sigma} \arrow[hookrightarrow]{r} \& \cA/\cAlt_{\cA,\sigma} \arrow{r}{\xi} \& \cSymd_{\cA,\sigma} \arrow{r} \& 0
\end{tikzcd}
\]
where $\cSkew_{\cA,\sigma}$ is the kernel of $\Id+\sigma \colon \cA \to \cA$. Taking a portion of the long exact cohomology sequence we have connecting morphisms into \v{C}ech cohomology
\[
\begin{tikzcd}
\cSymd_{\cA,\sigma}(S) \arrow[equals]{d} \arrow{r}{\delta} & \check{H}^1(S,\cSkew_{\cA,\sigma}) \arrow{d} \\
\cSymd_{\cA,\sigma}(S) \arrow{r}{\delta'} & \check{H}^1(S,\cSkew_{\cA,\sigma}/\cAlt_{\cA,\sigma})
\end{tikzcd}
\]
For a locally quadratic involution we define the \emph{weak obstruction} $\delta'(1_{\cA})=\Weak(\cA,\si)\in H^1(S,\cSkew_{\cA,\si}/\cAlt_{\cA,\si})$, and the \emph{strong obstruction} $\delta(1_{\cA})=\Strong(\cA,\si)\in H^1(S,\cSkew_{\cA,\si})$. We show in Theorem \ref{prop_obstructions}:
\begin{thm}\label{thm_B}
Let $(\cA,\si)$ be an Azumaya $\cO$--algebra with a locally quadratic involution.
\begin{enumerate}[label={\rm (\roman*)}]
\item There exists a linear map $f\co \cA \to \cO$ such that $(\si,f|_{\cSym_{\cA,\si}})$ is a quadratic pair on $\cA$ if and only if $\Strong(\cA,\si)=0$. In this case $f=\Trd_{\cA}(\ell \und)$ for an element $\ell \in \cA(S)$ with $\ell+\si(\ell)=1$.
\item There exists a linear map $f\co \cSym_{\cA,\si} \to \cO$ such that $(\si,f)$ is a quadratic pair on $\cA$ if and only if $\Weak(\cA,\si)=0$.
\end{enumerate}
\end{thm}
This result is independent of any characteristic assumptions, however these obstructions are both trivial if $2$ is invertible. Therefore, the interest of this theory is when one works without that assumption. Furthermore, these obstructions capture global phenomenon, since locally, for instance over an affine scheme, both obstructions are zero. We prove in Section \ref{obs_examples} that these obstructions are non-trivial in general.

In the sequel to this paper, \cite{GNR}, we will use the results of this paper to establish an equivalence between the groupoids of adjoint semisimple groups of type $A_1 \times A_1$ and of type $D_2$ over a scheme $S$, extending \cite[15.7]{KMRT}.

\subsection*{Contents}
In Section \ref{sec_prelim} we review the technicalities of our sheaf point of view. Section \ref{quamo} reviews the notion of bilinear and quadratic forms on $\cO$--modules, in particular introducing the adjoint anti-automorphism associated with a regular bilinear form. Such adjoint anti-automorphisms are actually involutions when the bilinear form is suitably symmetric.

Section \ref{quad_pairs} defines quadratic pairs, gives examples, and describes procedures for constructing quadratic pairs. In the case when $\si$ is adjoint to a regular symmetric bilinear form $b$, Proposition \ref{qpex} connects the existence of a quadratic pair on $(\cEnd_{\cO}(\cM),\si)$ to the existence of a quadratic form whose polar form is $b$. The second half of Section \ref{quad_pairs} contains the proof of Theorem \ref{thm_A}.

Section \ref{tens} is the generalization of \cite[5.18, 5.20]{KMRT}, which deals with quadratic pairs on tensor products of algebras with involution.

Section \ref{obs} introduces the strong and weak obstructions and culminates in the proof of Theorem \ref{thm_B}.

Finally, Section \ref{obs_examples} contains various examples of Azumaya algebras with locally quadratic involutions such that their obstructions are non-trivial. In Section \ref{example_strong} we construct an Azumaya algebra with quadratic pair, and therefore trivial weak obstruction, which we show in Lemma \ref{lem_mukai} has non-trivial strong obstruction. Additionally, in Section \ref{example_weak} we construct a quaternion algebra with orthogonal involution which by Lemma \ref{lem_serre} has non-trivial weak obstruction, and therefore also non-trivial strong obstruction. The two previous examples are constructed over a field of characteristic $2$. In Section \ref{sec_weird_example} we construct an example of a locally quadratic quaternion algebra and show in Lemmas \ref{lem_example_3_nontrivial} and \ref{lem_ex_base_change_trivial} that both obstructions are non-trivial while the obstructions of the base change to $\Spec(\FF_2)$ are both trivial. This demonstrates that considering schemes where $2$ is neither invertible nor zero is an important part of the theory.

The authors would also like to thank the referee, whose careful reading and comments helped us improve this paper.
}

\section{Preliminaries}\label{sec_prelim}
Throughout this paper $S$ is a scheme with structure sheaf $\calO_S$. Following the style of \cite{CF}, we consider ``objects over $S$'' as sheaves on the category of schemes over $S$ equipped with the fppf topology. Below, we explain this viewpoint for the notions that are most important for this paper.
\smallskip

\subsection{The Site $\Sch_S$}\label{sec_ringed_site}
Let $\Sch_S$ be the big fppf site of $S$ as in \cite[Expos\'e IV]{SGA3}. Recall that its objects are schemes with structure morphism $T\to S$ and a cover of some $T\in \Sch_S$ is a family $\{T_i \to T\}_{i\in I}$ of flat morphisms which are locally of finite presentation and are jointly surjective. When considering a cover, we will use the notation $T_{ij}=T_i \times_T T_j$ for $i,j\in I$. When a scheme is affine we will usually denote it by $U$, and so an affine cover will usually be written $\{U_i \to T\}_{i\in I}$ with $U_{ij}=U_i \times_T U_j$ for $i,j\in I$. We warn that $U_{ij}$ may not be affine since $T$ need not be separated.

\begin{remark}
In \cite[Tag 021S]{Stacks}, the Stacks project authors introduce ``a'' big fppf site of $S$, instead of ``the''. They do so because of set theoretic nuances avoided in \cite{SGA3} through the use of universes. We also avoid such difficulties and work with ``the'' big fppf site.
\end{remark}

If $\calF$ is any sheaf on $\Sch_S$ and $T\in \Sch_S$, we denote by $\calF|_T$ the restriction of the sheaf to the site $\Sch_S/T = \Sch_T$. We use the same notation for the restriction of elements, namely if $f \co V \to T$ is a morphism in $\Sch_S$ and $t\in \calF(T)$, then we write $t|_V = \calF(f)(t)$. If we have multiple morphisms between $V$ and $T$, we may write $t|_f$ for $\cF(f)(t)$. This overlap of notation is borrowed from \cite{Stacks}. Since these operations apply only to sheaves or sections respectively, the meaning of restriction will be clear from the context.

\subsection{Cohomologies}
We will make use of the first non-abelian cohomology set for sheaves of groups on $\Sch_S$. However, there are various relevant cohomology theories whose $H^1$ are isomorphic. We clarify this now.

First, recall from the discussion before \cite[III.4.5]{M} that for a sheaf of groups $\cG \co \Sch_S \to \Grp$ and $T\in \Sch_S$, the first \v{C}ech cohomology is
\[
\check{H}^1(T,\calG) = \colim_{\calU=\{T_i \to T\}_{i\in I}} \check{H}^1(\calU,\calG)
\]
where the colimit is over fppf covers of $T$. The \v{C}ech cohomology relative to a cover is
\[
\check{H}^1(\calU,\calG)=\{(g_{ij})_{i,j\in I} \in \prod_{i,j\in I} \cG(T_{ij}) \mid (g_{ij}|_{T_{ijk}})(g_{jk}|_{T_{ijk}})=(g_{ik}|_{T_{ijk}})\}/\sim
\]
where the equivalence relation $\sim$ says that cocycles $(g_{ij})_{i,j\in I}$ and $(g_{ij}')_{i,j\in I}$ are cohomologous if there exists $(h_k)\in \prod_{k\in I}\cG(T_k)$ such that we have $g_{ij}'=(h_i|_{T_{ij}})g_{ij}(h_j|_{T_{ij}})^{-1}$. This is a pointed set with distinguished element $[(1)_{i,j\in I}]$. The following lemma describes the connecting morphism for this cohomology.

\begin{lem}[{\cite[III.4.5]{M}}]\label{app_connecting_image}
Let $1 \to \calG' \to \calG \overset{\pi}{\to} \calG'' \to 1$ be an exact sequence of sheaves of groups on $\Sch_S$. Then, there is an exact sequence of pointed sets
\[
1 \to \cG'(S) \to \cG(S) \to \cG''(S) \xrightarrow{\check{\delta}} \check{H}^1(S,\cG') \to \check{H}^1(S,\cG) \to \check{H}^1(S,\cG'')
\]
where $\check{\delta}(g'') = \big[\big((g_i|_{T_{ij}})^{-1}(g_j|_{T_{ij}})\big)_{i,j\in I}\big]$ for a section $g''\in \cG''(S)$ and any cover $\cU = \{T_i \to S\}_{i\in I}$ and elements $g_i \in \cG(T_i)$ such that $\pi(g_i)=g''|_{T_i}$.
\end{lem}
When this is applied to an exact sequence of abelian sheaves $0\to \cF' \to \cF \to \cF'' \to 0$ we will write $\check{\delta}(f'')=[(f_j|_{T_{ij}}-f_i|_{T_{ij}})_{i,j\in I}]$.

Second, we have the non-abelian cohomology of Giraud. Definition \cite[III.2.4.2]{Gir} interpreted in our setting defines $H^1(S,\cG)$ to be the set of isomorphism classes of $\cG$--torsors for the fppf site on $\Sch_S$. It is also a pointed set with base point given by the class of the trivial torsor. There is similarly a connecting morphism for this cohomology.
\begin{lem}[{\cite[III.3.1.3, III.3.3.1]{Gir}}]\label{lem_torsor_connecting}
Let $1 \to \calG' \to \calG \overset{\pi}{\to} \calG'' \to 1$ be an exact sequence of sheaves of groups on $\Sch_S$. Then, there is an exact sequence of pointed sets
\[
1 \to \cG'(S) \to \cG(S) \to \cG''(S) \xrightarrow{\delta} H^1(S,\cG') \to H^1(S,\cG) \to H^1(S,\cG'')
\]
where $\delta(g'') = [\pi^{-1}(g'')]$ for any $g''\in \cG''(S)$. Here, by $\pi^{-1}(g'')$ we mean the subsheaf of $\cG$ given by
\begin{align*}
\pi^{-1}(g'') \colon \Sch_S &\mapsto \Sets \\
T &\mapsto \pi(T)^{-1}(g''|_T).
\end{align*}
\end{lem}
Since $\pi^{-1}(g'')$ is a coset of the subgroup $\cG'$ in the group $\cG$, there is a natural simply transitive action of $\cG'$ on $\pi^{-1}(g'')$. By surjectivity of $\cG \to \cG''$, there will be a cover over which $\pi^{-1}(g'')$ is non-empty. Therefore, it is in fact a $\cG'$--torsor and so its isomorphism class is in $H^1(S,\cG')$ as claimed. The maps between the cohomology sets are induced from the corresponding map between groups via the contracted product construction of \cite[III.1.3]{Gir}.

Despite not being a long exact cohomology sequence in the true sense since the groups may be non-abelian, we will nonetheless refer to the exact sequences in Lemmas \ref{app_connecting_image} and \ref{lem_torsor_connecting} as long exact sequences.

By \cite[III.3.6.4, III.3.6.5]{Gir} (or by \cite[III.4.6]{M} where the inverse map is constructed) there is an isomorphism of pointed sets $\check{H}^1(S,\cG) \iso H^1(S,\cG)$, and it is easy to see that this isomorphism is compatible with the connecting morphisms described above. In particular the diagram
\[
\begin{tikzcd}
\cG''(S) \arrow[equals]{d} \arrow{r}{\check{\delta}} & \check{H}^1(S,\cG) \arrow{d}{\rotatebox{90}{$\sim$}} \\
\cG''(S)  \arrow{r}{\delta} & H^1(S,\cG)
\end{tikzcd}
\]
commutes. Because of this, we will simply use $\delta$ to denote both connecting morphisms.

Finally, we note that when dealing with an abelian sheaf $\cF \colon \Sch_S \to \Ab$, there is also the derived functor cohomology as in \cite[01FT]{Stacks}, which we denote $H^1\fppf(S,\cF)$. By \cite[III, 2.10]{M}, there is an isomorphism $H^1\fppf(T,\calF)\cong \check{H}^1(T,\calF)$.
\sm

\subsection{$\cO$--Modules}\label{sec_O_modules}
By \cite[Tag 03DU]{Stacks}, the contravariant functor
\begin{align*}
\cO \co \Sch_S &\to \Rings \\
T &\mapsto \cO_T(T)
\end{align*}
where $\Rings$ is the category of commutative rings, is an fppf-sheaf on $\Sch_S$, making $(\Sch_S,\cO)$ a ringed site in the sense of \cite[Tag 03AD]{Stacks}. We call $\cO$ the \emph{structure sheaf}. If $T\to S$ is an open immersion, then $\cO(T)=\cOS(T)$. From \cite[Tag 03CW]{Stacks}, an $\cO$--module is then an fppf-sheaf $\cM \co \Sch_S \to \Ab$ together with a map of sheaves
\[
\cO \times \cM \to \cM
\]
such that for each $T\in \Sch_S$, the map $\cO(T)\times \cM(T) \to \cM(T)$ makes $\cM(T)$ a usual $\cO(T)$--module. Given two $\cO$--modules $\cM$ and $\cN$, their internal homomorphism functor
\begin{align*}
\cHom_{\cO}(\cM,\cN)\co \Sch_S &\to \Ab \\
T &\mapsto \Hom_{\cO|_T}(\cM|_T,\cN|_T)
\end{align*}
is another $\cO$--module by \cite[03EM]{Stacks}.

We refer to \cite[Tags 03DE, 03DL]{Stacks} for definitions of various properties of $\cO$--modules. In particular, we will make use of modules which are locally free and/or of finite type, therefore we present these definitions for the convenience of the reader.

An $\cO$--module $\cM$ is called \emph{locally free} if for all $T\in \Sch_S$, there is a covering $\{T_i \to T\}_{i\in I}$ such that for each $i\in I$, the restriction $\cM|_{T_i}$ is a free $\cO|_{T_i}$--module. Explicitly, $\cM|_{T_i} \cong \bigoplus_{j\in J_i}\cO|_{T_i}$ for some index set $J_i$. If all $J_i$ have the same cardinality then $\cM$ has \emph{constant rank $|J_i|$}. A locally free $\cO$--module of constant rank $1$ is called a \emph{line bundle}. Isomorphism classes of line bundles form a group under tensor product, denoted $\Pic(S)$. Since the category of line bundles is equivalent to the category of $\GG_m$--torsors over $S$ \cite[prop. 2.4.3.1]{CF}, the group $\Pic(S)$ is isomorphic to the group $H^1(S,\GG_m)$. It is also  isomorphic to the group $H^1\Zar(S, \GG_m)$ because $\GG_m$--torsors are locally trivial for the Zariski topology.

An $\cO$--module $\cM$ is called \emph{of finite type} if for all $T\in \Sch_S$, there is a covering $\{T_i \to T\}_{i\in I}$ such that for each $i\in I$, the restriction $\cM|_{T_i}$ is an $\cO|_{T_i}$--module which is generated by finitely many global sections. That is, there exists a surjection of $\cO|_{T_i}$--modules $\cO^{\oplus n_i} \to \cM$ for some $n_i\in \NN$.

An $\cO$--module $\cM$ is called \emph{of finite presentation} if for all $T\in \Sch_S$, there is a covering $\{T_i \to T\}_{i\in I}$ such that for each $i\in I$, the restriction $\cM|_{T_i}$ is an $\cO|_{T_i}$--module which has a finite global presentation. That is, there exists an exact sequence of $\cO|_{T_i}$--modules
\[
\bigoplus_{j\in J_i} \cO|_{T_i} \to \bigoplus_{k\in K_i} \cO|_{T_i} \to \cE|_{T_i} \to 0
\]
for some finite index sets $J_i$ and $K_i$.

We will use the terminology \emph{finite locally free} to mean locally free and of finite type. Since $S \in \Sch_S$ is a final object, \cite[Tag 03DN]{Stacks} applies and it suffices for us to check local conditions, such as the three detailed above or quasi-coherence detailed below, for an fppf-covering of $S$.

We call an $\cO$--submodule $\cN \subset \cM$ a \emph{direct summand} of $\cM$ if there exists another $\cO$--module $\cN'$ such that $\cM=\cN\oplus \cN'$. We say that $\cN$ is \emph{locally a direct summand} if there exists a cover $\{T_i\to S\}_{i\in I}$ such that each $\cN|_{T_i}$ is a direct summand of $\cM|_{T_i}$.

We associate with an endomorphism $\si$ of an $\cO$--module $\cM$ satisfying $\si^2 = \Id_{\calM}$ the following $\cO$--modules:
\begin{align*}
 \cSym_{\calM,\si} &= \Ker(\Id - \si) \quad \text{({\em symmetric elements}),} \\
 \cAlt_{\calM,\si} &= \Ima(\Id - \si) \quad \text{({\em alternating elements}),}\\
\cSkew_{\calM,\si} &= \Ker(\Id + \si) \quad \text{({\em skew-symmetric elements}), }\\
 \cSymd_{\calM,\si} &= \Ima(\Id + \si) \quad \text{({\em symmetrized elements}),}
\end{align*}
where $\Ima(\und)$ is the image fppf-sheaf. Since $\cAlt_{\cM,\sigma} \subseteq \cSkew_{\cA,\sigma}$, the map $\Id+\sigma$ descends to a well defined map $\xi \colon \cM/\cAlt_{\cM,\sigma} \to \cSymd_{\cM,\sigma}$. We then have a large diagram with exact rows and columns which we reference later in Sections \ref{quad_pairs} and \ref{obs}.
\begin{equation}\label{eq_big_diagram}
\begin{tikzcd}[ampersand replacement=\&,row sep=0.15in, column sep=0.2in]
 \& 0 \arrow{d} \& 0 \arrow{d} \& \& \\
0 \arrow{r} \& \cAlt_{\cM,\sigma} \arrow[equals]{r} \arrow{d} \& \cAlt_{\cM,\sigma} \arrow{r} \arrow{d} \& 0 \arrow{d} \& \\
0 \arrow{r} \& \cSkew_{\cM,\sigma} \arrow{d} \arrow{r} \& \cM \arrow{d} \arrow{r}{\Id+\sigma} \& \cSymd_{\cM,\sigma} \arrow[equals]{d} \arrow{r} \& 0 \\
0 \arrow{r} \& \cSkew_{\cM,\sigma}/\cAlt_{\cM,\sigma} \arrow{d} \arrow{r} \& \cM/\cAlt_{\cM,\sigma} \arrow{d} \arrow{r}{\xi} \& \cSymd_{\cM,\sigma} \arrow{d} \arrow{r} \& 0 \\
 \& 0 \& 0 \& 0 \&
\end{tikzcd}
\end{equation}

\subsection{Quasi-coherent Modules}\label{quasi_coh}
From \cite[Tag 03DK]{Stacks}, an $\cO$--module $\cE$ is called \emph{quasi-coherent} if for all $T\in \Sch_S$ there is a covering $\{T_i \to T\}_{i\in I}$ such that for each $i\in I$, $\cE|_{T_i}$ has a global presentation. That is there is an exact sequence of $\cO|_{T_i}$--modules
\[
\bigoplus_{j\in J_i} \cO|_{T_i} \to \bigoplus_{k\in K_i} \cO|_{T_i} \to \cE|_{T_i} \to 0
\]
for some index sets $J_i$ and $K_i$. Thanks to \cite[Tag 03OJ]{Stacks}, in our context the classical notion of a quasi-coherent sheaf on $S$ and the notion of a quasi-coherent $\cO$--module are essentially equivalent. Given an $\cO$--module $\cM$ we denote by $\cM_{\mathrm{small}}$ its restriction to the small Zariski site of $S$, which gives a classical $\cO_S$--module on the open subsets of $S$. If $\cE$ is a quasi-coherent $\cO$--module, then $\cE_{\mathrm{small}}$ is a classical quasi-coherent $\cO_S$--module. Conversely, given a classical quasi-coherent $\cO_S$--module $E$ on the small Zariski site, we define
\begin{align*}
E\fppf \colon \Sch_S &\to \Ab \\
(f\colon T\to S) &\mapsto f^*(E)(T) 
\end{align*}
which is a quasi-coherent $\cO$--module on the big fppf site. By \cite[Tag 03DX]{Stacks} (where they use the notation $E^a$ for $E\fppf$), this provides an equivalence of categories between the category of classical quasi-coherent $\cO_S$--modules and the category of quasi-coherent $\cO$--modules. However, we warn that this equivalence does not extend to the categories of all modules. In particular, there exists $\cO$--modules $\cM$ which are not quasi-coherent but for which $\cM_{\mathrm{small}}$ is quasi-coherent. Of course then $(\cM_{\mathrm{small}})\fppf \not\cong \cM$. Some examples are mentioned below in Example \ref{ex_kernel_not_qc}.

We also consider the dual of a quasi-coherent $\cO$--module $\cE$, defined to be $\cE\ch = \cHom_{\cO}(\cE,\cO)$. The functor $\cE\ch$ is represented by the affine $S$--scheme $\uV(\cE)=\Spec\bigl( \mathrm{Sym}^\bullet(\cE_{\mathrm{small}})\bigr)$ by \cite[I, 4.6.3.1]{SGA3}.
If $\cE_{\mathrm{small}}$ is of finite type, then $\uV(\cE)$ is of finite type as an $S$--scheme, \cite[I, 9.4.11]{EGA-neu}.

If $\cE$ is finite locally free, then so is $\cEnd(\cE) = \cHom_{\cO}(\cE,\cE)\cong \cE\ch \ot_{\calO} \cE$. As in \cite[I, 9.6.2]{EGA-neu}, we view its associated scheme $\uV\big( \cEnd(\cE)\big)$ as an $S$--scheme of unital associative algebras. Moreover, the functor
\begin{align*}
\GL(\cE)\co \Sch_S &\to \Grp \\
T &\mapsto \Aut_{\calO|_T}(\cE|_{T})
\end{align*}
is representable by an open $S$--subscheme of $\uV\big(\cEnd(\cE)\big)$, denoted $\uGL(\cE)$, \cite[I, 9.6.4]{EGA-neu}.
\sm

A convenient property of quasi-coherent $\cO$--modules and kernels of morphisms between such modules, is that their cohomology vanishes on affine schemes.
\begin{lem}\label{lem_qc_affine_cohom}
Let $\cM$ be a quasi-coherent $\cO$--module and $U\in \Sch_S$ an affine scheme. Let $i>0$.
\begin{enumerate}[label={\rm(\roman*)}]
\item\label{lem_qc_affine_cohom_i} We have $H^i\fppf(U,\cM)=0$.
\item\label{lem_qc_affine_cohom_ii} Let $\varphi \colon \cM \to \cM'$ be a morphism between two $\cO$--modules and let $\cK$ be its kernel. If $\cM'$ is also quasi-coherent, then $H^i\fppf(U,\cK)=0$.
\end{enumerate}
\end{lem}
\begin{proof}
\noindent\ref{lem_qc_affine_cohom_i}: Since quasi-coherence is a local property, $\cM|_U$ is also quasi-coherent, and therefore corresponds to a small quasi-coherent module $(\cM|_U)_{\mathrm{small}}$ on $U$. Then, since $\big((\cM|_U)_{\mathrm{small}}\big)\fppf \iso \cM|_U$ we may apply \cite[Tag 03P2]{Stacks} which says that
\[
H^i\fppf(U,\cM|_U) = H^i(U,(\cM|_U)_{\mathrm{small}})
\]
where the right hand side is the usual sheaf cohomology. However, since $U$ is affine, \cite[Tag 01XB]{Stacks} says that the right hand side is zero, and therefore
\[
H^i\fppf(U,\cM) = H^i\fppf(U,\cM|_U) = 0
\]
as claimed.

\noindent\ref{lem_qc_affine_cohom_ii}: Since we may first restrict to the affine scheme $U$, we may assume $S$ is affine. By definition, for a scheme $T\in \Sch_S$, we have that $\cK(T)$ is the kernel of $\varphi(T)\colon \cM(T) \to \cM'(T)$. In particular, $\cK_{\mathrm{small}}$ is the kernel of the related morphism $\varphi_{\mathrm{small}} \colon \cM_{\mathrm{small}} \to \cM'_{\mathrm{small}}$ between small quasi-coherent $\cO_S$--modules. Therefore, as is well-known, for example by \cite[7.19(1)]{GW}, $\cK_{\mathrm{small}}$ is a quasi-coherent $\cO_S$--module also. Going back up to the big fppf site, we obtain a quasi-coherent $\cO$--module $\cK' = (\cK_{\mathrm{small}})\fppf$.

Here, we briefly use the small flat site over $S$, denoted by $\frS\fppf$. It is the site whose underlying category consists of schemes over $S$ whose structure morphism is flat and locally of finite presentation, morphisms are also flat and locally of finite presentation, and where the covers are fppf covers as usual. Note, this subcategory $\frS\fppf \subset \Sch_S$ satisfies the assumption of \cite[III.3.1]{M}. As in \cite{M}, we have a morphism of sites $f\colon \Sch_S \to \frS\fppf$ induced by the identity on $S$, alternatively, induced by the inclusion functor $\frS\fppf \to \Sch_S$ of the underlying categories. One can compute that for a sheaf $\cF$ on $\Sch_S$ the pushforward $f_*(\cF)$ is simply the restriction of the sheaf to objects in $\frS\fppf$.

We claim that there is an isomorphism $f_*(\cK)\iso f_*(\cK')$, i.e., that these sheaves are isomorphic on the small flat site. For a flat morphism $g\colon S' \to S$, as a special case of \cite[(7.18.1)]{GW}, we know that the pullback functor $g^*\colon \QCoh(\cO_S) \to \QCoh(\cO_{S'})$ is exact. Therefore, the small quasi-coherent $\cO_{S'}$--module $g^*(\cK_{\mathrm{small}})$ is the kernel of the morphism $g^*(\cM_{\mathrm{small}}) \to g^*(\cM'_{\mathrm{small}})$. However, the global sections of this morphism are canonically isomorphic to the morphism $\cM(S') \to \cM'(S')$, and therefore
\[
\cK'(S') = g^*(\cK_{\mathrm{small}})(S') \iso \cK(S')
\]
as claimed.

Now we apply \cite[III.3.1]{M}. Denoting cohomology on the small flat site by $H^i_{\mathrm{sfppf}}$, we have
\[
H^i\fppf(S,\cK) \simla H^i_{\mathrm{sfppf}}(S,f_*(\cK)) \simla H^i_{\mathrm{sfppf}}(S,f_*(\cK')) \iso H^i\fppf(S,\cK').
\]
where the middle isomorphism follows from the paragraph above. But now $\cK'$ is a quasi-coherent $\cO$--module and we have assumed $S$ is affine, so all these sets are zero by \ref{lem_qc_affine_cohom_i}. Hence, going back to the original affine $U\in \Sch_S$, we conclude that $H^i\fppf(U,\cK)=0$ as desired.
\end{proof}

\begin{remark}\label{rem_not_quasicoh}
In the classical situation, the inclusion functor $\QCoh(\cO_S) \to \fMod(\cO_S)$ from small quasi-coherent sheaves into $\cO_S$--modules is exact and thus preserves kernels. However, on the fppf site the inclusion functor $\QCoh_{\cO} \to \fMod_{\cO}$ from quasi-coherent $\cO$--modules into all $\cO$--modules is no longer exact. Of course, $\QCoh_{\cO}$ is an abelian category because it is equivalent to $\QCoh(\cO_S)$ by \cite[Tag 03DX]{Stacks}. Given a morphism $\varphi \colon \cM_1 \to \cM_2$ in $\QCoh_{\cO}$, the kernel with respect to the abelian category structure will be $K\fppf$ where $K = \Ker(\varphi_{\mathrm{small}}\colon (\cM_1)_{\mathrm{small}} \to (\cM_2)_{\mathrm{small}})$ is the kernel of the morphism between the classical quasi-coherent modules in $\QCoh(\cO_S)$. However, the kernel of $\varphi$ computed just as a morphism of $\cO$--modules may differ from $K\fppf$. In particular, there are examples where the module kernel of a morphism between two quasi-coherent $\cO$--modules is not quasi-coherent, see Example \ref{ex_kernel_not_qc} below. This motivates our wording in Lemma \ref{lem_qc_affine_cohom}\ref{lem_qc_affine_cohom_ii}.
\end{remark}
\sm
Given a quasi-coherent $\cO$--module $\cM$ and $\si \in \cEnd_{\cO}(\cM)$ of order $2$, we will apply Lemma \ref{lem_qc_affine_cohom} to the $\cO$--modules $\cSym_{\cM,\si}$ and $\cSkew_{\cM,\si}$, which are kernels of a morphism $\cM \to \cM$. They are not quasi-coherent in general, as the following example shows.
\begin{example}\label{ex_kernel_not_qc}
Let $S=\Spec(\ZZ)$ and consider the split quaternion algebra $\Mat_2(\cO)$. We equip this with the orthogonal involution $\sigma$ given by
\[
\sigma\left(\begin{bmatrix} a & b \\ c & d \end{bmatrix}\right) = \begin{bmatrix} d & b \\ c & a \end{bmatrix}.
\]
Then, the submodule $\cSkew_{\Mat_2(\cO),\sigma}$ is the kernel of the morphism $\Id + \sigma \colon \Mat_2(\cO) \to \Mat_2(\cO)$ between quasi-coherent $\cO$--modules, however for $T\in \Sch_S$
\[
\cSkew_{\Mat_2(\cO),\sigma}(T) = \left\{ \begin{bmatrix} a & b \\ c & -a \end{bmatrix} \mid a,b,c \in \cO(T) \text{ and } 2b=2c=0\right\}.
\]
If $\cSkew_{\Mat_2(\cO),\sigma}$ were quasi-coherent, then by first using the equivalence of categories \cite[Tag 0GZU]{Stacks} to restrict to the category of affine schemes over $S$ and then applying \cite[Tag 0GZV $(1)\Leftrightarrow (7)$]{Stacks}, we would have an isomorphism
\[
\cSkew_{\Mat_2(\cO),\sigma}(\Spec(\ZZ))\otimes_\ZZ \FF_2 \iso \cSkew_{\Mat_2(\cO),\sigma}(\Spec(\FF_2))
\]
coming from the morphism $\Spec(\FF_2) \to \Spec(\ZZ)$. However, the left hand side is a $1$--dimensional $\FF_2$--vector space, while the right hand side is a $3$--dimensional $\FF_2$--vector space, so this cannot happen. Hence, $\cSkew_{\Mat_2(\cO),\sigma}$ is not a quasi-coherent $\cO$--module on $\Sch_S$. Using the notation of Remark \ref{rem_not_quasicoh} with $\varphi=\Id + \sigma$, in this case we would have
\[
K\fppf(T) = \left\{ \begin{bmatrix} a & 0 \\ 0 & -a \end{bmatrix} \mid a \in \cO(T)\right\}.
\]
\end{example}

\subsection{$\cO$--Algebras}
An $\cO$--module $\calB \co \Sch_S \to \mathfrak{nc}\text{-}\mathfrak{Rings}$ from $\Sch_S$ to the category of all (not necessarily commutative) rings such that each $\calB(T)$ is a $\cO(T)$--algebra will be called an \emph{$\cO$--algebra}. We call it unital, associative, commutative, etc., if each $\calB(T)$ is so. If $\calB$ is a unital associative $\cO$--algebra which is finite locally free, then the functor of invertible elements
\begin{align*}
\uGL_{1,\calB} \co \Sch_S &\to \Grp \\
T &\mapsto \calB(T)^\times
\end{align*}
is representable by an affine $S$-group scheme \cite[2.4.2.1]{CF}. A section $u \in \uGL_{1,\calB}(S)$ induces an inner automorphism of $\calB$, denoted $\Inn(u)$ which is given on $\calB(T)$ by $b\mapsto u|_T \cdot b \cdot u|_T^{-1}$. If  $\calB$  is a separable $\calO$--algebra which is locally free of finite type, then  $\uGL_{1,\calB}$ is a reductive $S$--group scheme  \cite[3.1.0.50]{CF}. We also use the notation $\uGL_{1,\cO} = \cO\ti = \GG_m$.

\subsection{Azumaya Algebras}
A key object of our interest is an Azumaya $\cO$--algebra. First, we recall that over a commutative ring $R$, an Azumaya $R$--algebra is a central separable $R$--algebra. Equivalently, the sandwich map
\begin{align*}
\Sand \co A\otimes_R A\op &\to \End_R(A)\\
a\otimes b &\mapsto (x\mapsto axb)
\end{align*}
is an isomorphism. For separable and Azumaya algebras over rings (equivalently, over affine schemes) we refer to \cite{Ford} or \cite[III \S 5]{K}. Following \cite[5.1]{Gro-Brau}, we consider an \emph{Azumaya $\cO$--algebra}, or simply an \emph{Azumaya algebra}, to be a finite locally free $\cO$--algebra $\cA$ such that the sandwich morphism
\[
\Sand \co \cA \otimes_{\cO} \cA\op \to \cEnd_{\cO}(\cA)
\]
is an isomorphism. This is equivalent to the definition of \cite[2.5.3.4]{CF}, which asks that $\cA$ be finite locally free and that for all affine schemes $U \in \Sch_S$, we have that $\cA(U)$ is an Azumaya $\cO(U)$--algebra.

By \cite[2.5.3.6]{CF}, the rank, viewed as a locally constant integer valued function, of an Azumaya algebra is always square. Following \cite[2.5.3.7]{CF} we call the square root of an Azumaya algebra's rank the \emph{degree} of the algebra. It is a locally constant integer valued function.

Azumaya algebras of the form $\cEnd_{\cO}(\cM)$, where $\cM$ is a locally free $\cO$--module of finite positive rank, will be called \emph{neutral} algebras.

Given any Azumaya algebra $\cA$, it is locally isomorphic to matrix algebras, cf. \cite[2.5.3.8]{CF}. Furthermore, it has a unique linear map $\Trd_\cA \colon \cA \to \cO$, called the \emph{reduced trace of $\cA$}, which agrees with the usual trace locally wherever $\cA$ is isomorphic to a matrix algebra, see \cite[2.5.3.15]{CF}.

\section{Azumaya Algebras with Involutions} \label{quamo}
\subsection{Basic Concepts of Bilinear Forms} \label{bcbf}
We recall some concepts from \cite[2.6]{CF}. A {\em bilinear form} is a pair $(\calM, b)$ consisting of a finite locally free $\cO$--module $\calM$ and a bilinear morphism $b \co \calM\times \calM \to \cO$ of $\cO$--modules, equivalently, a morphism $\calM \to \calM\ch$ of $\cO$--modules. Given $b$, the associated map $\phi_b \co \calM \to \calM\ch$ is defined over $T\in \Sch_S$ on $m\in \cM(T)$ to be the map which is
\begin{align*}
\phi_b(m) \co \cM(V) &\to \cO(V) \\
m' &\mapsto b(m|_V,m')
\end{align*}
over $V\in \Sch_T$. We summarize this by writing $\phi_b(m)= b(m,\und)$. A bilinear form $(\calM,b)$ is called {\em regular\/} if $\phi_b$ is an isomorphism. If for all $T\in \Sch_S$ and all sections $x,y \in \cM(T)$ we have $b(x,y) = b(y,x)$ we call $b$ {\em symmetric}, if $b(x,y)=-b(y,x)$ we call $b$ \emph{skew-symmetric}, and if $b(x,x)=0$ we call $b$ \emph{alternating}.

We associate with a regular bilinear form $(\calM,b)$ the so-called {\em adjoint anti-automorphism\/} $\eta_b \co \cEnd_{\cO}(\calM)\to \cEnd_{\cO}(\calM)$, defined on $a\in \cEnd_{\cO}(\cM)(T)$ by
\[
 \eta_{b}(a) := \phi_b|_T^{-1} \circ a^\vee \circ \phi_b|_T \co
  \calM|_T  \xrightarrow{  \phi_b|_T }
    \calM\ch|_T\  \xrightarrow{a^\vee} \calM\ch|_T \xrightarrow{\phi_b|_T \me}  \calM|_T
\]
where $a\ch$ is the dual endomorphism of $a$. Thus, on sections $m_1, m_2\in \calM(V)$ for $V\in \Sch_T$ and $a\in \cEnd_{\cO}(\calM)(T)$ we have
\begin{equation} \label{quamo-1}
  b\big( \eta_b(a)(m_1), m_2\big) = b\big(m_1, a(m_2)\big).
\end{equation}
If $\calM$ is locally free of finite positive rank, $\eta_b$ is an involution if and only if there exists $\veps \in \mu_2(S)$ such that $b(m_1, m_2) = \veps|_T \cdot b(m_2, m_1)$ for all $T\in \Sch_S$ and $m_1,m_2\in \cM(T)$. In this case, $\veps$ is unique and we call it the \emph{type} of $\eta_b$.

Let $(\calM, b)$ be a regular bilinear form. Since $\calM \ot_{\cO} \calM\ch \cong \cEnd_{\cO}(\calM)$ we have an $\cO$--module isomorphism
\begin{equation} \label{quamo1}
\vphi_b \co \calM \ot_{\cO}\calM \simlgr \cEnd_{\cO}(\calM),
\end{equation}
defined by $m_1\ot m_2\mapsto  \big(m \mapsto m_1|_V \cdot b(m_2|_V, m)\big)$ for sections $m_1,m_2$ over $T\in \Sch_S$ and $m$ over $V\in \Sch_T$.
If $\eta_b$ is an involution of type $\veps$, its pull-back to $\calM \ot_{\cO} \calM$ is the ``$\veps$-switch'':
\begin{equation}\label{quamo2}
 \eta_b \big( \vphi_b (m_1 \ot m_2)\big) = \veps|_T \cdot \vphi_b(m_2 \ot m_1).
\end{equation}
\lv{
Proof of \eqref{quamo2}:
\begin{align*}
 & b\Big( \eta_b\big(\vphi_b(\m_1 \ot m_2)\big) (m_3), \, m_4\Big)
   = b\big(m_3, \vphi_b(m_1\ot m_2)(m_4)\big)
\\ & = b\big(m_3, \, m_1 b(m_2, m_4)\big) = b(m_3, m_1)\, b(m_2, m_4)
    = \veps b(m_1, m_3) \, b(m_2, m_4)
\\ & = \veps \, b\big( m_2, b(m_1, m_3), \, m_4\big)
       = b\big(\veps \vphi_b(m_2 \ot m_1)(m_3), \, m_4)
\end{align*}
}

A {\em quadratic form\/} is a pair $(\calM,q)$ where $\calM$ is a finite locally free $\cO$--module, $q \co \calM \to \cO$ is quadratic with respect to
scalar multiplication, 
and the associated polar form $b_q \co \calM \times \calM \to \cO$, defined on sections by $b_q(m_1,m_2) = q(m_1+m_2) - q(m_1)-q(m_2)$, is bilinear. We  say that $(\calM, q)$ is {\em regular\/} if $(\calM, b_q)$ is so.

Let $b\colon \cM\times\cM \to \cO$ be a bilinear form. For a submodule $\cN\subset \cM$, its \emph{orthogonal compliment} is another $\cO$--submodule of $\cM$, denoted $\cN^\perp \co \Sch_S \to \Ab$, given over $T\in \Sch_S$ by
\[
\cN^\perp(T)= \{a \in \cM(T) \mid b(a|_V,m)=0, \;\forall\, V\in \Sch_T, \;\forall\, m\in \cM(V)\}.
\]
\begin{lem}\label{lem_orth_comp}
Let $\cM$ be a finite locally free $\cO$--module and let $b\colon \cM \times\cM \to \cO$ be a regular bilinear form. Let $\cN\subseteq \cM$ be an $\cO$--submodule which is locally a direct summand of $\cM$. Then, the orthogonal compliment $\cN^\perp$ is also locally a direct summand of $\cM$. In particular, it is finite locally free. In addition,
\[
\rank(\cM) = \rank(\cN) + \rank(\cN^\perp).
\]
\end{lem}
\begin{proof}
This is the global version of \cite[I,\S 2 Prop.~1]{Kn}. We include a proof for the convenience of the reader.

First, choose a cover $\{T_i \to S\}_{i\in I}$ over which $\cM|_{T_i} = \cN|_{T_i} \oplus \cP_i$ for some $\cO|_{T_i}$--modules $\cP_i$. Being direct summands of a finite locally free $\cO|_{T_i}$--module, both $\cN|_{T_i}$ and $\cP_i$ are also finite locally free. Thus, by refining our cover if necessary, we may assume that all modules involved are free. Let $\{n_1,\ldots,n_k\}$ be a free basis of $\cN|_{T_i}$ and let $\{p_1,\ldots,p_\ell\}$ be a free basis of $\cP_i$. Let $\{n_1^*,\ldots,n_k^*,p_1^*,\ldots,p_\ell^*\}$ be a dual basis of $\cM|_{T_i}^*$. Since $b$ is a regular bilinear form, each $p_j^* = b(p_j',\und)$ for some $p_j' \in \cM(T_i)$. It is then clear from the properties of a dual basis that
\[
(\cN|_{T_i})^\perp = \Span_{\cO}(\{p_1',\ldots,p_\ell'\}).
\]
In particular, $(\cN|_{T_i})^\perp = \phi_b^{-1}(\Span_\cO(\{p_1^*,\ldots,p_\ell^*\}))$, and thus
\[
\cM|_{T_i} = \phi_b^{-1}(\Span_\cO(\{n_1^*,\ldots,n_k^*\})) \oplus (\cN|_{T_i})^\perp.
\]
Since $(\cN^\perp)|_{T_i} = (\cN|_{T_i})^\perp$, this shows that $\cN^\perp$ is also locally a direct summand of $\cM$, as claimed.

Regarding the claim about ranks, it is clear from the discussion above that locally we have $\rank((\cM^\perp)|_{T_i})=\rank(\cP_i)$ and therefore
\[
\rank(\cM|_{T_i}) = \rank(\cN|_{T_i}) + \rank((\cN^\perp)|_{T_i}).
\]
Since this holds locally, it holds globally as well. This finishes the proof.
\end{proof}

A particularly useful example of a regular bilinear form on a finite locally free $\cO$--module is the trace form of an Azumaya algebra. If $\cA$ is an Azumaya $\cO$--algebra, then the \emph{trace form}
\begin{align*}
\cA \times \cA &\to \cO \\
(a,b) &\mapsto \Trd_\cA(ab)
\end{align*}
is a regular bilinear form by \cite[2.5.3.16]{CF} (the statement in \cite{CF} is only for even degree algebras, but the proof holds for algebras of any degree).

\subsection{Involutions of Azumaya algebras}\label{orin}
We use the conventions of \cite{CF}, which are slightly different from  those of \cite{KMRT} in the case of base  fields. The  advantage of Calm\`es-Fasel's definitions are that they behave well under arbitrary base change.

Let $\calA$ be an Azumaya $\cO$--algebra. An {\em involution of the first kind\/}
is an anti-isomorphism  $\sigma\colon \calA \to \calA$ of Azumaya algebras
which satisfies $\sigma^2=\Id_{\calA}$. We say that two involutions $\sigma$ and $\sigma'$ are {\em conjugate\/} if there exists $\phi  \in \Aut_{\cO}(\calA)$ such that $\sigma'= \phi \circ \sigma \circ \phi^{-1}$.
In that case we have $\cSym_{\calA, \sigma'}= \phi(\cSym_{\calA, \sigma})$
and similarly for the other $\cO$--modules $\cAlt_{\cA, \si}$, $\cSkew_{\cA, \si}$ and $\cSymd_{\cA, \si}$. The reduced trace of $\cA$ is invariant under any involution by \cite[2.7.0.28]{CF}, i.e., $\Trd_\cA(\sigma(a))=\Trd_\cA(a)$ for any involution $\sigma$ on $\cA$ and all sections $a\in \cA$.
\sm

An Azumaya algebra with involution of the first kind $(\cA,\si)$ is \'etale-locally isomorphic to $\big(\cEnd_{\cO}(\calM), \eta_b\big)$ for some regular bilinear form $(\calM, b)$ \cite[2.7.0.25]{CF}. Thus, since the localization of $\si$ is still an involution, the bilinear form $b$ is $\veps$--symmetric for some $\veps \in \bmu_2(S)$. It is known that $\veps$ only depends on $\si$. Following \cite[2.7.0.26]{CF} (and not \cite{KMRT}), we call $\si$
{\em orthogonal\/} if $\veps = 1$ (so $b$ is symmetric), we call it {\em weakly symplectic\/} if $\veps = -1$ (so $b$ is skew-symmetric), and we call it \emph{symplectic} if $\veps = -1$ and $b$ is alternating. These notions are well-defined, stable under base change, and are local for the fppf topology on $\Sch_S$.
\sm

Let $(\calA,\si)$ be an Azumaya $\cO$--algebra with orthogonal involution. We mention some basic facts about the submodules of $\cA$ related to $\si$.
\begin{lem}\label{locdirsum}
Let $(\cA,\si)$ be an Azumaya $\cO$--algebra with orthogonal involution $\si$.
\begin{enumerate}[label={\rm (\roman*)}]
\item \label{locdirsum_i} If $S$ is affine, then $\cSym_{\calA,\si}$ and $\cAlt_{\calA,\si}$ are direct summands of $\cA$. In particular, $\cSym_{\cA,\si}(S)$ and $\cAlt_{\cA,\si}(S)$ are direct summands of $\cA(S)$ and hence are finite projective $\cO(S)$--modules.
\item \label{locdirsum_ii} $\cSym_{\calA,\si}$ and $\cAlt_{\calA,\si}$ are locally direct summands of $\cA$ and hence are finite locally free $\cO$--modules.
\item \label{locdirsum_iii} If $\calA$ is locally free of constant rank $n^2\in \NN_+$, then $\cSym_{\calA, \si}$ and $\cAlt_{\calA, \si}$ have rank $n(n+1)/2$ and $n(n-1)/2$ respectively.
\end{enumerate}
\end{lem}
\begin{proof}
\noindent \ref{locdirsum_i}: Over an affine scheme we are considering the case of an Azumaya algebra with orthogonal involution $(A,\si)$ over a ring $R$. Being finite projective can be checked after a faithfully flat extension, and by \cite[0, 6.7.5]{EGA-neu} this is also true for the property of  being a direct summand. Therefore, by \cite[2.7.0.25]{CF} we can assume that $(A,\si) \cong (\End_R(M), \eta_b)$ where $M$ is free of finite rank and $b\co M \times M \to R$ is a regular symmetric bilinear form. Thus, as $R$--modules with involution we have $(A,\si) \cong (M^{\ot 2}, \tau)$ where $\tau$ is the switch. Let $\{m_1, \ldots, m_n\}$ be a basis of $M$. Then $m_{ij} = m_i \ot m_j$ is a basis of $M^{\ot 2}$. It follows that
\[
\{ m_{ii}\mid 1\le i \le n\} \cup \{m_{ij} + m_{ji} \mid 1 \le i< j\le n\}
\]
is a basis of $\Symm(M^{\ot 2}, \tau)$, while
\[
\{ m_{ij} - m_{ji} \mid 1 \le i< j\le n\}
\]
is a basis for $\Alt(M^{\ot 2}, \tau)$. Now, observe that $\Symm(M^{\ot 2}, \tau)$ is complemented by $\Span_R(\{m_{ij} \mid 1\le i < j \le n\})$, while $\Alt(M^{\ot 2}, \tau)$ is complemented by
\[
\Span_R(\{ m_{ii} \mid 1\le i \le n\} \cup \{m_{ij} \mid 1 \le i< j\le n\}),
\]
and so both are direct summands of $A$.

\noindent \ref{locdirsum_ii}: This follows from \ref{locdirsum_i} by considering an affine cover of $S$.

\noindent \ref{locdirsum_iii}: If $\cA$ is of constant rank $n^2$, then locally we are in the case of the proof of \ref{locdirsum_i} above where the basis $\{m_1,\ldots,m_n\}$ of $M$ has $n$ elements. Therefore, the basis given for $\Symm(A,\si)$ has $n(n+1)/2$ elements, and the basis given for $\Alt(A,\si)$ has $n(n-1)/2$ elements. Thus the ranks of $\cSym_{\cA,\si}$ and $\cAlt_{\cA,\si}$ are as claimed.
\end{proof}
Some of the results of Lemma \ref{locdirsum} above are proved in \cite[2.7.0.29(1), (2)]{CF} under the assumption that the bilinear form $b$ appearing in the proof is the polar of a regular quadratic form, an assumption that is not needed for the result above.

In particular, since Lemma \ref{locdirsum} above tells us they are finite locally free, $\cSym_{\calA,\si}$ and $\cAlt_{\calA,\si}$ are represented by vector group schemes. The reader will have noticed that the proof of Lemma \ref{locdirsum} works for an involution of any type $\varepsilon\in \bmu_2(S)$. For example, when $\si$ is weakly symplectic and $\varepsilon=-1$, the regular bilinear form $b\co M\times M \to R$ will be skew symmetric and induce an isomorphism $(A,\si)\cong(M^{\otimes 2},-\tau)$ of $R$--modules with involution. In this case, the same roles are played by $\cSkew_{\cA,\si}$ and $\cSymd_{\cA,\si}$ instead.
\begin{cor}\label{locdirsum_symp}
Let $(\cA,\si)$ be an Azumaya $\cO$--algebra with weakly symplectic involution $\si$.
\begin{enumerate}[label={\rm (\roman*)}]
\item \label{locdirsum_symp_i} If $S$ is affine, then $\cSkew_{\calA,\si}$ and $\cSymd_{\calA,\si}$ are direct summands of $\cA$. In particular, $\cSkew_{\cA,\si}(S)$ and $\cSymd_{\cA,\si}(S)$ are direct summands of $\cA(S)$ and hence are finite projective $\cO(S)$--modules.
\item \label{locdirsum_symp_ii} $\cSkew_{\calA,\si}$ and $\cSymd_{\calA,\si}$ are locally direct summands of $\cA$ and hence are finite locally free $\cO$--modules.
\item \label{locdirsum_symp_iii} If $\calA$ is locally free of constant rank $n^2\in \NN_+$, then $\cSkew_{\calA, \si}$ and $\cSymd_{\calA, \si}$ have rank $n(n+1)/2$ and $n(n-1)/2$ respectively.
\end{enumerate}
\end{cor}

These submodule pairs, $\cSym_{\cA,\si}$ and $\cAlt_{\cA,\si}$ when $\si$ is orthogonal, and $\cSkew_{\cA,\si}$ and $\cSymd_{\cA,\si}$ when $\si$ is weakly symplectic, enjoy convenient orthogonality properties with respect to the trace form.

\begin{lem}\label{lem_perp}
Let $\cA$ be an Azumaya $\cO$--algebra. Then the following hold.
\begin{enumerate}[label={\rm (\roman*)}]
\item \label{lem_perp_ii} Let $\si$ be an involution of the first kind on $\cA$. Then,
\begin{align*}
\Trd_{\cA}(\cSym_{\cA,\si}\cdot \cAlt_{\cA,\si})&=0, \text{ and}\\
\Trd_{\cA}(\cSkew_{\cA,\si}\cdot \cSymd_{\cA,\si})&=0,
\end{align*}
by which we mean, for $T\in \Sch_S$ and sections $x\in \cSym_{\cA,\si}(T)$, $y\in \cAlt_{\cA,\si}(T)$ we have $\Trd_{\cA}(xy)=0$, and similarly for the second statement.
\item \label{lem_perp_iii} If $\si$ is an orthogonal involution on $\cA$, then
\[
\cSym_{\cA,\si}^\perp = \cAlt_{\cA,\si} \text{ and } \cAlt_{\cA,\si}^\perp = \cSym_{\cA,\si}.
\]
\item \label{lem_perp_iv} If $\si$ is a weakly symplectic involution on $\cA$, then
\[
\cSymd_{\cA,\si}^\perp = \cSkew_{\cA,\si} \text{ and } \cSkew_{\cA,\si}^\perp = \cSymd_{\cA,\si}.
\]
\end{enumerate}
\end{lem}
\begin{proof}
\noindent \ref{lem_perp_ii}: Let $T\in \Sch_S$ and consider $x\in \cSym_{\cA,\si}(T)$ and $y\in \cAlt_{\cA,\si}(T)$. After passing to a cover of $T$ if needed, we may assume $y=a-\si(a)$ for some $a\in \cA(T)$. Thus, we have
\begin{multline*}
\Trd_{\cA}(x(a-\si(a))=\Trd_{\cA}(xa)-\Trd_{\cA}(x\si(a))=\Trd_{\cA}(xa)-\Trd_{\cA}(a\si(x))\\
=\Trd_{\cA}(xa)-\Trd_{\cA}(ax)=0.
\end{multline*}
The second statement is seen similarly.

\noindent \ref{lem_perp_iii}: The first statement is \cite[2.7.0.29(3)]{CF}, and because of \ref{lem_perp_ii}, a symmetric argument produces the second statement.

\noindent \ref{lem_perp_iv}: In the case when $\si$ is weakly symplectic, $\cSkew_{\cA,\si}$ and $\cSymd_{\cA,\si}$ are locally direct summands of $\cA$ by Corollary \ref{locdirsum_symp}\ref{locdirsum_symp_ii}, which also means that $\cSkew_{\cA,\si}^\perp$ and $\cSymd_{\cA,\si}^\perp$ are finite locally free by Lemma \ref{lem_orth_comp}. Therefore, we may argue as done in \cite{CF} for the case above. From \ref{lem_perp_ii} we have that
\[
\cSymd_{\cA,\si} \subseteq \cSkew_{\cA,\si}^\perp.
\]
Over residue fields this will be an equality for dimension reasons because the dimensions of $\cSkew_{\cA,\si}$ and $\cSymd_{\cA,\si}$ will be the same as the ranks given in Corollary \ref{locdirsum_symp}\ref{locdirsum_symp_iii}. Thus by Nakayama's lemma, we will have equality over local rings, and thus
\[
\cSymd_{\cA,\si} = \cSkew_{\cA,\si}^\perp
\]
globally. A symmetric argument shows the second claim.
\end{proof}

The trace can also be used to define pairings on $\cSymd_{\calA,\sigma}$ and $\cAlt_{\calA,\sigma}$ by extending a construction provided by the exercise \cite[II, exercise 15]{KMRT}.

\begin{lem}\label{lem_pairing}
Let $(\cA,\si)$ be an Azumaya $\cO$--algebra with orthogonal involution.
\begin{enumerate} [label={\rm (\roman*)}]
\item \label{lem_pairing1}
There is a unique symmetric bilinear form
$b_+ \colon \cSymd_{\calA,\sigma} \times \cSymd_{\calA,\sigma} \to \calO$
such that for each affine scheme $U \in \Sch_S$ and
$x,y \in \calA(U)$,
we have
$$
b_+\bigl(  (x+\sigma(x)),  (y+\sigma(y))\bigr)=
\Trd_{\calA}\bigl((x+\sigma(x))y \bigr).
$$

\item \label{lem_pairing2}  There is a unique symmetric bilinear form
$b_- \colon \cAlt_{\calA,\sigma} \times \cAlt_{\calA,\sigma} \to \calO$
such that for each affine scheme $U\in \Sch_S$ and $x,y \in \calA(U)$,
we have
$$
b_-\bigl(  (x-\sigma(x)),  (y-\sigma(y))\bigr)=
\Trd_{\calA}\bigl((x-\sigma(x))y \bigr).
$$
In addition, this bilinear form is regular.
\end{enumerate}
\end{lem}

\begin{proof}
 \ref{lem_pairing1}: We observe first that
 for each affine scheme $U$ over $S$ and
 $x,y \in \calA(U)$,
 we have
 $\Trd_{\calA}((x+\sigma(x))y)= \Trd_{\calA}(x(y+\sigma(y)))$ since $\Trd_{\cA}$ is invariant under $\si$. This implies that $b_{+}$ is well-defined
 and also that it is symmetric.

 \sm

\noindent  \ref{lem_pairing2}: The fact that $b_{-}$ is well-defined and symmetric is similar to \ref{lem_pairing1}. To see that it is regular, we have to show that the induced map
$\cAlt_{\calA,\sigma} \to \cAlt_{\calA,\sigma}^\vee$ is an isomorphism.
Since $\cAlt_{\calA,\sigma}$ is locally free of finite rank,
the Nakayama lemma reduces us to the field case.
For all $x \in \Alt(A, \sigma)$ and $y \in A$, we
we have $b_{-}\bigl(x, (y-\sigma(y))\bigr)= \Trd_A(xy)$.
Since $\Trd_A$ is regular, the kernel of $b_{-}$ is zero.
Thus $b_{-}$ is regular.
\end{proof}

\begin{remark}
The proof of \ref{lem_pairing2} does not apply in case \ref{lem_pairing1} because $\cSymd_{\cA,\si}$ is not locally free of finite rank. However, $b_+$ is regular over fields.
\end{remark}

\section{Quadratic Pairs}\label{quad_pairs}
\begin{defn}[Quadratic pairs]\label{def-qp} Let $(\cA, \si)$ be an Azumaya $\cO$--algebra with orthogonal involution. Following \cite[\S5B]{KMRT} and \cite[2.7.0.30]{CF}, we call $(\si,f)$ a {\em quadratic pair\/} if
$f\colon \cSym_{\calA, \sigma} \to \calO$ is a linear form satisfying
\begin{equation}\label{eq_trace_form}
f(a + \sigma(a))=\Trd_{\calA}(a)
\end{equation}
for all $T\in \Sch_S$ and every $a \in \calA(T)$. The linear form $f$ is called a \emph{semi-trace} and we call $(\cA, \si,f)$ a {\em quadratic triple\/}.

Two quadratic triples $(\calA,\sigma,f)$ and $(\calA',\sigma',f')$ are called {\em isomorphic\/} if there exists an isomorphism $\phi \co \cA \to \cA'$ of $\cO$--algebras such that
$\sigma'= \phi \circ \sigma \circ \phi^{-1}$ and $f'= f \circ \phi^{-1}$ (this is well-defined since $\cSym_{\calA', \sigma'}= \phi(\cSym_{\calA, \sigma})$). The notion of a quadratic triple is stable under arbitrary base change.

\begin{remark}\label{rem_f_non_zero}
If $(\cA,\si,f)$ is a quadratic triple and $U\in \Sch_S$ is an affine scheme, then $f(U)\not\equiv 0$. This follows from \cite[11.1.6]{Ford}, which states that since $\cA(U)$ is an Azumaya $\cO(U)$--algebra, there exists a section $a\in \cA(U)$ with $\Trd_{\cA}(a)=1\in \cO(U)$. Then,
\[
f(a+\si(a))=\Trd_{\cA}(a) \neq 0.
\]
\end{remark}

\begin{examples}\label{quafoex} Let $(\cA, \si)$ be an Azumaya $\cO$--algebra with orthogonal involution.
\begin{enumerate}[wide, labelindent=0pt, label={\rm (\alph*)}]
\item \label{quafoex_a} If $(\si,f)$ is a quadratic pair, then for all $T\in \Sch_S$ and $s\in \cSym_{\cA,\si}(T)$ we have that
\[
2f(s)=f(s+\si(s))=\Trd_{\cA}(s).
\]
Therefore, when $2\in \cO\ti $ there exists a unique $f$ such that $(\si, f)$ is a quadratic pair, namely $f(s) = \frac{1}{2} \Trd_{\cA}(s)$.

\item \label{quafoex_d} Suppose $S=\Spec(R)$ where $R$ is an integral domain whose field of fractions $K$ has characteristic different from $2$. If $(\cA,\sigma)$ is an Azumaya $\cO$--algebra with orthogonal involution, then there exists at most one linear form $f$ making $(\sigma,f)$ a quadratic pair on $\cA$. Indeed, this follows from uniqueness in \ref{quafoex_a} after base change from $S$ to $\Spec(K)$.

\item \label{quafoex_b} Given an element $\ell \in \cA(S)$ with $\ell+\si(\ell)=1$, we define $f\co \cSym_{\cA,\si} \to \cO$ over $T\in \Sch_S$ as
\begin{align*}
f(T) \co \cSym_{\cA,\si}(T) &\to \cO(T) \\
s &\mapsto \Trd_{\cA}(\ell|_T \cdot s).
\end{align*}
This map is clearly linear, and for an element $a+\si(a) \in \cA(T)$ we obtain
\[
f(a+\si(a))=\Trd_{\cA}(\ell|_T \cdot (a+\si(a))) = \Trd_{\cA}((\ell|_T + \si(\ell|_T))\cdot a) = \Trd_{\cA}(a),
\]
and so $(\cA,\si,f)$ is a quadratic triple. By Lemma \ref{lem_perp}\ref{lem_perp_iii}, the function $f$ is determined by $\ell$ up to the addition of an element from $\cAlt_{\cA,\si}(S)$, i.e., $\ell_1 -\ell_2 \in \cAlt_{\cA,\si}(S)$ if and only if $\Trd_{\cA}(\ell_1 \und)$ agrees with $\Trd_{\cA}(\ell_2 \und)$ on $\cSym_{\cA,\si}$.
\end{enumerate}
More examples will be discussed below in Examples \ref{split_example} and \ref{ex_cover_construction}.
\end{examples}

The existence of quadratic pairs is related to quadratic forms. We describe this in the following Proposition~\ref{qpex} for the neutral Azumaya algebra $\cEnd_{\cO}(\calM)$,  generalizing \cite[5.11]{KMRT} for $S=\Spec(R)$ where $R$ is a field. Part \ref{qpex_a} of \ref{qpex} is stated in \cite[2.7.0.31]{CF}. \end{defn}

\begin{prop} \label{qpex} Let $\calM$ be a locally faithfully projective $\cO$--module with a regular symmetric bilinear form $b\co \calM \times \calM \to \cO$. Thus, $\cEnd_{\cO}(\calM)$ is an Azumaya $\cO$--algebra with orthogonal adjoint involution $\eta_b$ by Section {\rm \ref{quamo}}. One knows that%
\[ \vphi_b \co \calM \ot_{\cO}\calM \simlgr \cEnd_{\cO}(\calM),\]
defined on sections by $m_1\ot m_2\mapsto  \big(m \mapsto m_1 \, b(m_2, m)\big)$, is an isomorphism of $\cO$--modules.
\sm
\begin{enumerate}[wide,labelindent=0pt,label={\rm (\roman*)}]
\item \label{qpex_a} Assume $b=b_q$ for a regular quadratic form $q$. Then, there exists a unique linear form $f_q \co \cSym_{\cEnd_{\cO}(\calM), \eta_b} \to \cO$ such that
\begin{equation} \label{qpex_eq_i}
     (f_q \circ \varphi_b )(m \ot m) = q(m)
\end{equation}
holds for all $T\in \Sch_S$ and sections $m \in \calM(T)$. The pair $(\eta_b, f_q)$ is a quadratic pair.

\item \label{qpex_b} Assume $(\eta_b, f)$ is a quadratic pair for some $f$. Then, there exists a unique quadratic form $q_f \co \calM \to \cO$, defined over $T\in \Sch_S$ on sections $m\in \cM(T)$ by
\begin{equation} \label{qpex_eq_ii}
 q_f(m) =  (f \circ \varphi_b )(m \ot m),\end{equation}
whose polar form is $b$.

\item \label{qpex_c} The involution $\eta_b$ is part of a quadratic pair on $\cEnd_{\cO}(\calM)$ if and only if $b$ is the polar form of a regular quadratic form.
\end{enumerate}
\end{prop}

\begin{proof} Since in each statement a global function is given, it suffices to check the claims locally. It is therefore no harm to assume $S= \Spec(R)$ and to prove the lemma in the setting of a faithfully projective $R$--module $M$ with a regular symmetric bilinear form $b\co M \times M \to R$. \sm

\noindent \ref{qpex_a}: We assume that $b=b_q$ for a regular quadratic form $q\co M \to R$ and need to prove that  there exists a unique linear form $f_q \co \Symm(\End_R(M), \eta_b) \to R$ satisfying \eqref{qpex_eq_i}.

The formula  $\eta_b \big( \vphi_b(m_1 \ot m_2)\big) = \vphi_b(m_2 \ot m_1)$ is immediate from the definitions. Hence $\vphi_b$ induces an isomorphism between the symmetric elements of $(\End_R(M), \eta_b)$ and those of $(M^{\ot 2}, \tau)$ where $\tau$ is the ``switch''. We claim
\begin{equation}\label{qpex-3a}  \{ x\in M\ot_R M \mid \tau(x) = x \} = \Span_R \{ m \ot m \mid m \in M\}. \end{equation}
Indeed, it is obvious that the right-hand side is contained in the left-hand side. Thus, to prove equality we may localize in a maximal ideal, equivalently, we may assume that $R$ is a local ring and so $M$ is free. This allows us to consider a free basis $\{m_1,\ldots,m_n\}$ and express $x\in M^{\otimes 2}$ as $x=\sum_{i,j=1}^n c_{ij}m_i\otimes m_j$ for some coefficients $c_{ij}\in R$. If $\tau(x)=x$ then we must have that $c_{ij}=c_{ji}$, and so we may write
\[
x= \sum_{i=1}^n c_{ii}m_i\otimes m_i + \sum_{i,j=1 \atop i<j}^n c_{ij}(m_i\otimes m_j + m_j\otimes m_i).
\]
We see this belongs to the right side of the equality above by noting that
\[
m_i\otimes m_j + m_j\otimes m_i = (m_i+m_j)\otimes(m_i+m_j) - (m_i\otimes m_i) - (m_j\otimes m_j).
\]

We are now left with proving the existence of a linear form
\[ \wtl f \co \Span_R\{ m \ot m \mid m\in M\} \to R, \quad m \ot m \mapsto q(m). \]
If $\wtl f$ exists as a set map, it is linear and unique by \eqref{qpex-3a}. Thus, the point is to prove that $\wtl f$ is well-defined, i.e.,
\begin{equation} \label{qpex-4}
\textstyle \sum_i m_i \ot m_i = 0 \quad \implies \quad
   \sum_i q(m_i) = 0 .
\end{equation}
To show \eqref{qpex-4} we can again argue by localization and therefore assume that $M$ is free. But if $M$ is free, we can use the argument in the proof of \cite[5.11]{KMRT} establishing the existence of $f$ over fields.
\sm

\noindent \ref{qpex_b}: We assume that $(\eta_b,f)$ is a quadratic pair for the Azumaya $R$--algebra $\End_R(M)$ and must show that then $q_f \co M \to R$, defined by \eqref{qpex_eq_ii} 
is a quadratic form whose polar form is $b$.
To do so, we can use the proof of \cite[5.11]{KMRT}, which  carries over without any change.

\noindent \ref{qpex_c}: This summarizes \ref{qpex_a} and \ref{qpex_b}.
\end{proof}

\begin{examples}\label{split_example} We give some further examples. In particular we outline the \emph{split examples} in all degrees.
\begin{enumerate}[wide, labelindent=0pt, label={\rm (\alph*)}]
\item \label{split_example_a} In the setting of Proposition \ref{qpex}, a quadratic pair $(\eta_b, f)$ forces $b=b_q$ for a regular quadratic form $q$. Hence, if $\calM$ has odd rank, then necessarily $2\in \cO\ti$.
\sm

\item \label{split_example_b} Let $(\cM_0,q_0)=(\cO^{2n},q_0)$ where $q_0$ is the split hyperbolic form on $\cO^{2n}$ defined on sections by
\[
q_0(x_1,\ldots,x_{2n})=\sum_{i=1}^n x_{2i-1}x_{2i}.
\]
It is regular. We let $\eta_0$ be the associated orthogonal involution of $\cA_0=\cEnd_{\cO}(\cM_0)=\Mat_{2n}(\cO)$, and denote by $f_0$ the unique linear form on $\cSym_{\cA_0,\eta_0}$ satisfying \eqref{qpex_eq_i}. It is computed explicitly in the example after \cite[2.7.0.31]{CF}. We will refer to $(\cA_0,\eta_0,f_0)$ as the \emph{split example} (in degree $2n$). We point out that there may be many different linear forms $f$ making $(\cA_0,\eta_0,f)$ a quadratic triple.

\item \label{split_example_c} Assume $2\in \cO\ti$ and let $(\cM_0,q_0)=(\cO^{2n+1},q_0)$ where $q_0$ is the split form on $\cO^{2n+1}$ defined on sections by
\[
q_0(x_1,\ldots,x_{2n+1})=\left(\sum_{i=1}^n x_{2i-1}x_{2i}\right) + x_{2n+1}^2
\]
It is regular since $2$ is invertible. We let $\eta_0$ be the associated orthogonal involution of $\cA_0=\cEnd_{\cO}(\cM_0)=\Mat_{2n+1}(\cO)$. It has a unique linear form on $\cSym_{\cA_0,\eta_0}$, namely $f_0=\frac{1}{2}\Trd_{\cA_0}$ by Example \ref{quafoex}\ref{quafoex_a}, such that $(\cA_0,\si_0,f_0)$ is a quadratic triple. We also refer to this as the \emph{split example} (in degree $2n+1$).

\item \label{split_example_d} Other examples are given in \cite[5.12 and 5.13]{KMRT}.
\end{enumerate}
\end{examples}

\begin{cor}{\cite[2.7.0.32]{CF}}\label{cor_split_locally}
Every quadratic triple $(\cA,\si,f)$ is split \'etale-locally.
\end{cor}
\begin{proof}
By first considering the cover by connected components, we may assume that $\cA$ is of constant degree $n$. If $n$ is odd, then by Example \ref{split_example}\ref{split_example_a} above and Example \ref{quafoex}\ref{quafoex_a}, we must have $2\in \cO\ti$ and $f=\frac{1}{2}\Trd_{\cA,\si}$. Then there exists an \'etale cover over which $\si$ and $\eta_0$ are isomorphic, which is sufficient since those isomorphisms will identify $f$ and $f_0$ due to uniqueness.

Now assume $n$ is even. We have already started with \cite[2.7.0.25]{CF}: every quadratic triple is \'etale-locally of the form $(\cEnd_{\cO}(\cM),\eta_b,f)$ considered in Proposition \ref{qpex}. By Proposition \ref{qpex}\ref{qpex_b}, the regular bilinear form of that lemma is the polar form of a regular quadratic form $q\co \cM \to \cO$, given by \eqref{qpex_eq_ii}. By \cite[2.6.1.13]{CF}, after a second \'etale extension, we can then assume that $(\cM,q)=(\cM_0,q_0)$ as in Example \ref{split_example}\ref{split_example_b}, and that the linear form is the one of Proposition \ref{qpex}\ref{qpex_a}, i.e., that the quadratic triple is split.
\end{proof}

\begin{remark}
Corollary \ref{cor_split_locally} shows that, if one wishes, it is sufficient to work with the \'etale topology on the category $\Sch_S$. In this paper, our arguments do not use particular properties of the fppf topology which are not shared with the \'etale topology, and so our results still hold in that setting.
\end{remark}

\begin{cor}\label{cor_f_of_one}
Let $(\cA,\si,f)$ be a quadratic triple. Assume $\cA$ is of constant rank $n^2 \in \NN$. Then, we have
\[
f(1_{\cA}) = \frac{n}{2} \in \cO(S).
\]
\end{cor}
\begin{proof}
Since $\cA$ is of constant rank, there will exist a cover $\{T_i \to S\}_{i\in I}$ such that $\cA|_{T_i}\cong \Mat_n(\cO|_{T_i})$ for all $i\in I$. Evaluating locally, where $1_{\cA}|_{T_i}$ is the identity matrix, we have that $2f|_{T_i}(1_{\cA}|_{T_i})=\Trd_{\Mat_n(\cO|_{T_i})}(1_{\cA}|_{T_i})=n$, and hence
\[
2f(1_{\cA}) = n \in \cO(S)
\]
globally as well. If $n$ is odd then $2\in \cO\ti$ and so $f(1_{\cA})=\frac{n}{2}$ makes sense. If $n=2m$ is even, we need to argue that $f(1_{\cA})=m$. We may assume by Corollary \ref{cor_split_locally} that in the above argument, the cover was chosen such that it splits the quadratic triple, and so our local evaluation may be done on the split example. In the calculations of $f_0$ following \cite[2.7.0.31]{CF} the authors show that $f_0(E_{ii}+E_{i+1,i+1})=1$, and therefore it follows immediately that $f_0(1)=m$.
\lv{
By refining the cover if needed so that it consists of affine schemes, we may further assume by Corollary \ref{lem_pair_affine} below that there are elements $\ell_i \in \cA(T_i)$ such that $f|_{T_i}=\Trd_{\cA|_{T_i}}(\ell_i \und)$ and $\ell_i + \eta_0(\ell_i)=1$. The involution $\eta_0$ acts as follows
{\scriptsize \[
\left[ \begin{array}{c;{2pt/2pt}c;{2pt/2pt}c;{2pt/2pt}c}
\begin{array}{cc} a_1 & b_1 \\ c_1 & d_1 \end{array} & * & \cdots & * \\ \hdashline[2pt/2pt]
* & \begin{array}{cc} a_2 & b_2 \\ c_2 & d_2 \end{array} & \cdots & * \\ \hdashline[2pt/2pt]
\vdots & \vdots & & \vdots \\ \hdashline[2pt/2pt]
* & * & \cdots & \begin{array}{cc} a_m & b_m \\ c_m & d_m \end{array}
\end{array}\right] \overset{\eta_0}{\mapsto}
\left[ \begin{array}{c;{2pt/2pt}c;{2pt/2pt}c;{2pt/2pt}c}
\begin{array}{cc} d_1 & b_1 \\ c_1 & a_1 \end{array} & * & \cdots & * \\ \hdashline[2pt/2pt]
* & \begin{array}{cc} d_2 & b_2 \\ c_2 & a_2 \end{array} & \cdots & * \\ \hdashline[2pt/2pt]
\vdots & \vdots & & \vdots \\ \hdashline[2pt/2pt]
* & * & \cdots & \begin{array}{cc} d_m & b_m \\ c_m & a_m \end{array}
\end{array}\right]
\]}%
Therefore, after writing any element $\ell_i$ in this form, it must satisfy $a_j+d_j = 1$ for all $j=0,\ldots,m$. Hence,
\[
f|_{T_i}(1_{\cA}|_{T_i})=\Trd_{\Mat_n(\cO|_{T_i})}(\ell_i)=\sum_{j=1}^m a_j+d_j = m
\]
as desired.%
}
\end{proof}

\begin{example}\label{no_f_example}
In contrast to Example \ref{quafoex}\ref{quafoex_a}, if $2\notin \cO\ti$, then not every orthogonal $\si$ is part of a quadratic pair. For example, this is so for $S=\Spec(R)$ with $2R=0 \ne R$, and the Azumaya $\cO$--algebra $\cA=\Mat_n(\cO)$ equipped with $\tau = $ the transpose involution. Indeed, we would need that
\[
f(E_{11} + \tau(E_{11})) = \Trd_{\Mat_n(\cO)}(E_{11}) = 1,
\]
however $E_{11} + \tau(E_{11}) = 2E_{11} = 0$ and so this cannot happen. Another example where no $f$ exists is given in Example \ref{ex_weakly_tens_not_quad}.
\end{example}

\subsection{Classification of Quadratic Triples}
In light of Example \ref{no_f_example} above, we seek to characterize which orthogonal involutions can participate in a quadratic pair. Furthermore, for a given Azumaya $\cO$--algebra with orthogonal involution $(\cA,\sigma)$, we wish to classify all possible such $f$ which make $(\cA,\sigma,f)$ a quadratic triple. We first provide this classification over affine schemes, beginning with the following.

Given a quadratic triple $(\cA,\si,f)$, it is of the form of Example \ref{quafoex}\ref{quafoex_b} above exactly when $f$ can be extended to a linear form on all of $\cA$.
\begin{lem}\label{lem_extend_l}
Let $(\cA,\si,f)$ be a quadratic triple. Then, the following are equivalent.
\begin{enumerate}[label={\rm (\roman*)}]
\item \label{lem_extend_l_i} There exists a linear form $f'\co \cA \to \cO$ such that $f'|_{\cSym_{\cA,\si}} = f$,
\item \label{lem_extend_l_ii} There exists $\ell \in \cA(S)$ with $\ell+\si(\ell)=1$ such that for all $T\in \Sch_S$ and $s\in \cSym_{\cA,\si}(T)$, we have
\[
f(s)=\Trd_{\cA}(\ell|_T \cdot s).
\]
\end{enumerate}
\end{lem}
\begin{proof}
\noindent \ref{lem_extend_l_ii}$\Rightarrow$\ref{lem_extend_l_i}: This implication is clear, see Example \ref{quafoex}\ref{quafoex_b}.

\noindent \ref{lem_extend_l_i}$\Rightarrow$\ref{lem_extend_l_ii}: We follow the second half of the proof of \cite[4.2.0.12]{CF}. The map $f'$ is of the form $f'=\Trd_{\cA}(\ell\und)$ for some $\ell \in \cA(S)$ due to regularity of the trace form by Lemma \ref{lem_perp}\ref{lem_perp_i}. Then, for any $a\in \cA(S)$ we have $\Trd_{\cA}(\ell \si(a))=\Trd_{\cA}(a\si(\ell))=\Trd_{\cA}(\si(\ell)a)$ and so
\[
\Trd_{\cA}(a)=f'(a+\si(a))=\Trd_{\cA}(\ell(a+\si(a)))=\Trd_{\cA}((\ell+\si(\ell))a).
\]
Hence, $\ell+\si(\ell)=1$, once again by regularity of the trace form.
\end{proof}

When $S$ is affine, the construction of Example \ref{quafoex}\ref{quafoex_b} produces all quadratic triples and we have the following classification.
\begin{prop}[Affine Classification]\label{lem_pair_affine}
Let $S$ be an affine scheme and let $(\cA,\sigma)$ be an Azumaya $\cO$--algebra with orthogonal involution. Consider $\xi \colon \cA/\cAlt_{\cA,\sigma} \to \cSymd_{\cA,\sigma}$ from \eqref{eq_big_diagram}.
\begin{enumerate}[label={\rm(\roman*)}]
\item \label{lem_pair_affine_i} Suppose $(\cA,\sigma,f)$ is a quadratic triple. Then there exists $\ell_f \in \cA(S)$ such that $\ell_f+\si(\ell_f)=1_{\cA}$, and for all $T\in \Sch_S$ and $s\in \cSym_{\cA,\si}(T)$, we have
\[
f(s)=\Trd_{\cA}(\ell_f|_T \cdot s).
\]
In particular, $\ell_f+\si(\ell_f)=1_{\cA}$ means that $1_{\cA} \in \cSymd_{\cA,\sigma}(S)$.
\item \label{lem_pair_affine_ii} Assume that $1_{\cA} \in \cSymd_{\cA,\sigma}(S)$. Then, $\xi(S)^{-1}(1_{\cA})$ is a non-empty subset of $(\cA/\cAlt_{\cA,\sigma})(S)$ and the map of sets
\begin{align*}
\left\{\begin{array}{c} f \text{ such that } (\cA,\sigma,f) \\ \text{is a quadratic triple}\end{array}\right\} &\to \xi(S)^{-1}(1_{\cA}) \\
f &\mapsto \overline{\ell_f}
\end{align*}
is a bijection. In particular, for a given $f$ the element $\ell_f$ is uniquely determined up to addition by an element of $\cAlt_{\cA,\si}(S)$.
\item \label{lem_pair_affine_iii} The condition $1_{\cA} \in \cSymd_{\cA,\sigma}(S)$ is necessary and sufficient for there to be a quadratic triple extending $(\cA,\sigma)$.
\end{enumerate}
\end{prop}
\begin{proof}
\noindent \ref{lem_pair_affine_i}: Since $S$ is affine, by Lemma \ref{locdirsum} we know $\cSym_{\cA,\si}$ is a direct summand of $\cA$. Therefore, for any quadratic triple $(\cA,\si,f)$, the function $f$ may be extended to $\cA$. The result then follows from Lemma \ref{lem_extend_l}.

\noindent \ref{lem_pair_affine_ii}: First we show $\xi(S)^{-1}(1_{\cA})$ is non-empty. We have a short exact sequence
\[
0 \to \cSkew_{\cA,\sigma} \to \cA \xrightarrow{\Id+\sigma} \cSymd_{\cA,\sigma} \to 0
\]
with associated long exact sequence
\[
0 \to \cSkew_{\cA,\sigma}(S) \to \cA(S) \to \cSymd_{\cA,\sigma}(S) \to H^1\fppf(S,\cSkew_{\cA,\sigma}) \to \ldots
\]
Since $S$ is affine and $\cSkew_{\cA,\sigma}$ is the kernel of a morphism of quasi-coherent $\cO$--modules, applying Lemma \ref{lem_qc_affine_cohom} gives that $H^1\fppf(S,\cSkew_{\cA,\sigma})=0$ and so $\cA(S) \to \cSymd_{\cA,\sigma}(S)$ is surjective. Hence $\xi(S)^{-1}(1_{\cA}) \neq \O$.

Next we show that the map is bijective as claimed. Lemma \ref{lem_qc_affine_cohom} similarly shows that $H^1\fppf(S,\cAlt_{\cA,\sigma})=0$ and so there is a surjection $\cA(S) \to (\cA/\cAlt_{\cA,\sigma})(S) \cong \cA(S)/\cAlt_{\cA,\sigma}(S)$.

Let $(\cA,\sigma,f_1)$ and $(\cA,\sigma,f_2)$ be two quadratic triples. By \ref{lem_pair_affine_i} we have elements $\ell_{f_1},\ell_{f_2} \in \cA(S)$ defining the linear forms. If $\overline{\ell_{f_1}}=\overline{\ell_{f_2}}$ then $\ell_{f_1}-\ell_{f_2} \in \cAlt_{\cA,\sigma}(S)$, and so by Lemma \ref{lem_perp}\ref{lem_perp_iii} the maps $\Trd_{\cA}(\ell_{f_1}\und)$ and $\Trd_{\cA}(\ell_{f_2}\und)$ agree on $\cSym_{\cA,\sigma}$, i.e., $f_1=f_2$ and so the map in question is injective.

Since $\cA(S) \to (\cA/\cAlt_{\cA,\sigma})(S)$ is surjective and by the commutativity of \eqref{eq_big_diagram}, every element of $\xi(S)^{-1}(1_{\cA})$ is of the form $\overline{\ell}$ for some $\ell\in \cA(S)$ with $\ell+\sigma(\ell)=1$. We may then use Example \ref{quafoex}\ref{quafoex_b} to see that $f=\Trd_{\cA}(\ell\und)|_{\cSym_{\cA,\sigma}}$ extends $(\cA,\sigma)$ to a quadratic triple. The process in \ref{lem_pair_affine_i} then produces a possibly different $\ell_f \in \cA(S)$ such that $f=\Trd_{\cA}(\ell_f \und)|_{\cSym_{\cA,\sigma}}$ also. Since this is the same function on $\cSym_{\cA,\sigma}$, we use Lemma \ref{lem_perp}\ref{lem_perp_iii} again to see that $\ell-\ell_f \in \cAlt_{\cA,\sigma}(S)$ and so $\overline{\ell}=\overline{\ell_f}$. Therefore the map in question is surjective and we are done.

\noindent \ref{lem_pair_affine_iii}: This follows from \ref{lem_pair_affine_i} and \ref{lem_pair_affine_ii}.
\end{proof}

\begin{remark}\label{rem_affine_classification}
When $S=\Spec(R)$ is an affine scheme we are dealing with an Azumaya $R$--algebra $A=\cA(S)$ and an orthogonal involution $\sigma(S)$ on $A$, which by abuse of notation we also denote $\sigma$. Then, $\Symm(A,\sigma)=\cSym_{\cA,\sigma}(S)$ and a quadratic triple is given by an $R$--linear map $f\colon \Symm(A,\sigma) \to R$ satisfying $f(a+\sigma(a))=\Trd_A(a)$ for all $a\in A$. By Lemma \ref{lem_qc_affine_cohom}\ref{lem_qc_affine_cohom_ii} implying that $H^1(S,\cSym_{\cA,\sigma})=0$ and $H^1(S,\cSkew_{\cA,\sigma})=0$, we know that $\cA(S)\to \cAlt_{\cA,\sigma}(S)$ and $\cA(S)\to \cSymd_{\cA,\sigma}(S)$ are surjective, thus
\begin{align*}
\Alt(A,\sigma) = \cAlt_{\cA,\sigma}(S) &= \{a-\sigma(a) \mid a\in A\}, \\
\Symd(A,\sigma) =\cSymd_{\cA,\sigma}(S) &= \{a+\sigma(a) \mid a\in A\}.
\end{align*}
Furthermore, since $\sigma$ is orthogonal, $\cAlt_{\cA,\sigma}$ is finite locally free by Lemma \ref{locdirsum}\ref{locdirsum_ii} and hence is quasi-coherent, meaning $H^1(S,\cAlt_{\cA,\sigma})=0$ as well. This means that $\cA(S) \to (\cA/\cAlt_{\cA,\sigma})(S)$ is surjective and so $(\cA/\cAlt_{\cA,\sigma})(S)= A/\Alt(A,\sigma)$. The map $\xi$ of \eqref{eq_big_diagram} behaves over $S$ as
\begin{align*}
\xi \colon A/\Alt(A,\sigma) &\to \Symd(A,\sigma) \\
\overline{a} &\mapsto a+\sigma(a)
\end{align*}
which is a well-defined linear map. Then Proposition \ref{lem_pair_affine} says that any quadratic triple $(A,\sigma,f)$ is of the form $f=\Trd_A(\ell\und)$ for some $\ell\in A$ with $\ell+\sigma(\ell)=1$. Furthermore, having $1\in \Symd(A,\sigma)$ is sufficient to extend $(A,\sigma)$ to a quadratic triple since the preimage $\xi^{-1}(1)$ is non-empty. In this case, all such linear forms extending $(A,\sigma)$ are classified by $\xi^{-1}(1) \subseteq A/\Alt(A,\sigma)$. Over fields this recovers the classification (for fixed $\sigma$) given in \cite[\S 5]{KMRT}.
\end{remark}
By considering an affine open cover, Proposition \ref{lem_pair_affine}\ref{lem_pair_affine_i} recovers \cite[4.2.0.12]{CF}, which states that Zariski locally a quadratic triple is of the form in Example \ref{quafoex}\ref{quafoex_b}.
\sm

Given an Azumaya $\cO$--algebra with orthogonal involution $(\cA,\sigma)$, Proposition \ref{lem_pair_affine}\ref{lem_pair_affine_iii} states that when $S$ is affine, the condition $1_{\cA} \in \cSymd_{\cA,\sigma}(S)$ is necessary and sufficient for there to be a quadratic triple extending $(\cA,\sigma)$. We show now that this condition is necessary in general, however Lemma \ref{lem_serre} shows that it is no longer sufficient.

\begin{lem}\label{lem_1_in_Symd}
Let $(\cA,\si,f)$ be a quadratic triple. Then, $1_{\cA} \in \cSymd_{\cA,\si}(S)$.
\end{lem}
\begin{proof}
We know by Proposition \ref{lem_pair_affine}\ref{lem_pair_affine_i} that for an affine cover $\{U_i\to S\}_{i\in I}$ of $S$ we have that
\[
f|_{U_i} = \Trd_{\cA}(\ell_i \und)
\]
for some $\ell_i \in \cA(U_i)$ with $\ell_i + \si(\ell_i)=1_{\cA(U_i)}$. In particular, $1_{\cA(U_i)} \in \cSymd_{\cA,\si}(U_i)$. Since the elements $1_{\cA(U_i)}$ clearly agree on overlaps of the cover $\{U_i\to S\}_{i\in I}$ and $\cSymd_{\cA,\si}$ is a sheaf, we obtain that $1_{\cA} \in \cSymd_{\cA,\si}(S)$.
\end{proof}

\begin{lem}\label{lem_smart_equiv}
Let $(\cA,\si)$ be an Azumaya $\cO$--algebra with orthogonal involution. Then the following are equivalent.
\begin{enumerate}[label={\rm (\roman*)}]
\item \label{lem_smart_equiv_i} $1_{\cA} \in \cSymd_{\cA,\si}(S)$
\item \label{lem_smart_equiv_ii} For all affine $U\in \Sch_S$, there exists $f_U\co \cSym_{\cA|_U,\si|_U} \to \cO|_U$ such that $(\cA|_U,\si|_U,f_U)$ is a quadratic triple.
\item \label{lem_smart_equiv_iii} For all affine open subschemes $U\subseteq S$, there exists $f_U\co \cSym_{\cA|_U,\si|_U} \to \cO|_U$ such that $(\cA|_U,\si|_U,f_U)$ is a quadratic triple.
\item \label{lem_smart_equiv_iv} There exists an fppf cover $\{T_i\to S\}_{i\in I}$ such that for each $i\in I$, there exists $f_i \co \cSym_{\cA|_{T_i},\si|_{T_i}}\to \cO|_{T_i}$ such that $(\cA|_{T_i},\si|_{T_i}, f_i)$ is a quadratic triple.
\item \label{lem_smart_equiv_v} There exists an fppf cover $\{T_i \to S\}_{i\in I}$ and regular quadratic forms $(\cM_i,q_i)$ over $\cO|_{T_i}$ such that for each $i \in I$,
\[
(\cA|_{T_i},\si|_{T_i})\cong (\cEnd_{\cO|_{T_i}}(\cM_i),\eta_q)
\]
where $\eta_q$ is the adjoint involution of the polar bilinear form of $q$.
\end{enumerate}
\end{lem}
\begin{proof}
We argue that \ref{lem_smart_equiv_i}$\Rightarrow$\ref{lem_smart_equiv_ii}$\Rightarrow$\ref{lem_smart_equiv_iii}$\Rightarrow$\ref{lem_smart_equiv_iv}$\Rightarrow$\ref{lem_smart_equiv_i}, and then that \ref{lem_smart_equiv_iv}$\Leftrightarrow$\ref{lem_smart_equiv_v}.

\noindent \ref{lem_smart_equiv_i}$\Rightarrow$\ref{lem_smart_equiv_ii}: Assume $1_{\cA} \in \cSymd_{\cA,\si}(S)$. Then, for an affine $U \in \Sch_S$, by restriction we have that $1_{\cA(U)} \in \cSymd_{\cA,\si}(U)$. Hence Proposition \ref{lem_pair_affine}\ref{lem_pair_affine_ii} guarantees that there is at least one semi-trace $f_U\co \cSym_{\cA|_U,\si|_U} \to \cO|_U$ such that $(\cA|_U,\si|_U,f_U)$ is a quadratic triple.

\noindent \ref{lem_smart_equiv_ii}$\Rightarrow$\ref{lem_smart_equiv_iii} is obvious, and \ref{lem_smart_equiv_iii}$\Rightarrow$\ref{lem_smart_equiv_iv} by taking an affine open cover of $S$.

\noindent \ref{lem_smart_equiv_iv}$\Rightarrow$\ref{lem_smart_equiv_i}: Let such a cover be given. By Lemma \ref{lem_1_in_Symd} we have $1_{\cA(T_i)}\in \cSymd_{\cA,\si}(T_i)$ for all $i\in I$. These agree on overlaps and so $1_{\cA} \in \cSymd_{\cA,\si}(S)$ by the sheaf property.

\noindent \ref{lem_smart_equiv_iv}$\Rightarrow$\ref{lem_smart_equiv_v}: Since $(\cA,\si)$ is locally isomorphic to a neutral algebra, we may assume by refining the given cover that $(\cA|_{T_i},\si|_{T_i}, f_i) = (\cEnd_{\cO}(\cM_i),\si|_{T_i},f_i)$ for some $\cO|_{T_i}$--modules $\cM_i$. There are then quadratic forms $q_i\co \cM_i \to \cO|_{T_i}$ whose polar bilinear form has adjoint involution $\si|_{T_i}$ by Proposition \ref{qpex}\ref{qpex_b}.

\noindent \ref{lem_smart_equiv_v}$\Rightarrow$\ref{lem_smart_equiv_iv}: This is Proposition \ref{qpex}\ref{qpex_a}.
\end{proof}

\begin{defn}[Locally Quadratic Involution]\label{defn_locally_quadratic}
We call an orthogonal involution on an Azumaya $\cO$--algebra \emph{locally quadratic} if it satisfies the equivalent conditions of Lemma \ref{lem_smart_equiv}.
\end{defn}

We now give two more examples of constructions which produce quadratic triples.
\begin{examples}\label{ex_cover_construction}
Let $(\cA,\sigma)$ be an Azumaya $\cO$--algebra with orthogonal involution.
\begin{enumerate}[wide, labelindent=0pt, label={\rm (\alph*)}]
\item \label{quafoex_c} Given an fppf cover $\{T_i \to S\}_{i\in I}$ and elements $\ell_i \in \cA(T_i)$ satisfying $\ell_i + \si(\ell_i) = 1$ and $\ell_i|_{T_{ij}} - \ell_j|_{T_{ij}} \in \cAlt_{\cA, \si}(T_{ij})$ for all $i,j\in I$, we can construct a unique quadratic triple $(\cA,\si,f)$ such that $f|_{T_i} = \Trd_{\cA}(\ell_i\und)$ holds for all $i\in I$. To do so, locally define maps $f_i \co \cSym_{\cA, \si}|_{T_i} \to \cO|_{T_i}$ by restricting the linear forms $\Trd_{\cA}(\ell_i\und)$ to $\cSym_{\cA, \si}|_{T_i}$. By assumption and Lemma \ref{lem_perp}\ref{lem_perp_ii}, these local maps agree on overlaps of the cover $\{T_i \to S\}_{i\in I}$ and therefore glue together to define a global linear form $f\co \cSym_{\cA,\si}\to\cO$. The fact that $(\cA, \si, f)$ is a quadratic triple can be checked locally on the $T_i$, where it follows from Example~\ref{quafoex}\ref{quafoex_b}.
\item \label{ex_cover_construction_b} Let $\sigma$ be a locally quadratic involution. Let $\xi \colon \cA/\cAlt_{\cA,\sigma} \to \cSymd_{\cA,\sigma}$ be the map of \eqref{eq_big_diagram}. Given an element $\lambda \in \xi(S)^{-1}(1_{\cA}) \subseteq (\cA/\cAlt_{\cA,\sigma})(S)$, we define a quadratic triple $(\cA,\sigma,f_{\lambda})$ as follows. There will exist a cover $\{T_i \to S\}_{i\in I}$ and elements $\ell_i \in \cA(T_i)$ such that $\overline{\ell_i} = \lambda|_{T_i}$. By the commutativity of \eqref{eq_big_diagram} and the fact that $\xi(\lambda|_{T_i})=1$, each $\ell_i$ will satisfy $\ell_i + \sigma(\ell_i)=1$. Since $(\lambda|_{T_i})|_{T_{ij}} = (\lambda|_{T_j})|_{T_{ij}}$, we will have that $\ell_i|_{T_{ij}}-\ell_j|_{T_{ij}} \in \cAlt_{\cA,\sigma}(T_{ij})$ and hence \ref{quafoex_c} above applies to construct $f_{\lambda}$. Since we can always compare over a common refinement, it is clear that $f_\lambda$ is independent of which cover was chosen.
\end{enumerate}
\end{examples}

We are now able to show that all quadratic triples arise from the construction of Examples \ref{ex_cover_construction}\ref{ex_cover_construction_b}. This gives a classification of all quadratic triples which extend a given locally quadratic involution. In contrast to Proposition \ref{lem_pair_affine}\ref{lem_pair_affine_ii}, in the following Theorem the set $\xi(S)^{-1}(1_{\cA})$ may be empty, see Lemma \ref{lem_serre}.
\begin{thm}[Classification]\label{thm_classification}
Let $(\cA,\sigma)$ be an Azumaya $\cO$--algebra with locally quadratic involution. Let $\xi \colon \cA/\cAlt_{\cA,\sigma} \to \cSymd_{\cA,\sigma}$ be the map of \eqref{eq_big_diagram}.
\begin{enumerate}[label={\rm(\roman*)}]
\item \label{thm_classification_i} Suppose $(\cA,\sigma,f)$ is a quadratic triple. Then, there exists $\lambda_f \in \xi(S)^{-1}(1_{\cA})$ such that $f$ is the linear form $f_{\lambda_f}$ of Example \ref{ex_cover_construction}\ref{ex_cover_construction_b}.
\item \label{thm_classification_ii} The map of sets
\begin{align*}
\left\{\begin{array}{c} f \text{ such that } (\cA,\sigma,f) \\ \text{is a quadratic triple}\end{array}\right\} &\to \xi(S)^{-1}(1_{\cA}) \\
f &\mapsto \lambda_f
\end{align*}
is a bijection.
\end{enumerate}
\end{thm}
\begin{proof}
Both claims follow from Proposition \ref{lem_pair_affine} and the fact that the functor
\begin{align*}
\xi^{-1}(1_{\cA}) \colon \Sch_S &\to \Sets \\
T &\mapsto \xi(T)^{-1}(1_{\cA}|_T)
\end{align*}
is a subsheaf of $\cA/\cAlt_{\cA,\sigma}$. In detail:

\noindent\ref{thm_classification_i}: Let $(\cA,\sigma,f)$ be a quadratic triple. Let $\{U_i \to S\}_{i\in I}$ be an affine cover of $S$. Then $(\cA|_{U_i},\sigma|_{U_i},f|_{U_i})$ is a quadratic triple over an affine scheme, and so by Proposition \ref{lem_pair_affine} there exists an $\ell_i \in \cA(U_i)$ such that $\overline{\ell_i} \in \xi^{-1}(1_{\cA})(U_i)$ is the unique element corresponding to $f|_{U_i}$. Since $(f|_{U_i})|_{U_{ij}} = (f|_{U_j})|_{U_{ij}}$  we have that $\ell_i|_{U_{ij}}-\ell_j|_{U_{ij}} \in \cAlt_{\cA,\sigma}(U_{ij})$ and so $\overline{\ell_i}|_{U_{ij}} = \overline{\ell_j}|_{U_{ij}}$. Therefore, these sections glue together into an element $\lambda_f \in \xi^{-1}(1_{\cA})(S)$. It is then clear that $f=f_{\lambda_f}$ since $f|_{U_i} = \Trd_{\cA}(\lambda_f|_{U_i}\cdot\und)=f_{\lambda_f}|_{U_i}$ for all $i\in I$ by construction.

\noindent\ref{thm_classification_ii}: Bijectivity is clear from the proof of \ref{thm_classification_i} since the correspondence $f|_{U_i} \rightarrow \lambda_f|_{U_i}$ will be bijective over each $U_i$ in the chosen affine cover.
\end{proof}
\begin{remark}\label{rem_skew_alt_torsor}
We may equivalently view Theorem \ref{thm_classification} as saying that there is an isomorphism of sheaves between
\begin{align*}
\cF \colon \Sch_S &\to \Sets \\
T &\mapsto \left\{\begin{array}{c} f \text{ such that } (\cA|_T,\sigma|_T,f) \\ \text{is a quadratic triple over } T\end{array}\right\}
\end{align*}
and the sheaf $\xi^{-1}(1_{\cA})$. Taking part of \eqref{eq_big_diagram}, we have an exact sequence
\[
0 \to \cSkew_{\cA,\sigma}/\cAlt_{\cA,\sigma} \to \cA/\cAlt_{\cA,\sigma} \xrightarrow{\xi} \cSymd_{\cA,\sigma} \to 0
\]
with corresponding long exact sequence
\[
\ldots \to (\cA/\cAlt_{\cA,\sigma})(S) \xrightarrow{\xi(S)} \cSymd_{\cA,\sigma}(S) \xrightarrow{\delta'} H^1(S,\cSkew_{\cA,\sigma}/\cAlt_{\cA,\sigma}) \to \ldots
\]
By Lemma \ref{lem_torsor_connecting}, $\delta'(1_{\cA}) \in H^1(S,\cSkew_{\cA,\sigma}/\cAlt_{\cA,\sigma})$ is the isomorphism class of $\xi^{-1}(1_{\cA})$. In Definition \ref{defn_obs} below we call this element the weak obstruction.
\end{remark}

\section{Quadratic Pairs on Tensor Products}\label{tens}

Given two Azumaya $\cO$--algebras $\cA_1$ and $\cA_2$, their tensor product $\cA_1\otimes_{\cO} \cA_2$ is again an Azumaya $\cO$--algebra. If in addition these algebras are equipped with involutions $\si_1$ and $\si_2$ respectively, then $(\cA_1\otimes_{\cO} \cA_2,\si_1\otimes \si_2)$ is an algebra with involution. Throughout this section we will use the notation
\[
(\cA,\si):=(\cA_1\otimes_{\cO} \cA_2,\si_1\otimes \si_2).
\]
Let $\varepsilon_1,\varepsilon_2 \in \bmu_2(S)$ be the types of $\si_1$ and $\si_2$. Since the type of an involution can be determined locally, we may apply \cite[8.1.3(1)]{K} which applies over affine schemes, to see that the type of $\si_1\otimes \si_2$ will be the product $\varepsilon_1\varepsilon_2$. For our purposes we want $\si_1\otimes \si_2$ to be orthogonal, so we will focus on two cases: when $\si_i$ are both orthogonal ($\varepsilon_i=1$), and when they are both symplectic (in which case $\varepsilon_i=-1$). In preparation, we generalize \cite[5.17]{KMRT} to our setting. We use the notation $\cF^\sharp$ to denote the fppf sheafification of a presheaf $\cF$ as in \cite[Tag 03NS]{Stacks}.

\begin{lem}\label{lem_tensor_Symdecomp}
Let $(\cA_1,\si_1)$ and $(\cA_2,\si_2)$ be two Azumaya $\cO$--algebras with involution. If $\si_i$ are both orthogonal, then
\begin{enumerate}[label={\rm (\roman*)}]
\item \label{lem_tensor_Symdecomp_i} $\cSym_{\cA,\si} = \big(\cSymd_{\cA,\si} + (\cSym_{\cA_1,\si_1}\otimes_{\cO} \cSym_{\cA_2,\si_2})\big)^\sharp$,
\end{enumerate}
If instead $\si_i$ are both weakly symplectic, then
\begin{enumerate}[label={\rm (\roman*)}]
\setcounter{enumi}{1}
\item \label{lem_tensor_Symdecomp_iii} $\cSym_{\cA,\si} = \big(\cSymd_{\cA,\si} + (\cSkew_{\cA_1,\si_1}\otimes_{\cO} \cSkew_{\cA_2,\si_2})\big)^\sharp$.
\end{enumerate}
\end{lem}
\begin{proof}
\noindent \ref{lem_tensor_Symdecomp_i}: If there exists an fppf cover over which the restrictions of each sheaf are equal, then the sheaves are equal globally. Since $\cA_1$ and $\cA_2$ are Azumaya $\cO$--algebras, there exists a sufficiently fine cover $\{T_i\to S\}_{i\in I}$ such that $(\cA_1,\si_1)|_{T_i} \cong (\cM_1\otimes_{\cO} \cM_1,\tau)$ and $(\cA_2,\si_2)|_{T_i} \cong (\cM_2\otimes_{\cO} \cM_2,\tau)$ for free $\cO|_{T_i}$--modules of finite rank $\cM_1$ and $\cM_2$ with the switch involution $\tau$ by \eqref{quamo1}. Now, focusing on one of the $T_i$, let $U\in \Sch_{T_i}$ be an affine scheme. Letting $A_i=\cA_i(U)$ with $\cA_0=\cA$, we know that
\[
\cSym_{\cA_i,\sigma_i}(U) = \{a\in A_i \mid \sigma_i(a)=a\} = \Symm(A_i,\sigma_i)
\]
and, arguing as in Remark \ref{rem_affine_classification}, we have that
\[
\cSymd_{\cA_i,\sigma_i}(U) = \{a+\sigma_i(a) \mid a\in A_i\} = \Symd(A_i,\sigma_i).
\]
Certainly
\begin{align*}
&\Symd(A,\sigma)+\big(\Symm(A_1,\sigma_1)\otimes_{\cO(U)}\Symm(A_2,\sigma_2)\big) \\
\subseteq &\big(\cSymd_{\cA,\si} + (\cSym_{\cA_1,\si_1}\otimes_{\cO} \cSym_{\cA_2,\si_2})\big)^\sharp(U) \\
\subseteq &\Symm(A,\sigma),
\end{align*}
so we are left to show that the reverse inclusion holds. Because we have localized sufficiently so that $(\cA_1,\si_1)|_{T_i} \cong (\cM_1\otimes_{\cO} \cM_1,\tau)$ and $(\cA_2,\si_2)|_{T_i} \cong (\cM_2\otimes_{\cO} \cM_2,\tau)$, it is sufficient to consider the case of $M_1$ and $M_2$ being free modules of finite rank over a ring $R$, and prove that
\begin{align*}
&\Symm\big((M_1\otimes_R M_1)\otimes_R (M_2\otimes_R M_2),\tau\otimes\tau\big)\\
=\; &\Symd\big((M_1\otimes_R M_1)\otimes_R (M_2\otimes_R M_2),\tau\otimes\tau\big)\\
&+ \Symm\big(M_1\otimes_R M_1,\tau)\otimes_R\Symm\big(M_2\otimes_R M_2,\tau)
\end{align*}
Clearly the second module is included in the first. Now, let $\{m_1,\ldots,m_a\}$ and $\{n_1,\ldots,n_b\}$ be free bases of $M_1$ and $M_2$ respectively. Then the module $\Symm\big((M_1\otimes_R M_1)\otimes_R (M_2\otimes_R M_2),\tau\otimes\tau\big)$ has a basis consisting of the elements
\begin{align*}
(m_i\otimes m_i)&\otimes (n_k\otimes n_k), \\
(m_i\otimes m_j)\otimes(n_k\otimes n_l) &+ (m_j\otimes m_i)\otimes(n_l\otimes n_k)
\end{align*}
for $i,j\in\{1,\ldots,a\}$, $k,l\in\{1,\ldots,b\}$ with either $i\neq j$ or $k\neq l$. The first type of basis element is contained in $\Symm\big(M_1\otimes_R M_1,\tau)\otimes_R\Symm\big(M_2\otimes_R M_2,\tau)$, and the second type is contained in $\Symd\big((M_1\otimes_R M_1)\otimes_R (M_2\otimes_R M_2),\tau\otimes\tau\big)$, so therefore the converse inclusion holds as well.

Hence, we have shown that
\[
\cSym_{\cA,\si}|_{T_i}(U) = \big(\cSymd_{\cA,\si} + (\cSym_{\cA_1,\si_1}\otimes_{\cO} \cSym_{\cA_2,\si_2})\big)^\sharp|_{T_i}(U)
\]
for all affine schemes $U\in \Sch_{T_i}$. By applying \cite[Tag 021V]{Stacks}, this means that these sheaves are equal over $T_i$ and thus we are done.

\noindent \ref{lem_tensor_Symdecomp_iii}: Since $\si_i$ are now weakly symplectic, i.e., of type $-1$, we will have that $(\cA_j,\si_j)|_{U_i} \cong (\cM_j\otimes_{\cO} \cM_j,-\tau)$ for $j=1,2$. Accounting for the facts that $(-\tau)\otimes(-\tau)=\tau\otimes\tau$,
\begin{align*}
\Skew(M_j\otimes_R M_j,-\tau) &= \Symm(M_j\otimes_R M_j,\tau) \text{, and}\\
\Alt(M_j\otimes_R M_j,-\tau) &= \Symd(M_j\otimes_R M_j,\tau),
\end{align*}
all other details are the same as in \ref{lem_tensor_Symdecomp_i}.
\end{proof}

\begin{remark}
With the same methods as above one can also generalize \cite[5.16]{KMRT}: If $\si_i$ are both orthogonal, then
\begin{enumerate}[label={\rm (\roman*)}]
\item \label{lem_tensor_Symdecomp_ii} $\begin{aligned}[t]
\cSymd_{\cA,\si}&\cap \big(\cSym_{\cA_1,\si_1}\otimes_{\cO}\cSym_{\cA_2,\si_2}\big)\\
=&\left(\big(\cSymd_{\cA_1,\si_1}\otimes_{\cO}\cSym_{\cA_2,\si_2}\big)+\big(\cSym_{\cA_1,\si_1}\otimes_{\cO}\cSymd_{\cA_2,\si_2}\big)\right)^\sharp.
\end{aligned}$
\end{enumerate}
and if $\si_i$ are both weakly symplectic, then
\begin{enumerate}[label={\rm (\roman*)}]
\setcounter{enumi}{1}
\item \label{lem_tensor_Symdecomp_iv} $\begin{aligned}[t]
\cSymd_{\cA,\si}&\cap \big(\cSkew_{\cA_1,\si_1}\otimes_{\cO}\cSkew_{\cA_2,\si_2}\big)\\
=&\left(\big(\cAlt_{\cA_1,\si_1}\otimes_{\cO}\cSkew_{\cA_2,\si_2}\big)+\big(\cSkew_{\cA_1,\si_1}\otimes_{\cO}\cAlt_{\cA_2,\si_2}\big)\right)^\sharp.
\end{aligned}$
\end{enumerate}
We warn that the naive generalization of \cite[5.15]{KMRT}, which states that for central simple algebras $(A_1,\si_1)$ and $(A_2,\si_2)$ with involutions of the first kind over a field $\FF$ of characteristic $2$ we have
\begin{align*}
&\Symd(A_1,\si_1)\otimes_{\FF}\Symd(A_2,\si_2) \\
=& \big(\Symd(A_1,\si_1)\otimes_\FF \Symm(A_2,\si_2)\big)\cap \big(\Symm(A_1,\si_1)\otimes_{\FF}\Symd(A_2,\si_2)\big),
\end{align*}
does not hold in our generality. It fails when $2$ is neither invertible nor zero. For example, take $S=\Spec(\ZZ)$ and consider the Azumaya algebra $\cA=\Mat_2(\cO)$ with the split orthogonal involution $\eta_0$ in degree $2$ from Example \ref{split_example}\ref{split_example_b}. In particular,
\[
\eta_0\left(\begin{bmatrix} a & b \\ c & d \end{bmatrix}\right) = \begin{bmatrix} d & b \\ c & a \end{bmatrix}.
\]
Considering just the $\ZZ$--module with involution $(A,\eta_0)=(\Mat_2(\ZZ),\eta_0)$, we have
\[
\Symm(A,\eta_0) = \left\{ \begin{bmatrix} a & b \\ c & a \end{bmatrix} \right\} \text{ and } \Symd(A,\eta_0) = \left\{ \begin{bmatrix} a & 2b \\ 2c & a \end{bmatrix} \right\}.
\]
Therefore,
\[
\Symd(A,\eta_0)\otimes_{\ZZ} \Symd(A,\eta_0) = \left\{ \begin{bmatrix} a & 2b & 2c & 4d \\ 2e & a & 4f & 2c \\ 2g & 4h & a & 2b \\ 4i & 2g & 2e & a \end{bmatrix}\right\}
\]
while
\begin{align*}
&\big(\Symd(A,\eta_0)\otimes_\ZZ \Symm(A,\eta_0)\big)\cap(\Symm(A,\eta_0)\otimes_\ZZ \Symd(A,\eta_0)\big) \\
=& \left\{ \begin{bmatrix} a & 2b & 2c & 2d \\ 2e & a & 2f & 2c \\ 2g & 2h & a & 2b \\ 2i & 2g & 2e & a \end{bmatrix}\right\}
\end{align*}
and so these modules are not equal.
\end{remark}

Lemma \ref{lem_tensor_Symdecomp} will be used similarly to how \cite[5.16, 5.17]{KMRT} are used to prove uniqueness claims about quadratic pairs on tensor products. If the algebra with involution $(\cA,\si)$ extends to a quadratic triple with some linear form $f$, then the behaviour of $f$ is prescribed on $\cSymd_{\cA,\si}$. Therefore, by Lemma \ref{lem_tensor_Symdecomp} such $f$ will be uniquely determined by its behaviour on $\cSym_{\cA_1,\si_1}\otimes_{\cO}\cSym_{\cA_2,\si_2}$. We use this in the next two propositions, generalizing \cite[5.18, 5.20]{KMRT}.

\begin{prop}\label{tens_triple_invol}
Let $(\cA_1,\si_1,f_1)$ be a quadratic triple and let $(\cA_2,\si_2)$ be an Azumaya $\cO$--algebra with orthogonal involution. Then there exists a unique quadratic triple $(\cA,\si, f_{1*})$ such that, for sections $s_i\in \cSym_{\cA_i,\si_i}$ we have
\begin{equation} \label{ten_triple_invol_eq}
f_{1*}(s_1\otimes s_2) = f_1(s_1)\Trd_{\cA_2}(s_2).
\end{equation}
\end{prop}
\begin{proof}
Consider an affine cover $\{U_i\to S\}_{i\in I}$ of $S$. We know by Proposition \ref{lem_pair_affine}\ref{lem_pair_affine_i} that each $f|_{U_i}$ is described by an $\ell_i \in \cA_1(U_i)$ with $\ell_i + \si_1(\ell_i)=1$. Then
\[
(\ell_i\otimes 1)+(\si_1\otimes\si_2)(\ell_i\otimes 1) = (\ell_i+\si_1(\ell_i))\otimes 1 = 1\otimes 1 = 1
\]
and since $\ell_i|_{U_{ij}}-\ell_j|_{U_{ij}} \in \cAlt_{\cA_1,\si_1}(U_{ij})$, we also have that
\[
(\ell_i\otimes 1)|_{U_{ij}}-(\ell_j\otimes 1)|_{U_{ij}} = \big(\ell_i|_{U_{ij}}-\ell_j|_{U_{ij}}\big)\otimes 1 \in \cAlt_{\cA_1\otimes_{\cO}\cA_2,\si_1\otimes \si_2}(U_{ij}).
\]
Therefore, we may use the construction of Example \ref{ex_cover_construction}\ref{quafoex_c} with the elements $\ell_i\otimes 1 \in \cA(U_i)$ to define a global $f_{1*}$. It remains to show that this $f_{1*}$ behaves as claimed in \eqref{ten_triple_invol_eq} on sections $s_i\in \cSym_{\cA_i,\si_i}(T)$ for some $T\in \Sch_S$. Consider such sections $s_1$ and $s_2$. With respect to the cover $\{T_i:=U_i \times_S T \to T\}_{i\in I}$, our new linear form $f_{1*}$ will be locally described by the elements $(\ell_i\otimes 1)|_{T_i}$. So, locally we have
\begin{align*}
f_{1*,i}((s_1\otimes s_2)|_{T_i}) &= \Trd_{\cA_1\otimes_{\cO} \cA_2}\big((\ell_i\otimes 1)|_{T_i}(s_1\otimes s_2)|_{T_i} \big)\\
&= \Trd_{\cA_1\otimes_{\cO} \cA_2}\big((\ell_i|_{T_i} \cdot s_1|_{T_i})\otimes s_2|_{T_i} \big)\\
&= \Trd_{\cA_1}\big(\ell_i|_{T_i} \cdot s_1|_{T_i}\big)\Trd_{\cA_2}\big(s_2|_{T_i}\big)\\
&= f_1|_{T_i}(s_1|_{T_i})\Trd_{\cA_2}\big(s_2|_{T_i}\big)
\end{align*}
and hence gluing yields $f_{1*}(s_1\otimes s_2) = f_1(s_1)\Trd_{\cA_2}(s_2)$ as desired.
\end{proof}

\begin{example}
Here we generalize \cite[Example 5.19]{KMRT}. On one hand, let $(\calM_1,q_1)$ be a regular quadratic form and let $(\End_{\cO}(\calM_1), \si_1, f_{q_1})$ be the
quadratic pair defined in Proposition \ref{qpex}\ref{qpex_a}. On the other hand, let $(\calM_2,b_2)$ be a regular symmetric bilinear form on
a locally free $\calO$--module of finite positive rank $\calM_2$. We consider the tensor product quadratic form $q$ on $\calM=\calM_1 \otimes_{\calO} \calM_2$
as defined by Sah \cite[Thm. 1]{Sah} in the ring case. This $q$ is regular, and so we consider the attached quadratic pair $(\End_{\cO}(\calM, \si, f_q)$.
The proof of \cite[Example 5.19]{KMRT} shows that there is an isomorphism
\[
(\End_{\cO}(\calM_1), \si_1, f_{q_1}) \otimes  (\End_{\cO}(\calM_2), \si_{b_2}) \simlgr    (\End_{\cO}(\calM), \si, f_q)
\]
where the tensor product stands for the construction of Proposition \ref{tens_triple_invol}.
\end{example}

Turning our attention to the weakly symplectic case, it turns out that two weakly symplectic involutions may not have a quadratic triple structure on their tensor product, see Example \ref{ex_weakly_tens_not_quad} below. To obtain a result analogous to \cite[5.20]{KMRT} we will assume we have two symplectic involutions, i.e., where the types are $\varepsilon_i=-1$ and the local bilinear forms are alternating. This is due in part to the following lemma.

\begin{lem}\label{lem_symp_1_symd}
Let $(\cA,\si)$ be an Azumaya $\cO$--algebra with a weakly symplectic involution. Then, the following are equivalent.
\begin{enumerate}[label={\rm (\roman*)}]
\item \label{lem_symp_1_symd_i} $\si$ is symplectic.
\item \label{lem_symp_1_symd_ii} $\Trd_{\cA}(s) = 0$ for all sections $s\in \cSkew_{\cA,\si}$.
\item \label{lem_symp_1_symd_iii} $1_{\cA} \in \cSymd_{\cA,\si}(S)$.
\end{enumerate}
\end{lem}
\begin{proof}
Since $(\cA,\si)$ is weakly symplectic, there exists an fppf cover $\{T_i \to S\}_{i\in I}$ over which
\[
(\cM_i\otimes_{\cO|_{T_i}}\cM_i,-\tau) \overset{\varphi_b}{\longrightarrow} (\cEnd_{\cO|_{T_i}}(\cM_i),\eta_b) \cong (\cA,\si)|_{T_i}
\]
where $\tau$ is the switch involution and $\varphi_b$ is the isomorphism \eqref{quamo1} for a regular skew-symmetric bilinear form $b$.

\noindent \ref{lem_symp_1_symd_i}$\Leftrightarrow$\ref{lem_symp_1_symd_ii}: For any $T\in \Sch_S$ and a section $s\in \cSkew_{\cA,\si}(T)$, we may compute locally with respect to the cover $\{T_i\times_S T \to T\}_{i\in I}$ where the section will be a linear combination of elements of the form
\begin{align*}
\varphi_b(m&\otimes m) \\
\varphi_b(m\otimes n) &- \eta_b(\varphi_b(m\otimes n))
\end{align*}
for $m,n\in \cM_i(T_i\times_S T)$. Since we have
\[
\Trd_{\cEnd_{\cO|_{T_i}}(\cM_i)}(\varphi_b(m\otimes n)) = b(m,n),
\]
the trace of the second type of element is clearly $0$, while the trace vanishes for all elements of the first type if and only if $b$ is alternating. Hence, $\Trd_{\cA}(s)=0$ for all sections $s\in \cSkew_{\cA,\si}$ if and only if the underlying bilinear forms are alternating, i.e., $\si$ is symplectic.

\noindent \ref{lem_symp_1_symd_ii}$\Leftrightarrow$\ref{lem_symp_1_symd_iii}: By Lemma \ref{lem_perp}\ref{lem_perp_iv} we know $\cSkew_{\cA,\si}^\perp = \cSymd_{\cA,\si}$ and $\cSymd_{\cA,\si}^\perp = \cSkew_{\cA,\si}$, from which the equivalence is immediate.
\end{proof}

Similar to the orthogonal case, having $1_{\cA}\in \cSymd_{\cA,\si}$ will grant us local $\ell_i$ with $\ell_i+\si(\ell_i)=1$ from which we can build a quadratic triple. The following result generalizes \cite[5.20]{KMRT}.

\begin{prop}\label{tens_symplectic}
Let $(\cA_1,\si_1)$ and $(\cA_2,\si_2)$ be two Azumaya $\cO$--algebras with symplectic involutions. Then there exists a unique quadratic triple $(\cA_1\otimes_{\cO} \cA_2,\si_1\otimes\si_2,f_\otimes)$ such that, for sections $s_i\in \cSkew_{\cA_i,\si_i}$ we have
\[
f_\otimes(s_1\otimes s_2)=0.
\]
\end{prop}
\begin{proof}
We first establish uniqueness. Let $f,f' \colon \cSym_{\cA,\si} \to \cO$ be two solutions. Then $f-f'$ vanishes on $\cSymd_{\cA,\si}$
and on $\cSkew_{\cA_1,\si_1} \otimes \cSkew_{\cA_2,\si_2}$. According to Lemma \ref{lem_tensor_Symdecomp}\ref{lem_tensor_Symdecomp_iii},
$\cSymd_{\cA,\si}$ and $\cSkew_{\cA_1,\si_1} \otimes \cSkew_{\cA_1,\si_1}$ generate $\cSymd_{\cA,\si}$ so that $f'=f$.

We now establish existence. Since $(\cA_1,\si_1)$ is symplectic we have $1_{\cA_1}\in \cSymd_{\cA_1,\si_1}$ by Lemma \ref{lem_symp_1_symd}, and so there is an fppf cover $\{T_i\to S\}_{i\in I}$ such that for each $i\in I$ there exists $\ell_i\in \cA_1(T_i)$ with $\ell_i+\si_1(\ell_i)=1$. Then we will have
\[
(\ell_i\otimes 1)+(\si_1\otimes\si_2)(\ell_i\otimes 1)= 1,
\]
so to use the construction of Examples \ref{ex_cover_construction}\ref{quafoex_c} we need to check that \break $(\ell_i\otimes 1)|_{T_{ij}} - (\ell_j\otimes 1)|_{T_{ij}} \in \cAlt_{\cA,\si}(T_{ij})$. By Lemma \ref{lem_perp}\ref{lem_perp_iii} and Lemma \ref{lem_tensor_Symdecomp}\ref{lem_tensor_Symdecomp_iii}, it is sufficient to show that these elements are orthogonal to sections of the form $a+\si(a)$ for any $a\in \cA|_{T_{ij}}$, and $s_1\otimes s_2$ for sections $s_i\in \cSkew(\cA_i,\si_i)|_{T_{ij}}$. For the first type, we have
\[
\Trd_{\cA}((\ell_i\otimes 1)|_{T_{ij}}(a+\si(a))) = \Trd_{\cA_1\otimes_{\cO}\cA_2}(a) = \Trd_{\cA}((\ell_j\otimes 1)|_{T_{ij}}(a+\si(a)))
\]
and so
\[
\Trd_{\cA}((\ell_i\otimes 1|_{T_{ij}}-\ell_j\otimes 1|_{T_{ij}})(a+\si(a)))=0.
\]
For the second type, we have
\begin{align*}
\Trd_{\cA}\big((\ell_i\otimes 1)|_{T_{ij}}(s_1\otimes s_2) \big)=\Trd_{\cA_1}(\ell_i|_{T_{ij}} s_1)\Trd_{\cA_2}(s_2)=0
\end{align*}
since $\Trd_{\cA_2}(s_2)=0$ by Lemma \ref{lem_symp_1_symd}\ref{lem_symp_1_symd_ii}. Similarly $\Trd_{\cA}\big((\ell_j\otimes 1)|_{T_{ij}}(s_1\otimes s_2) \big)=0$ and so $(\ell_i\otimes 1)|_{T_{ij}}-(\ell_j\otimes 1)|_{T_{ij}}$ is orthogonal to $s_1\otimes s_2$ as well. Thus there exists a global $f_\otimes \colon \cSym_{\cA,\si}\to\cO$ constructed as in Example \ref{ex_cover_construction}\ref{quafoex_c}. The above local calculations also imply that $f_{\otimes}(s_1\otimes s_2)=0$ for sections $s_1\in \cSkew_{\cA_1,\si_1}$ and $s_2\in \cSkew_{\cA_2,\si_2}$, and so we have constructed the unique quadratic triple we desired.
\end{proof}

Note that both $(\cA_1,\si_1)$ and $(\cA_2,\si_2)$ were required to be symplectic in the above proposition. We needed $(\cA_1,\si_1)$ to be symplectic to have $1_{\cA_1}\in \cSymd_{\cA_1,\si_1}$, and $(\cA_2,\si_2)$ to be symplectic to have $\Trd_{\cA_2}(s_2)=0$. If $(\cA_2,\si_2)$ were only weakly symplectic, then we could construct the local maps $f_{\otimes,i}$ just the same, but it is not clear if the elements $(\ell_i\otimes 1)|_{T_{ij}}-(\ell_j\otimes 1)|_{T_{ij}}$ belong to $\cAlt_{\cA,\si}(T_{ij})$, and so there may be examples with no quadratic triple on the tensor product.

If both involutions are only weakly symplectic there is almost no hope, even over affine schemes as in the following example.
\begin{example}\label{ex_weakly_tens_not_quad}
Let $S=\Spec(\ZZ/2\ZZ)$. Since we are over an affine scheme we may simply work with Azumaya algebras over $R=\ZZ/2\ZZ$. Consider the Azumaya algebra $\Mat_2(R)$ with involution $\si(B) = u^{-1} B^T u$ for
\[
u = \begin{bmatrix} 1 & 1 \\ 1 & 0 \end{bmatrix}.
\]
This is weakly symplectic since $u^T=-u$, but not symplectic (for example by Lemma \ref{lem_symp_1_symd}).
\lv{It is weakly symplectic since $u^T = -u$, but it is not symplectic since
\[
b((1,0),(1,0))= \begin{bmatrix} 1 & 0\end{bmatrix}\begin{bmatrix} 1 & 1 \\ 1 & 0 \end{bmatrix}\begin{bmatrix} 1 \\ 0\end{bmatrix} = 1 \neq 0.
\]}%
Note that
\[
s = \begin{bmatrix} 1 & 1 \\ 0 & 0 \end{bmatrix} \in \Skew(\Mat_2(R),\si)
\]
is a skew-symmetric element. If there were a quadratic triple $(\Mat_2(R)\otimes_R \Mat_2(R),\si\otimes\si,f)$, then necessarily
\[
2f(s\otimes s) = \Trd_{\Mat_2(R)\otimes_R \Mat_2(R)}(s\otimes s) = \Trd_{\Mat_2(R)}(s)^2 = 1^2 = 1,
\]
however there is no $x\in \ZZ/2\ZZ$ such that $2x=1$. Therefore no such $f$ exists. Since we are over an affine scheme this means that $\si\otimes\si$ fails to even be locally quadratic.
\end{example}

\section{Obstructions to Quadratic Pairs}\label{obs}
In this section we introduce cohomological obstructions which prevent an Azumaya $\cO$--algebra with locally quadratic involution $(\cA,\sigma)$ from being extended to a quadratic triple, either outright or where the triple has certain properties. We begin by considering portions of various long exact cohomology sequences arising from \eqref{eq_big_diagram}. In particular we have the following diagram
\begin{equation}\label{eq_long_exacts}
\begin{tikzcd}[column sep=0.3in]
\cA(S) \arrow{r}{\Id+\sigma} \arrow{d} & \cSymd_{\cA,\sigma}(S) \arrow{r}{\delta} \arrow[equals]{d} & \check{H}^1(S,\cSkew_{\cA,\sigma}) \arrow{d} \\
(\cA/\cAlt_{\cA,\sigma})(S) \arrow{r}{\xi(S)} \arrow{d}{c} & \cSymd_{\cA,\sigma}(S) \arrow{r}{\delta'} & \check{H}^1(S,\cSkew_{\cA,\sigma}/\cAlt_{\cA,\sigma}) \\
\check{H}^1(S,\cAlt_{\cA,\sigma}) & &
\end{tikzcd}
\end{equation}
with exact left column and exact rows.

\begin{defn}[Strong and Weak Obstructions]\label{defn_obs}
Let $(\cA,\sigma)$ be an Azumaya $\cO$--algebra with locally quadratic involution. Then $1_{\cA}\in \cSymd_{\cA,\sigma}(S)$ and we call $\Strong(\cA,\si)=\delta(1_{\cA})$ the \emph{strong obstruction} and call $\Weak(\cA,\si)=\delta'(1_{\cA})$ the \emph{weak obstruction}.
\end{defn}

The strong and weak obstructions prevent the existence of a quadratic triple involving $\si$ in the following way.

\begin{thm}\label{prop_obstructions}
Let $(\cA,\si)$ be an Azumaya $\cO$--algebra with a locally quadratic involution. Then,
\begin{enumerate}[label={\rm (\roman*)}]
\item \label{prop_obstructions_i} There exists a linear map $f\co \cA \to \cO$ such that $(\cA,\si,f|_{\cSym_{\cA,\si}})$ is a quadratic triple if and only if $\Strong(\cA,\si)=0$. In this case $f=\Trd_{\cA}(\ell \und)$ for an element $\ell \in \cA(S)$ with $\ell+\si(\ell)=1$.
\item \label{prop_obstructions_ii} There exists a linear map $f\co \cSym_{\cA,\si} \to \cO$ such that $(\cA,\si,f)$ is a quadratic triple if and only if $\Weak(\cA,\si)=0$.
\end{enumerate}
\end{thm}
\begin{proof}
\noindent \ref{prop_obstructions_i}: Assume we have $f\co \cA \to \cO$ such that $(\cA,\si,f|_{\cSym_{\cA,\si}})$ is a quadratic triple. Then trivially $f|_{\cSym_{\cA,\si}}$ can be extended to $\cA$, and so by Lemma \ref{lem_extend_l} there exists an $\ell\in\cA(S)$ such that $\ell+\si(\ell)=1_{\cA}$. Thus $1_{\cA}$ is in the image of $\cA(S)\to \cSymd_{\cA,\si}(S)$ in the long exact cohomology sequence, and so $\Strong(\cA,\si)=0$. Conversely, if $\Strong(\cA,\si)=0$ then $1_{\cA}$ is in the same image and so we obtain such an $\ell$. Using Example \ref{quafoex}\ref{quafoex_b} we can construct $f\co \cA \to \cO$ with the desired property.

\noindent \ref{prop_obstructions_ii}: Since the rows in \eqref{eq_long_exacts} are exact, $\Weak(\cA,\sigma)=0$ if and only if $\xi(S)^{-1}(1_{\cA})$ is non-empty. By Theorem \ref{thm_classification}, this means $\Weak(\cA,\sigma)=0$ if and only if $(\cA,\sigma)$ can be extended to a quadratic triple.
\end{proof}
By Remark \ref{rem_skew_alt_torsor}, $\Weak(\cA,\sigma)$ is the isomorphism class of the $\cSkew_{\cA,\sigma}/\cAlt_{\cA,\sigma}$--torsor $\xi^{-1}(1_{\cA})$. Another view of Theorem \ref{prop_obstructions}\ref{prop_obstructions_ii} is that $\Weak(\cA,\sigma)=0$ if and only if this $\cSkew_{\cA,\sigma}/\cAlt_{\cA,\sigma}$--torsor is trivial. A torsor is trivial if and only if it has a global section and global sections of $\xi^{-1}(1_{\cA})$ correspond to forms $f$ which make $(\cA,\sigma,f)$ a quadratic triple.

For examples of locally quadratic involutions with non-trivial strong or weak obstructions, see section \ref{obs_examples}.

\begin{remark}{\rm
Since $2_{\cA} = 1_{\cA}+\sigma(1_{\cA})$ is always in the image of $\cA(S) \to \cSymd_{\cA,\sigma}(S)$, it follows that $\delta(2_{\cA})=2 \Strong(\cA,\si)=0$ so that $2\Weak(\cA,\si)=0$ also.
In particular, if $2$ is invertible over $S$ we recover the fact
 that $(\cA,\si)$ always extends to a quadratic triple.
  }
\end{remark}

The following lemma describes a special case in which we know the strong obstruction is non-trivial. An example where this occurs is given in Example \ref{example_strong} and Lemma \ref{lem_mukai}.
\begin{lem}\label{lem_nontrivial_strong}
Assume $S$ is non-empty. Let $(\cA,\si)$ be an Azumaya $\cO$--algebra with locally quadratic orthogonal involution such that $\cO(S) \iso \cA(S)$ and $2\cO(S) = 0$. Then, the strong obstruction $\Om(\cA, \si)$ is non-trivial.
\end{lem}
\begin{proof}  
The exact sequence
\[
0 \to \cSkew_{\calA,\sigma} \to \calA \to \cSymd_{\calA, \sigma}
\to 0
\]
gives rise to the long exact cohomology sequence
\begin{equation}\label{seq17}
0 \to \cSkew_{\calA, \sigma}(S) \to \calA(S) \to \cSymd_{\calA, \sigma}(S)
\xrightarrow{\delta} \check{H}^1(S, \cSkew_{\calA, \sigma}).
\end{equation}
Since $S$ is non-empty and $(\calA, \sigma)$ is locally quadratic, we have $0\neq 1 \in \cSymd_{\calA, \sigma}(S)$. We recall that $\Strong(\calA, \sigma)=\delta(1)$ by Definition \ref{defn_obs}. Since $\calA(S)=\cO(S)$ and $\si(S)$ is $\cO(S)$--linear, all elements of $\cA(S)$ are symmetric, and since $2\cO(S)=0$, all elements of $\cA(S)$ are also skew-symmetric. Hence, $\cSkew_{\cA,\si}(S)=\cO(S)$ also. Therefore, the sequence \eqref{seq17} is of the form
\[
0 \to \cO(S) \xrightarrow{\Id} \cO(S) \to \cSymd_{\calA, \sigma}(S)
\xrightarrow{\delta} \check{H}^1(S, \cSkew_{\calA, \sigma})
\]
and we obtain that $\Strong(\calA, \sigma)=\delta(1) \neq 0$ because 
$\cO(S) \to \cSymd_{\cA,\si}(S)$ is the zero map  in view of surjectivity of $\Id\colon\cO(S)\to \cO(S)$, and so $\delta$ is injective.
\end{proof}

We will make use of the following description in order to perform computations with $\Strong(A,\si)$.
\begin{lem}\label{lem_strong_rep}
Let $(\cA,\si)$ be an Azumaya $\cO$--algebra with locally quadratic involution. The strong obstruction takes the form
\[
\Strong(\cA,\si) = \left[(\ell_j|_{T_{ij}}-\ell_i|_{T_{ij}})_{i,j\in I}\right] \in \check{H}^1(S,\cSkew_{\cA,\sigma})
\]
for any cover $\{T_i\to S\}_{i\in I}$ and elements $\ell_i\in \cA(T_i)$ such that $\ell_i+\si(\ell_i)=1$.
\end{lem}
\begin{proof}
This follows immediately from Lemma \ref{app_connecting_image}.
\end{proof}

\begin{remarks}
\begin{enumerate}[wide, labelindent=0pt, label={\rm (\alph*)}]
\item By Lemma \ref{lem_smart_equiv}, if we wish we may represent $\Strong(\cA,\si)$ as in Lemma \ref{lem_strong_rep} with an affine open cover $\{U_i \to S\}_{i\in I}$.
\item Since the weak obstruction $\Weak(\cA,\si)$ is simply an image of the strong obstruction, it will be represented similarly by the class
\[
\left[(\overline{\ell_j|_{T_{ij}} - \ell_i|_{T_{ij}}})_{i,j\in I}\right] \in \check{H}^1(S,\cSkew_{\cA,\si}/\cAlt_{\cA,\si})
\]
for a cover $\{T_i\to S\}_{i\in I}$ and elements $\ell_i\in \cA(T_i)$ with $\ell_i+\si(\ell_i)=1$, where the overline denotes the image in $(\cSkew_{\cA,\si}/\cAlt_{\cA,\si})(T_{ij})$. This cover can of course also be taken to be an affine open cover if desired.
\end{enumerate}
\end{remarks}

\subsection{Alternate Obstructions to Extend a Form}\label{sec_alternate_obstructions}
Let $(\calA, \sigma, f)$ be a quadratic triple. Lemma \ref{lem_extend_l} above states that $f$ extends from $\cSym_{\calA, \sigma}$ to $\calA$
if and only if $f$ arises from some $\ell \in \cA(S)$ satisfying $\ell + \sigma(\ell)=1$. Since $\sigma$ is orthogonal, $\cAlt_{\cA,\sigma}$ is finite locally free by Lemma \ref{locdirsum}\ref{locdirsum_ii} and hence is quasi-coherent. The cohomology theory of quasi-coherent sheaves provides an obstruction to such an extension, in particular an obstruction $c(f) \in \check{H}^1(S, \cAlt_{\cA,\sigma})$ as follows. By Theorem \ref{thm_classification}, the form $f$ corresponds to an element $\lambda_f \in (\cA/\cAlt_{\cA,\sigma})(S)$, and we may consider its image under the map $c$ in \eqref{eq_long_exacts}. We set $c(f) = c(\lambda_f) \in \check{H}^1(S, \cAlt_{\cA,\sigma})$. Clearly this obstruction prevents $\lambda_f=\overline{\ell}$ for some $\ell\in \cA(S)$, which would be the case exactly when $f$ can be extended to all of $\cA$.

Consider an affine covering $\{U_i \to S\}_{i\in I}$ of $S$. By Proposition \ref{lem_pair_affine}\ref{lem_pair_affine_i}, each restriction $f|_{U_i}$ will be of the form $\Trd_{\cA|_{U_i}}(\ell_i\und)$ for an $\ell_i\in \cA(U_i)$ with $\overline{\ell_i} = \lambda_f|_{U_i}$. Lemma \ref{app_connecting_image} then tells us that the 1-cocycle $(\ell_j|_{U_{ij}} - \ell_i|_{U_{ij}})_{i,j\in I}$ represents the cohomological obstruction $c(f)$.

However, we also obtain another obstruction in the following way. We have an exact sequence of sheaves
\[
0 \to \cSym_{\cA,\si} \to \cA \to \cAlt_{\cA,\si} \to 0
\]
which, since the sheaves are finite locally free, we may dualize to obtain another exact sequence,
\[
0 \to \cAlt_{\cA,\si}^\vee \to \cA^\vee \to \cSym_{\cA,\si}^\vee \to 0.
\]
A portion of the long exact cohomology sequence associated with this exact sequence is
\[
\ldots \to \cA^\vee \to \cSym_{\calA, \sigma}^\vee \to \check{H}^1(S,\cAlt_{\cA,\si}^\vee) \to \ldots
\]
Since $f \in \cSym_{\calA, \sigma}^\vee(S)$, we then get a class $c'(f) \in  \check{H}^1( S, \cAlt_{\calA,\sigma}^\vee)$ which is the obstruction to extend $f$ to $\calA$. Not surprisingly, we can compare the two obstructions.

\begin{lem}\label{lem_comparison}
Let $(\cA,\si,f)$ be a quadratic triple, and let $\widehat{b_-}\colon \cAlt_{A, \sigma} \simlgr  \cAlt_{A, \sigma}^\vee$ be the isomorphism associated to the regular bilinear form $b_-$ provided by Lemma \ref{lem_pairing}\ref{lem_pairing2}. Then,
\begin{enumerate}[label={\rm (\roman*)}]
\item \label{lem_comparison1}
$c'(f)= \widehat{b_-}(c(f))$.
\item \label{lem_comparison2}
The image of $c(f)$ under the map $\check{H}^1(S, \cAlt_{\calA, \sigma}) \to \check{H}^1(S, \cSkew_{\calA, \sigma})$ induced by the inclusion $\cAlt_{\cA,\si}\to \cSkew_{\cA,\si}$, is $\Strong(\cA, \sigma)$.
\end{enumerate}
\end{lem}

\begin{proof}
\ref{lem_comparison1}: Take an affine open covering $\{U_i\to S\}_{i \in I}$ of $S$.
Then $f|_{U_i}(s)=\Trd_{\calA}(\ell_i s)$ for some
$\ell_i  \in \calA(U_i)$
 satisfying $\ell_i+ \sigma(\ell_i)=1$ by Proposition \ref{lem_pair_affine}\ref{lem_pair_affine_i}.
By definition $c(f)$ is the class of the $1$--cocycle
$c_{ij}=\ell_j|_{U_{ij}} - \ell_i|_{U_{ij}}$ with values in
$\cAlt_{\cA, \sigma}$. We define  $\widetilde f_i(a)=  \Trd_{\calA}(\ell_i a)$ on $\calA|_{U_i}$ for $i\in I$.

On the other hand, $c'(f)$ is the class of the $1$--cocycle
$c'_{ij}=(\widetilde f_j)|_{U_{ij}} - (\widetilde f_i)|_{U_{ij}}$
with values in $\cAlt^\vee_{\calA, \sigma}$.

We need to check that $c_{ij}'= \widehat{b_-}(c_{ij})$ for all $i,j\in I$, or equivalently
that $c'_{ij}= b_{-}(c_{ij},\und)$ over $U_{ij}$.
Up to localization to an affine open subset $V \subseteq U_{ij}$,
we can deal with elements of the form
$x=x'- \sigma(x')$ for some $x' \in \calA(V )$.
It follows that
\begin{eqnarray} \nonumber
c'_{ij}(x) &=& (\widetilde f_j)|_{U_{ij}}(x'-\sigma(x')) -
(\widetilde f_i)|_{U_{ij}}(x'-\sigma(x')) \\ \nonumber
&=&  \Trd_{\calA}\bigl(\ell_j|_{U_{ij}} (x'- \sigma(x') \bigr) -
\Trd_{\calA}\bigl(\ell_i|_{U_{ij}} (x'- \sigma(x') \bigr) \\ \nonumber
&=&  \Trd_{\calA}\bigl( (\ell_j|_{U_{ij}} -\ell_i|_{U_{ij}})
(x'- \sigma(x') \bigr) \\ \nonumber
&=&  b_{-}\bigl( c_{ij}, x\bigr)
\end{eqnarray}
as desired.

\sm

\noindent \ref{lem_comparison2}: We have that $1_{\calA(U_i)}= \ell_i +\sigma(\ell_i)$
for each $i$. By Lemma \ref{lem_strong_rep}, the strong obstruction $\Strong(A, \sigma)$ is represented
by the $1$--cocycle $\ell_j|_{U_{ij}} - \ell_i|_{U_{ij}}$ with values in
$\cSkew_{\cA, \sigma}$. We conclude that
the image of $c(f)$ by the map
 $\check{H}^1(S, \cAlt_{\calA, \sigma}) \to \check{H}^1(S, \cSkew_{\calA, \sigma})$
 is $\Strong(A, \sigma)$.
\end{proof}

\subsection{Obstructions for Tensor Products}\label{subsec_obs_tensor}
Consider Azumaya algebras $(\cA_1,\si_1)$ and $(\cA_2,\si_2)$ either both with orthogonal or both with symplectic involutions. Their tensor product $(\cA,\si)=(\cA_1\otimes_{\cO}\cA_2,\si_1\otimes \si_2)$ is an Azumaya algebra with orthogonal involution. There is a natural map
\begin{align*}
\cA_1 \times \cA_2 &\to \cA \\
(a_1,a_2) &\mapsto a_1\otimes a_2
\end{align*}
which restricts to a morphism
\[
\cSkew_{\cA_1,\si_1} \times \cSym_{\cA_2,\si_2} \to \cSkew_{\cA,\si}
\]
which in turn induces a morphism
\[
(\cSkew_{\cA_1,\si_1}/\cAlt_{\cA_1,\si_1})\times \cSym_{\cA_2,\si_2} \to \cSkew_{\cA,\si}/\cAlt_{\cA,\si}.
\]
We will make use of these tensor morphisms throughout this section to investigate the strong and weak obstructions for $(\cA,\si)$. When $\si_i$ are both orthogonal we will assume $\si_1$ is locally quadratic and we will relate $\Strong(\cA,\si)$ to $\Strong(\cA_1,\si_1)$. When $\si_i$ are both symplectic we also have $1_{\cA_1}\in \cSymd_{\cA_1,\si_1}$ by Lemma \ref{lem_symp_1_symd} and a connecting morphism $\delta_1 \co \cSymd_{\cA_1,\si_1}(S) \to \check{H}^1(S,\cSkew_{\cA_1,\si_1})$. In this case we will relate $\Strong(\cA,\si)$ to $\delta_1(1_{\cA_1})$, which is only technically a strong obstruction when $2=0\in \cO$ and so $\si_1$ would be simultaneously locally quadratic as well as symplectic.

\begin{lem}\label{tens_triple_invol_obs}
Let $(\cA_1,\si_1)$  and  $(\cA_2,\si_2)$ be two Azumaya $\cO$--algebra with orthogonal involution. Let  $(\cA,\si)$ be their tensor product. Assume that $(\cA_1,\si_1)$ is  locally quadratic. Then, we have the following.

\begin{enumerate}[label={\rm (\roman*)}]

\item \label{obs1}   $(\cA,\si)$ is locally quadratic.

\item \label{obs2} We have
\[
\Strong(\cA,\si) =\Strong(\cA_1,\si_1)  \cup 1_{\calA_2} \in \check{H}^1(S, \cSkew_{\calA, \sigma})
\]
where $1_{\calA_2} \in \cSym_{\calA_2, \sigma_2}(S) = H^0(S,\cSym_{\cA_2,\si_2})$ and the cup-product arises from the morphism
\[
\cSkew_{\calA_1, \sigma_1} \times  \cSym_{\calA_2, \sigma_2} \to \cSkew_{\calA, \sigma}.
\]

\item \label{obs3} We have $\Weak(\cA,\si) =\Weak(\cA_1,\si_1) \cup 1_{\calA_2} \in \check{H}^1(S, \cSkew_{\calA, \sigma}/ \cAlt_{\calA, \sigma})$ where the cup-product arises from the morphism
\[
(\cSkew_{\calA_1, \sigma_1} / \cAlt_{\calA_1, \sigma_1}) \times  \cSym_{\calA_2, \sigma_2} \to \cSkew_{\calA, \sigma}/ \cAlt_{\calA, \sigma}.
\]
\end{enumerate}

\end{lem}

\begin{proof}

\noindent \ref{obs1}: The statement is local and follows then from the construction of tensor products in Proposition \ref{tens_triple_invol}.
\sm

\noindent \ref{obs2}: The map $\otimes 1_{\calA_2}\colon \calA_1 \to \calA$, $a_1 \mapsto a_1 \otimes 1_{\calA_2}$, restricts to both $\cSkew_{\cA_1,\si_1}\to\cSkew_{\cA,\si}$ and $\cSymd_{\cA_1,\si_1}\to\cSymd_{\cA,\si}$. Therefore, we have a diagram
\[\xymatrix{
0 \ar[r] &\cSkew_{\cA_1,\si_1} \ar[d]^{\otimes 1_{\calA_2}} \ar[r] & \ar[d]^{\otimes 1_{\calA_2}} \cA_1 \ar[r]^{1+\si_1 \quad } & \cSymd_{\cA_1,\si_1} \ar[d]^{\otimes 1_{\cA_2}} \ar[r] & 0. \\
0 \ar[r] &\cSkew_{\cA,\si } \ar[r] & \cA \ar[r]^{1+\si \quad} & \cSymd_{\cA,\si} \ar[r] & 0.
}\]
whose rows are exact sequences of $\cO$--modules. We claim this diagram commutes. Commutativity of the left square is clear, so we check commutativity of the right square with a computation. Let $a\in \cA_1(T)$ be a section for some $T\in \Sch_S$, then
\begin{multline*}
(1+\si)\circ (\otimes 1_{\cA_2})(a) = (1+\si)(a\otimes 1_{\cA_2}) = a\otimes 1_{\cA_2} + (\si_1\otimes \si_2)(a\otimes 1_{\cA_2})\\
= a\otimes 1_{\cA_2} + \si_1(a)\otimes 1_{\cA_2} =(a+\si_1(a))\otimes 1_{\cA_2} = (\otimes 1_{\cA_2})\circ (1+\si_1)(a).
\end{multline*}
Thus, we get another commutative diagram involving boundary maps,
\[
\begin{tikzcd}
\cSymd_{\cA_1,\si_1}(S)\arrow{r}{\delta_1} \arrow{d}{\otimes 1_{\cA_2}} & \check{H}^1(S,\cSkew_{\cA_1,\si_1})  \arrow{d}{\cup  1_{\calA_2}}  \\
\cSymd_{\cA,\si}(S)  \arrow{r}{\delta} & \check{H}^1(S,\cSkew_{\cA,\si}).
\end{tikzcd}
\]
Since $1_{\cA} = 1_{\cA_1}\otimes 1_{\cA_2}$, the commutativity of the diagram implies that $\delta(1_{\calA})= \delta_1(1_{\calA_1}) \cup 1_{\calA_2}$
whence the desired formula
\[
\Strong(\cA,\si) =\Strong(\cA_1,\si_1)  \cup 1_{\calA_2} \in \check{H}^1(S, \cSkew_{\calA, \sigma}).
\]

\noindent \ref{obs3}: This follows from \ref{obs2} and the commutativity of the square
\[\xymatrix{
\cSkew_{\cA_1,\si_1} \times \cSym_{\cA_2,\si_2}  \ar[d]^{\pi_1 \times \Id} \ar[r] &  \cSkew_{\cA,\si} \ar[d]^\pi \\
(\cSkew_{\cA_1,\si_1}/\cAlt_{\calA_1, \sigma_1}) \times \cSym_{\cA_2,\si_2}   \ar[r] &  \cSkew_{\cA,\si}/ \cAlt_{\calA, \sigma}.
}\]
where the horizontal maps are the tensor morphisms.
\end{proof}

\begin{remark}
Lemma \ref{tens_triple_invol_obs}\ref{obs3} shows that if $(\cA_1,\si_1)$ is extendable to a quadratic pair, then so is $(\cA,\si)$. We actually already knew this from Proposition \ref{tens_triple_invol}. We shall see later in Proposition \ref{prop_serre} that the converse is false.
\end{remark}

A consequence of Proposition \ref{tens_symplectic} is the following.
\begin{cor}\label{cor_symplectic_ext}
Let $(\cA_1,\si_1)$ and $(\cA_2,\si_2)$ be two Azumaya $\cO$--algebras with symplectic involutions. Let  $(\cA_1\otimes_{\cO} \cA_2,\si_1\otimes\si_2)$ be the  corresponding Azumaya $\cO$--algebra with orthogonal involution. Then $\Weak(\cA,\si )=0$.
\end{cor}

\begin{remark}
 {\rm In the setting of Corollary \ref{cor_symplectic_ext}, Remark \ref{rem_symplectic_obs} provides
 an explicit example such that $\Strong(\cA,\si ) \not =0$.
 }
\end{remark}

The case of the strong obstruction resulting from two symplectic involutions is similar to Lemma \ref{tens_triple_invol_obs}\ref{obs2}.
\begin{lem}\label{lem_symplectic_strong}
Let $(\cA_1,\si_1)$ and $(\cA_2,\si_2)$ be two Azumaya $\cO$--algebras with symplectic involutions. Let $(\cA,\si)=(\cA_1\otimes_{\cO} \cA_2,\si_1\otimes\si_2)$ be the tensor Azumaya $\cO$--algebra with orthogonal involution. According to Lemma \ref{lem_symp_1_symd}, we have $1_{\cA_1} \in \cSymd_{\cA_1,\si_1}(S)$. Consider the boundary map $\delta_1 \colon \cSymd_{\cA_1,\si_1}(S)\to \check{H}^1(S,\cSkew_{\cA_1, \si_1})$ arising from the exact sequence of $\cO$--modules $0 \to \cSkew_{\calA_1, \sigma_1} \to \calA_1 \to \cSymd_{\calA_1, \sigma_1} \to 0$. Then we have
\[
\Omega( \cA, \si)=  \delta_1(1_{\cA_1}) \cup  1_{\cA_2}
\]
where $1_{\calA_2} \in \cSym_{\cA_2,\si_2}(S)=H^0(S, \cSym_{\calA_2, \sigma_2}) $ and the cup-product arises from the tensor morphism $\cSkew_{\calA_1, \sigma_1} \times  \cSym_{\calA_2, \sigma_2} \to \cSkew_{\calA, \sigma}$.
\end{lem}

\begin{proof} Since $1_{\cA_1}\in \cSymd_{\cA_1,\sigma_1}(S)$, there is an fppf cover $\{T_i\to S\}_{i\in I}$ such that for each $i\in I$ there exists
$\ell_i\in \cA_1(T_i)$ with $\ell_i+\si_1(\ell_i)=1$. Then we will have
\[
(\ell_i\otimes 1)+(\si_1\otimes\si_2)(\ell_i\otimes 1)= 1_\cA,
\]
According to Lemma \ref{lem_strong_rep}, $\Omega(\cA,\si)$ is represented by the $1$--cocycle
\[
(\ell_i\otimes 1)|_{T_{ij}} - (\ell_j\otimes 1)|_{T_{ij}} \in \cSkew_{\cA,\si}(T_{ij})
\]
which also represents the class $\Omega( \cA_1, \si_1) \cup  1_{\cA_2}$.
\end{proof}

\section{Examples of Non-trivial Obstructions}\label{obs_examples}
The common point between the examples discussed in this section arises from a classical construction of torsors built from finite subgroups of algebraic groups. Let $k$ be a field, let $G$ be an affine $k$--algebraic group, and let $H \subset G$ be a finite $k$--subgroup (for example \'etale or constant). Given an $H$-torsor $f:Y \to X$ between $k$--schemes, we can consider its extension to $G$ defined by the contracted product $Y \wedge^H G$, which is a $G$--torsor. This kind of torsor occurs in the theory of essential dimension (Reichstein-Youssin \cite[\S 7] {RY}), in infinite dimensional Lie theory (Gille-Pianzola loop torsors \cite[\S 3]{GP}), and also in Brion's theory of homogeneous torsors over abelian varieties (\cite{Br1,Br2,Br3}, which extends that of Mukai in \cite{Mu} for vector bundles). In this last case we deal with an isogeny $f:Y \to X$ of abelian varieties such that $\ker(f)=H$, a special case of which is an isogeny of elliptic curves. More precisely, in the first example we use the multiplication by 2 map on an elliptic curve $E$ over a field $k$ of characteristic 2 and we embed its kernel $E[2]$ in $\PGL_2$, see the map \eqref{eq_embedding} later in the text.

\subsection{Non-trivial Strong Obstruction}\label{example_strong}
All concepts not defined below can be found in a standard textbook on elliptic curves, such as \cite{KM}. Let $k$ be a field of characteristic $2$ and let our base scheme be an ordinary elliptic curve $E$ over $k$. Since $E$ is ordinary, the $2$-torsion points are $E[2] \cong \bmu_2 \times_k \ZZ/2\ZZ$, where here $\ZZ/2\ZZ$ denotes the constant group scheme associated to the abstract group of two elements. We identify the Hopf $k$--algebra $H$ representing the group scheme $\bmu_2 \times_k \ZZ/2\ZZ$. It is
\begin{align*}
H &= k[\bmu_2]\otimes_k k[\ZZ/2\ZZ] \\
&= (k[x]/\langle x^2-1\rangle) \otimes_k (k\times k) \\
&\cong \left(k[x]/\langle x^2-1\rangle\right)\times\left(k[x]/\langle x^2-1\rangle\right).
\end{align*}
For a scheme $Y\in \Sch_E$, we have the standard fact that $(\bmu_2\times \ZZ/2\ZZ)(Y) = \Hom_\Rings(H,\cO(Y))$, for example by \cite[Tag 01I1]{Stacks}.

We set $E'=E$ and consider the fppf cover $\{E' \to E\}$ arising from multiplication by $2$. This is an fppf cover by \cite[2.3]{KM}. As a first step, we will identify a \v{C}ech $1$-cocycle with values in $\bmu_2\times_k \ZZ/2\ZZ$ over this cover. To do so, we identify the global sections of $E'\times_E E'$ and of $E'\times_E E'\times_E E'$, keeping in mind that $\cO(E')=k$.

Let $Y\in \Sch_k$ be any scheme. The set $E'(Y) = \Hom_k(Y,E')$ is a group, inheriting its group structure from $E'$. From the universal property of the fiber product, we have that
\[
(E'\times_E E')(Y) =\{ (a,b) \in E'(Y)\times E'(Y) \mid 2a=2b\}.
\]
Hence, for any pair $(a,b) \in (E'\times_E E')(Y)$ they differ by the element $b-a \in E'[2](Y) = (\bmu_2\times_k \ZZ/2\ZZ)(Y)$ in the kernel of the multiplication by $2$ map $E' \to E$. So, we may write
\begin{align*}
(E'\times_E E')(Y) &=\{ (a,a+s) \mid a\in E'(Y), s\in (\bmu_2\times_k \ZZ/2\ZZ)(Y)\} \\
&\cong E'(Y)\times (\bmu_2\times_k \ZZ/2\ZZ)(Y).
\end{align*}
Since this holds for all $Y$, we see that we have an isomorphism
\[
E'\times_E E' \cong E'\times_k (\bmu_2\times_k \ZZ/2\ZZ).
\]
To remove some notational clutter, we abuse both notation and the Yoneda lemma by omitting $Y$ and simply writing
\[
E'\times_E E' = \{(a,a+s) \mid a\in E', s\in \bmu_2\times_k \ZZ/2\ZZ\}
\]
and the isomorphism as
\begin{align*}
E'\times_E E' &\iso E'\times_k (\bmu_2\times_k \ZZ/2\ZZ) \\
(a,a+s) &\mapsto (a,s).
\end{align*}
We will continue using this abuse for computations throughout this section as well as in Section \ref{example_weak}. Due to the isomorphism above, we have
\[
\cO(E'\times_E E') = \cO(E'\times_k (\bmu_2\times_k \ZZ/2\ZZ)) = k\otimes_k H = H
\]
where the second equality follows from \cite[5.2.3, Cor. 2.27]{Liu}, which is about the flat base change of a scheme over a ring. Likewise, we may write
\[
E'\times_E E'\times_E E' = \{(a,a+s,a+(s+t)) \mid a\in E', s,t\in \bmu_2\times_k \ZZ/2\ZZ\}
\]
and so we have an isomorphism
\begin{align*}
E'\times_E E'\times_E E' &\iso E'\times_k (\bmu_2\times_k \ZZ/2\ZZ) \times_k (\bmu_2\times_k \ZZ/2\ZZ) \\
(a,a+s,a+(s+t)) &\mapsto (a,s,t).
\end{align*}
In turn, this means that
\[
\cO(E'\times_E E'\times_E E') = \cO(E'\times_k (\bmu_2\times_k \ZZ/2\ZZ) \times_k (\bmu_2\times_k \ZZ/2\ZZ)) = H\otimes_k H.
\]
Note that ``the $\bmu_2\times\ZZ/2\ZZ$ part" is all that occurs in the global sections. We now need to identify how the three projections, $p_{ij}\colon E'\times_E E'\times_E E' \to E'\times_E E'$ where the $i^\text{th}$ and $j^\text{th}$ factors are preserved, appear on global sections. We have
\begin{align*}
p_{12} \colon E'\times_E E'\times_E E' &\to E'\times_E E' \\
(a,a+s,a+(s+t)) &\mapsto (a,a+s),
\end{align*}
which alternatively appears as
\begin{align*}
p_{12} \colon E'\times_k (\bmu_2\times_k \ZZ/2\ZZ) \times_k (\bmu_2\times_k \ZZ/2\ZZ) &\to E'\times_k (\bmu_2\times_k \ZZ/2\ZZ) \\
(a,s,t) &\mapsto (a,s),
\end{align*}
after using the above isomorphisms. On the $\bmu_2\times \ZZ/2\ZZ$ part this is simply the first projection $(s,t)\mapsto s$, and hence on global sections we obtain the map
\[
\widetilde{p_{12}} = \Id\otimes 1 \colon H \to H\otimes_k H.
\]
When we consider $p_{23}$, it appears as $(a,a+s,a+(s+t)) \mapsto (a+s,a+(s+t))$, which alternatively becomes $(a,s,t)\mapsto (a+s,t)$. Thus, the $\bmu_2\times \ZZ/2\ZZ$ part is simply the second projection $(s,t)\mapsto t$ and we get that
\[
\widetilde{p_{23}} = 1\otimes\Id \colon H\to H\otimes_k H.
\]
Finally, applying the same procedure to $p_{13}$, we start with $(a,a+s,a+(s+t)) \mapsto (a,a+(s+t))$, which we rewrite as $(a,s,t) \mapsto (a,s+t)$, and then extract the $\bmu_2\times \ZZ/2\ZZ$ part, which is the addition map $(s,t)\mapsto s+t$. Therefore, we see that we obtain the comultiplication
\[
\widetilde{p_{13}} = \Delta \colon H \to H\otimes_k H
\]
of the Hopf algebra $H$.

Now, a \v{C}ech $1$-cocycle for $\bmu_2\times \ZZ/2\ZZ$ with respect to the cover $\{E'\to E\}$ consists of an element $y \in (\bmu_2\times \ZZ/2\ZZ)(E'\times_E E')$ such that
\[
y|_{p_{12}}\cdot y|_{p_{23}} = y|_{p_{13}},
\]
where $y|_{p_{ij}} = (\bmu_2\times \ZZ/2\ZZ)(p_{ij})(y)$ is the restriction along $p_{ij}$. Since the group $\bmu_2\times \ZZ/2\ZZ$ is represented by $\Spec(H)$, we have that
\begin{align*}
(\bmu_2\times \ZZ/2\ZZ)(E'\times_E E') &= \Hom_\Rings(H,\cO(E'\times_E E')) \\
&= \Hom_\Rings(H,H).
\end{align*}
We claim that $\Id \in \Hom_\Rings(H,H)$ is a $1$-cocycle. Indeed, for each projection we have that
\[
\Id|_{p_{ij}} = \widetilde{p_{ij}}\circ \Id = \widetilde{p_{ij}}
\]
as an element in $\Hom_\Rings(H,H\otimes_k H) = (\bmu_2\times \ZZ/2\ZZ)(E'\times_E E' \times_E E')$. Therefore, using the fact that the group structure of $\bmu_2\times \ZZ/2\ZZ$ comes from the Hopf algebra structure of $H$, the product $\widetilde{p_{12}}\cdot \widetilde{p_{23}}$ is the composition
\[
H \xrightarrow{\Delta} H\otimes_k H \xrightarrow{\widetilde{p_{12}}\otimes \widetilde{p_{23}}} (H\otimes_k H)\otimes_k (H\otimes_k H) \xrightarrow{\textrm{mult}} H\otimes_k H,
\]
which is easily checked to be
\[
\widetilde{p_{12}}\cdot \widetilde{p_{23}} = \Delta = \widetilde{p_{13}}
\]
and so $\Id \in \Hom_\Rings(H,H) = (\bmu_2\times \ZZ/2\ZZ)(E'\times_E E')$ is a $1$-cocycle.

Now, we will transport this cocycle to $\PGL_2$. We have an embedding of group schemes $i\co \bmu_2\times_k \ZZ/2\ZZ \hookrightarrow \PGL_2$ defined over $Y\in \Sch_E$ by
\begin{equation}\label{eq_embedding}
\varphi \mapsto \Inn\left(\begin{bmatrix} \varphi(0,1) & \varphi(1,0) \\ \varphi(x,0) & \varphi(0,x) \end{bmatrix}\right)
\end{equation}
for $\varphi \in \Hom_\Rings(H,\cO(Y))$. Intuitively, this is the map
\[
\varepsilon \mapsto \Inn\left(\begin{bmatrix} 1 & 0 \\ 0 & \varepsilon \end{bmatrix}\right) \text{ and } 1 \mapsto \Inn\left(\begin{bmatrix} 0 & 1 \\ 1 & 0 \end{bmatrix}\right)
\]
for $\varepsilon \in \bmu_2$ and $1\in \ZZ/2\ZZ$. The image $i(\Id) \in \PGL_2(E'\times_E E')$ is a $1$-cocycle since $\Id \in (\bmu_2\times_k \ZZ/2\ZZ)(E'\times_E E')$ is a $1$--cocycle. We denote $i(\Id)=\phi$ and of course have that
\[
\phi = \Inn\left(\begin{bmatrix} (0,1) & (1,0) \\ (x,0) & (0,x) \end{bmatrix}\right).
\]
The cocycle condition means that this $\phi$ is a gluing datum as in \cite[Tag 04TP]{Stacks}. Therefore, we may twist $\Mat_2(\cO)$ by gluing two copies of $\Mat_2(\cO)|_{E'}$ using $\phi$ as in \cite[Tag 04TR]{Stacks}, to obtain a quaternion $\cO$--algebra $\cQ$. For $T\in \Sch_E$, we have
\[
\cQ(T) = \{B\in \Mat_2(\cO(T\times_E E')) \mid \phi(B|_{T\times_E E'\times_E \times E'})= B|_{T\times_E E'\times_E \times E'}\}.
\]
Since $k$ has characteristic $2$, the canonical quaternion involution on $\Mat_2(\cO)$ is the split orthogonal involution in degree $2$ of Example \ref{split_example}\ref{split_example_b},
\[
\eta_0 \colon \begin{bmatrix} a & b \\ c & d \end{bmatrix} \mapsto \begin{bmatrix} d & b \\ c & a \end{bmatrix}.
\]
The canonical involution on $\cQ$, denoted by $\theta$, is the descent of $\eta_0$ and hence $(\cQ,\theta)$ is a quaternion algebra with orthogonal involution.
\sm

Let $n$ be a positive integer and now work over the abelian $k$--variety $S=E^n$.
We define the Azumaya algebra
\[
\calA = p_1^*(\calQ) \otimes_{\calO} \dots  \otimes_{\calO} p_n^*(\calQ)
\]
where $p_i \co E^n \to E$ is the $i^\text{th}$ projection and $p_i^*$ the pullback of quaternion algebras. It comes with the tensor product involution $\si=p_1^*(\theta) \otimes \cdots \otimes p_n^*(\theta)$, which is orthogonal. Alternatively, $\cA$ may be viewed as the following twist of $\Mat_{2^n}(\cO)$. Set $S'=S$ and view it as a scheme over $S$ with respect to the multiplication by $2$ map where each factor is multiplied by $2$. This makes $S'\to S$ a $(\bmu_2\times_k \ZZ/2\ZZ)^n$--torsor and $\{S'\to S\}$ an fppf cover. Since we are looking for a cocycle in $\PGL_{2^n}(S'\times_S S')$ we note that $\cO(S)=k$, and since $S'\times_S S' \cong (E'\times_E E')^n$ we have
\[
\cO(S'\times_S S') = \left( (k[x]/\langle x^2-1 \rangle)^2\right)^{\otimes n}.
\]
where the tensor product is over $k$. Using the decomposition
\[
\Mat_{2^n}(S'\times_S S') \cong \Mat_2((k[x]/\langle x^2-1\rangle)^2) \otimes_k \ldots\otimes_k \Mat_2((k[x]/\langle x^2-1\rangle)^2)
\]
we have $\phi\otimes\ldots\otimes\phi \in \PGL_{2^n}(S'\times_S S')$, and this is our desired cocycle. In this view, the tensor product involution on $\cA$ is the descent of the involution $\eta'=\eta_0\otimes\ldots\otimes \eta_0$ on $\Mat_{2^n}(\cO(S'))\cong (\Mat_2(\cO(S')))^{\otimes n}$, where $\cO(S')=k$.

\begin{lem} \label{lem_mukai}
Consider the Azumaya algebra with involution $(\cA,\si)$ defined above.
\begin{enumerate}[label={\rm (\roman*)}]

\item \label{lem_mukai1} $(\calA, \sigma)$ can be extended to a quadratic triple. In particular, $(\cA,\si)$ is locally quadratic with trivial weak obstruction.

\item \label{lem_mukai2}  $\calA(S)\cong k$.

\item \label{lem_mukai3}  The strong obstruction $\Strong(\calA, \sigma)$ is non-trivial.
\end{enumerate}

\end{lem}

\begin{proof}
Throughout the proof we will use the following computation. For a matrix of the form
\[
\begin{bmatrix}(b_1,b_1) & (b_2,b_2) \\ (b_3,b_3) & (b_4,b_4) \end{bmatrix} \in \Mat_2((k[x]/\langle x^2-1\rangle)^2)
\]
with $b_i \in k$, we have
\[
\phi\left(\begin{bmatrix}(b_1,b_1) & (b_2,b_2) \\ (b_3,b_3) & (b_4,b_4) \end{bmatrix}\right) = \begin{bmatrix} (b_4,b_1) & (b_3x,b_2x) \\ (b_2x,b_3x) & (b_1,b_4)\end{bmatrix}.
\]

\noindent \ref{lem_mukai1}: We begin by considering the linear form
\begin{align*}
f' \colon \cSym_{\Mat_{2^n}(\cO),\si_{2^n}} &\to \cO \\
B &\mapsto \Trd_{\Mat_{2^n}(\cO)}\left( \left(\begin{bmatrix} 1 & 0 \\ 0 & 0 \end{bmatrix}\otimes I_{2^{n-1}}\right)\cdot B \right)
\end{align*}
where $I_{2^{n-1}}$ denotes the $2^{n-1} \times 2^{n-1}$ identity matrix. This is an instance of the construction of Example \ref{quafoex}\ref{quafoex_b}, and so $(\Mat_{2^n}(\cO),\si_{2^n},f')$ is a quadratic triple. The form $f'|_{S'}$ will descend to a suitable linear form on $\cSym_{\cA,\si}$ if $f'|_{S'\times_S S'}\circ (\phi\otimes\ldots\otimes\phi) = f'|_{S'\times_S S'}$. We verify this with the following computation, which uses the fact that $\phi^2=\Id$.
\begin{align*}
&\Trd_{\Mat_{2^n}(\cO)}\left(\left(\begin{bmatrix} (1,1) & 0 \\ 0 & 0 \end{bmatrix}\otimes I_{2^{n-1}}\right)\cdot (\phi\otimes\ldots\otimes\phi)(\und)\right) \\
= &\Trd_{\Mat_{2^n}(\cO)}\left((\phi\otimes\ldots\otimes\phi)\left(\begin{bmatrix} (1,1) & 0 \\ 0 & 0 \end{bmatrix}\otimes I_{2^{n-1}}\right)\cdot\und\right)\\
= &\Trd_{\Mat_{2^n}(\cO)}\left(\left(\begin{bmatrix} (0,1) & 0 \\ 0 & (1,0) \end{bmatrix}\otimes I_{2^{n-1}}\right)\cdot\und\right).
\end{align*}
Then, because
\[
\left(\begin{bmatrix} (1,1) & 0 \\ 0 & 0 \end{bmatrix}\otimes I_{2^{n-1}}\right) - \left(\begin{bmatrix} (0,1) & 0 \\ 0 & (1,0) \end{bmatrix}\otimes I_{2^{n-1}}\right) = \begin{bmatrix} (1,0) & 0 \\ 0 & (1,0) \end{bmatrix}\otimes I_{2^{n-1}}
\]
which is an element of $\cAlt_{\Mat_{2^n}(\cO),\eta'}(S'\times_S S')$, we know by Example \ref{quafoex}\ref{quafoex_b} that
\[
f'|_{S'\times_S S'}\circ (\phi\otimes\ldots\otimes\phi)=f'|_{S'\times_S S'}.
\]
Therefore $f'|_{S'}$ descends, and there exists a linear form $f\colon \cSym_{\cA,\si}\to\cO$ such that $(\cA,\si,f)$ is a quadratic triple.
\sm

\noindent \ref{lem_mukai2}: By construction we have that
\[
\cA(S) = \{B\in \Mat_{2^n}(\cO(S')) \mid (\phi\otimes\ldots\otimes\phi)(B|_{S'\times_S S'})=B|_{S'\times_S S'}\}
\]
where $\cO(S')=k$ and so we work with $B\in \Mat_{2^n}(k)$ below. Given $B\in \cA(S)$, we may write $B$ uniquely as
\[
B = \begin{bmatrix} 1 & 0 \\ 0 & 0 \end{bmatrix}\otimes B_1 + \begin{bmatrix} 0 & 1 \\ 0 & 0 \end{bmatrix}\otimes B_2 + \begin{bmatrix} 0 & 0 \\ 1 & 0 \end{bmatrix}\otimes B_3 + \begin{bmatrix} 0 & 0 \\ 0 & 1 \end{bmatrix}\otimes B_4
\]
with $B_i \in \Mat_{2^{n-1}}(k)$, and then setting $B|_{S'\times_S S'}=\overline{B}$ we have
\begin{align*}
\overline{B} = &\begin{bmatrix} (1,1) & 0 \\ 0 & 0 \end{bmatrix}\otimes \overline{B_1} + \begin{bmatrix} 0 & (1,1) \\ 0 & 0 \end{bmatrix}\otimes \overline{B_2} \\
+&\begin{bmatrix} 0 & 0 \\ (1,1) & 0 \end{bmatrix}\otimes \overline{B_3} + \begin{bmatrix} 0 & 0 \\ 0 & (1,1) \end{bmatrix}\otimes \overline{B_4}.
\end{align*}
Now we apply $\phi\otimes\ldots\otimes\phi$, setting $\overline{\phi}=\phi^{\otimes n-1}$.
\begin{align*}
(\phi\otimes\ldots\otimes\phi)(\overline{B}) = &\begin{bmatrix} (0,1) & 0 \\ 0 & (1,0) \end{bmatrix}\otimes \overline{\phi}(\overline{B_1}) + \begin{bmatrix} 0 & (0,x) \\ (x,0) & 0 \end{bmatrix}\otimes \overline{\phi}(\overline{B_2}) \\
+&\begin{bmatrix} 0 & (x,0) \\ (0,x) & 0 \end{bmatrix}\otimes \overline{\phi}(\overline{B_3}) + \begin{bmatrix} (1,0) & 0 \\ 0 & (0,1) \end{bmatrix}\otimes \overline{\phi}(\overline{B_4}).
\end{align*}
Since this is equal to $\overline{B}$, linear independence then requires that
\begin{align*}
&\overline{B_1} = \overline{\phi}(\overline{B_4}) = \overline{B_4} = \overline{\phi}(\overline{B_1}), \text{ and} \\
&\overline{B_2} = \overline{B_3} = \overline{\phi}(\overline{B_2}) = \overline{\phi}(\overline{B_3}) = 0
\end{align*}
and so we can conclude that $\overline{B} = I_2\otimes \overline{B'}$ for some $B'\in \Mat_{2^{n-1}}(k)$ such that $(\phi\otimes\ldots\otimes\phi)(\overline{B'})=\overline{B'}$, now with only $n-1$ tensor factors. Hence, by induction we need only address the case of $2\times 2$ matrices. There we have that
\[
\phi(\overline{B})=\begin{bmatrix} (b_4,b_1) & (b_3x,b_2x) \\ (b_2x,b_3x) & (b_1,b_4)\end{bmatrix} = \begin{bmatrix}(b_1,b_1) & (b_2,b_2) \\ (b_3,b_3) & (b_4,b_4) \end{bmatrix} = \overline{B}.
\]
Since $b_i\in k$, this can only happen when $B=aI_2$ for some $a\in k$. Therefore, overall
\[
\cA(S) = \{I_{2^{n-1}}\otimes aI_2 \mid a\in k\} \cong k.
\]

\noindent \ref{lem_mukai3}: Part \ref{lem_mukai2} above shows that $\cA(S)=\cO(S)\cong k$, and since $k$ is characteristic $2$ we also have that $2\cO(S)=0$. Therefore applying Lemma \ref{lem_nontrivial_strong} yields that $\Strong(\cA,\sigma)\neq 0$ as claimed.
\end{proof}

\begin{remark}\label{rem_symplectic_obs}
Note that for $n=2$, $(\calA, \si)$ is a tensor product of two Azumaya algebras with symplectic involutions, and so Corollary \ref{cor_symplectic_ext} provides another proof that the weak obstruction is zero in Lemma \ref{lem_mukai}\ref{lem_mukai3}. However, the important point is that the strong obstruction does not vanish.
\end{remark}

\subsection{Non-trivial Weak Obstruction}\label{example_weak}
We continue working over an algebraically closed base field $k$ of characteristic $2$. According to a result by Serre \cite[prop.~15]{Se},
for each finite group $\Gamma$ there exists a Galois $\Gamma$--cover $Y \to S$ as in \cite[Tag 03SF]{Stacks} such that $S$ and $Y$ are connected smooth projective $k$--varieties. We use Serre's results with the group $\Gamma= \PGL_2(\FF_4)$ to obtain a $\Gamma$--cover $\pi \colon Y \to S$. We take $S$ to be our base scheme. We denote by $\Gamma_k$ the constant group scheme associated to $\Gamma$. It is affine and represented by $k^{|\Gamma|}$ with componentwise multiplication. We write this algebra as
\[
k^{|\Gamma|} = \{(c_g)_{g\in \Gamma} \mid c_g \in k\}.
\]
The map $Y\to S$ is then a $\Gamma_k$--torsor and $\{Y\to S\}$ is an fppf cover. Since $\Gamma_k$ embeds in $\PGL_2$, which we view as a group scheme over $k$, we can define the $\PGL_2$--torsor $P=Y \wedge^{\Gamma_k} \PGL_2$ over $S$. The twist of $\Mat_2(\calO)$ by $P$ is a quaternion $\calO$--algebra $\calQ$, which of course is not the same $\cQ$ as in section \ref{example_strong}. Here as well we may describe $\cQ$ explicitly using cocycles. As before, we first describe the global sections of $Y\times_S Y$ and $Y\times_S Y \times_S Y$. We know that $\cO(Y)\cong k$, and since $Y\to S$ is a Galois extension with Galois group $\Gamma$, we have isomorphisms
\begin{align*}
Y\times_S Y &\iso Y\times_k \Gamma_k \\
(y,yg) &\mapsto (y,g)
\end{align*}
and
\begin{align*}
Y\times_S Y \times_S Y &\iso Y\times_k \Gamma_k \times_k \Gamma_k \\
(y,yg,yh) &\mapsto (y,g,h)
\end{align*}
which we use to identify
\begin{align*}
\cO(Y \times_S Y) &\cong k^{|\Gamma|} \text{, and} \\
\cO(Y\times_S Y \times_S Y) &\cong k^{|\Gamma|}\otimes_k k^{|\Gamma|} \cong k^{|\Gamma\times\Gamma|}
\end{align*}
where we write $(d_{g,h})_{g,h\in \Gamma}$ for an element in $k^{|\Gamma\times\Gamma|}$. As in Example \ref{example_strong}, the Hopf algebra structure on $k^{|\Gamma|}$ representing $\Gamma_k$ plays a roll in describing the three restriction maps $\widetilde{p_{ij}} \colon \cO(Y\times_S Y) \to \cO(Y\times_S Y \times_S Y)$ as maps $k^{|\Gamma|} \to k^{|\Gamma\times\Gamma|}$. The canonical projection $p_{12} \colon Y\times_S Y \times_S Y \to Y\times_S Y$ can be written as
\[
(y,yg,yh) \mapsto (y,g)
\]
which corresponds to the first projection $\Gamma_k \times_k \Gamma_k \to \Gamma_k$ sending $(g,h) \mapsto g$. Therefore, on global sections, this is represented by the algebra map
\begin{align*}
\widetilde{p_{12}} = \Id\otimes 1 \colon k^{|\Gamma|} &\to k^{|\Gamma|}\otimes_k k^{|\Gamma|} \cong k^{|\Gamma\times\Gamma|} \\
(c_g)_{g\in \Gamma} &\mapsto (c_g)_{g\in \Gamma}\otimes (1)_{h\in \Gamma} \cong (c_g)_{g,h\in \Gamma}
\end{align*}
since $1=(1)_{h\in H}$ in the second factor $k^{|\Gamma|}$. Here $(c_g)_{g,h\in \Gamma}$ is the element $(d_{g,h})_{g,h\in \Gamma} \in k^{|\Gamma\times\Gamma|}$ with $d_{g,h}=c_g$. Similarly, the projection $p_{13}$ appears as $(y,yg,yh)\mapsto (y,yh)$ which corresponds to the second projection $\Gamma_k \times_k \Gamma_k \to \Gamma_k$, and thus
\begin{align*}
\widetilde{p_{13}} = \Id\otimes 1 \colon k^{|\Gamma|} &\to k^{|\Gamma|}\otimes_k k^{|\Gamma|} \cong k^{|\Gamma\times\Gamma|} \\
(c_g)_{g\in \Gamma} &\mapsto (1)_{g\in \Gamma}\otimes (c_h)_{h\in \Gamma} \cong (c_h)_{g,h\in \Gamma}.
\end{align*}
The remaining projection, $p_{23}$ sends $(y,yg,yh)\mapsto (yg,yh)$, which after applying the isomorphisms above becomes
\[
(y,g,h) \mapsto (yg,g^{-1}h)
\]
and thus it corresponds to the map $\Gamma_k \times_k \Gamma_k \to \Gamma_k$ which sends $(g,h)\mapsto g^{-1}h$. This is represented by the algebra map
\begin{align*}
\widetilde{p_{23}} = (i\otimes \Id)\circ \Delta \colon k^{|\Gamma|} &\to k^{|\Gamma\times\Gamma|} \\
(c_g)_{g\in \Gamma} &\mapsto  (c_{g^{-1}h})_{g,h\in \Gamma}.
\end{align*}
where $i \colon (c_g)_{g\in \Gamma} \mapsto (c_{g^{-1}})_{g\in \Gamma}$ is the antipode and
\[
\Delta \colon (c_g)_{g\in \Gamma} \mapsto (c_{gh})_{g,h \in \Gamma}
\]
is the comultiplication of the Hopf algebra $k^{|\Gamma|}$ representing $\Gamma_k$.

Now, we search for a $1$--cocycle in $\PGL_2(Y\times_S Y)$. We have that
\[
\PGL_2(Y\times_S Y)=\PGL_2(k^{|\Gamma|}) \cong \PGL_2(k)^{|\Gamma|}
\]
and likewise $\PGL_2(Y\times_S Y\times_S Y) \cong \PGL_2(k)^{|\Gamma\times\Gamma|}$. We reuse the notation $\widetilde{p_{ij}} \colon \PGL_2(Y\times_S Y) \to \PGL_2(Y\times_S Y\times_S Y)$ for the restriction maps since they appear similar to the ones above, namely
\begin{align*}
\widetilde{p_{12}} \colon (\varphi_g)_{g\in \Gamma} &\mapsto (\varphi_g)_{g,h\in \Gamma}\\
\widetilde{p_{13}} \colon (\varphi_g)_{g\in \Gamma} &\mapsto (\varphi_h)_{g,h\in \Gamma} \\
\widetilde{p_{23}} \colon (\varphi_g)_{g\in \Gamma} &\mapsto (\varphi_{g^{-1}h})_{g,h\in \Gamma}
\end{align*}
where now each $\varphi_g \in \PGL_2(k)$. Identifying $\Gamma_k$ with its embedding in $\PGL_2$, we have the element
\[
(g)_{g\in \Gamma} \in \PGL_2(Y\times_S Y)
\]
which we claim is a $1$--cocycle. Indeed,
\begin{align*}
\widetilde{p_{12}}((g)_{g\in \Gamma})\cdot \widetilde{p_{23}}((g)_{g\in \Gamma}) &= (g)_{g,h\in \Gamma} \cdot (g^{-1}h)_{g,h\in \Gamma} \\
&= (gg^{-1}h)_{g,h\in \Gamma} \\
&= (h)_{g,h\in \Gamma} \\
&= \widetilde{p_{13}}((g)_{g\in \Gamma}).
\end{align*}
Thus, again by \cite[Tag 04TR]{Stacks} the quaternion  $\cO$--algebra $\cQ$ is described over $T\in \Sch_S$ by
\[
\cQ(T) = \{ B \in \Mat_2(\cO(T\times_S Y)) \mid g(B)=B, \; \forall\, g\in \Gamma\}.
\]
Or, more concisely, $\cQ(T) = \Mat_2(\cO(T\times_S Y))^\Gamma$ are the fixed points. The canonical involution $\eta_0$ on $\Mat_2(\cO)$ of Example \ref{split_example}\ref{split_example_b} is orthogonal and descends to an orthogonal involution $\theta$ on $\cQ$, which is the canonical involution on $\cQ$.

\begin{lem} \label{lem_serre} $(\calQ, \theta)$ is locally quadratic, but cannot be extended to a quadratic triple. Thus, $\Weak(\calQ, \theta) \not = 0 \in
\check{H}^1( S, \cSkew_{\calA, \sigma}/ \cAlt_{\calA, \sigma})$, i.e., the weak obstruction is non-trivial.
\end{lem}

\begin{proof}
Since $(\cQ,\theta)$ is a twisted form of $\Mat_2(\cO)$ which splits over $\{Y\to S\}$, we have that
\[
(\calQ, \theta)|_Y \cong (\Mat_2(\cO),\eta_0)|_Y.
\]
Since we have the split quadratic triple $(\Mat_2(\cO), \eta_0, f_0)|_Y$ of Example \ref{split_example}\ref{split_example_b} and $\{Y\to S\}$ is an fppf cover, we obtain that $(\calQ, \theta)$ is locally quadratic by Lemma \ref{lem_smart_equiv}\ref{lem_smart_equiv_iv}.

Now, assume that we can extend $(\calQ, \theta)$ to a quadratic triple $(\calQ, \theta,f)$ over $S$. The linear map $f$ then fits into the following diagram where the rows are exact sequences coming from the sheaf equalizer diagrams for $\cSym_{\cQ,\theta}$ and $\cO$ respectively and which commutes since $f$ is a morphism of sheaves.
\[
\begin{tikzcd}
\cSym_{\cQ,\theta}(S) \arrow{r} \arrow{d}{f(S)} & \cSym_{\Mat_2(\cO),\eta_0}(Y)^2 \arrow{r}{\pi_1} \arrow{d}{f(Y)\times f(Y)} & \cSym_{\Mat_2(\cO),\eta_0}(Y\times_S Y) \arrow{d}{f(Y\times_S Y)} \\
k \arrow{r} & k\times k \arrow{r}{\pi_2} & k^{|\Gamma|}
\end{tikzcd}
\]
where $\pi_1(B_1,B_2) = (g(B_1)-B_2)_{g\in \Gamma}$ and $\pi_2(c_1,c_2)=(c_1-c_2)_{g\in \Gamma}$. Commutativity of the diagram then enforces that
\[
f(Y\times_S Y)\left((g(B_1)-B_2)_{g\in \Gamma}\right) = (f(Y)(B_1)-f(Y)(B_2))_{g\in \Gamma}
\]
which is equivalent to
\[
(f(Y)(g(B_1)))_{g\in \Gamma}=(f(Y)(B_1))_{g\in \Gamma}
\]
and hence $f(Y)$ must be $\Gamma$--equivariant. We now argue that this means $f(Y)$ must be zero. Certainly, $f(Y)$ is of the form
\[
f\left(\begin{bmatrix} a & b \\ c & a \end{bmatrix}\right) = au + bv + cw
\]
for some $u,v,w\in k=\cO(Y)$. By $\Gamma$--equivariance we obtain
\[
w = f\left(\begin{bmatrix} 0 & 0 \\ 1 & 0 \end{bmatrix}\right)=f\left(\begin{bmatrix} 0 & 1 \\ 1 & 0 \end{bmatrix}\begin{bmatrix} 0 & 0 \\ 1 & 0 \end{bmatrix}\begin{bmatrix} 0 & 1 \\ 1 & 0 \end{bmatrix}\right) = f\left(\begin{bmatrix} 0 & 1 \\ 0 & 0 \end{bmatrix}\right) = v,
\]
as well as
\[
w = f\left(\begin{bmatrix} 0 & 0 \\ 1 & 0 \end{bmatrix}\right)=f\left(\begin{bmatrix} 1 & 1 \\ 0 & 1 \end{bmatrix}\begin{bmatrix} 0 & 0 \\ 1 & 0 \end{bmatrix}\begin{bmatrix} 1 & 1 \\ 0 & 1 \end{bmatrix}\right) = f\left(\begin{bmatrix} 1 & 1 \\ 1 & 1 \end{bmatrix}\right) = u+v+w
\]
which implies that $u+v=0$, i.e., $u=v$. Therefore we must have $u=v=w$. Finally, consider $0,1\neq \lambda \in \FF_4$. We also have
\[
w = f\left(\begin{bmatrix} 0 & 0 \\ 1 & 0 \end{bmatrix}\right)=f\left(\begin{bmatrix} \lambda^{-1} & 0 \\ 0 & 1 \end{bmatrix}\begin{bmatrix} 0 & 0 \\ 1 & 0 \end{bmatrix}\begin{bmatrix} \lambda & 0 \\ 0 & 1 \end{bmatrix}\right) = f\left(\begin{bmatrix} 0 & 0 \\ \lambda & 0 \end{bmatrix}\right) = \lambda w
\]
and so $u=v=w=0$, meaning $f(Y)\equiv 0$. However, by Corollary \ref{cor_f_of_one}, since $\cQ$ is a degree $2$ algebra we have that $f(1_\cQ)=\frac{2}{2} \neq 0$ (since $S$ is not the empty scheme). This is a contradiction, therefore no such $f$ extending $(\cQ,\theta)$ can exist.
\end{proof}

Unlike in section \ref{example_strong}, where we first constructed a quaternion algebra with non-trivial strong obstruction and then used tensor products to construct $(\cA,\sigma)$ of degree $2^n$ which also has non-trivial strong obstruction, we cannot apply the same tensor product construction to the example of \ref{example_weak} to produce examples of larger degree which have non-trivial weak obstruction. The following proposition makes this precise.

\begin{prop}\label{prop_serre} Let $n \geq 1$ and
define on the $k$--variety  $S^n$ the following Azumaya algebra with orthogonal
involution of degree $2^n$
\[
(\calA_n, \sigma_n) = p_1^*(\calQ, \theta) \otimes_{\cO|_{S^n}} \dots   \otimes_{\cO|_{S^n}} p_n^*(\calQ, \theta)
\]
where $p_i\colon S^n \to S$ is the $i^\textrm{th}$ projection and $p_i^*$ is the pullback of $\cO$--algebras. Then $(\calA_n, \sigma_n)$ is locally quadratic for any $n\geq 1$, and $(\cA_n,\si_n)$ can be extended to a quadratic triple if and only if $n \geq 2$. Equivalently, $\Weak(\cA_n,\si_n)\neq 0$ if and only if $n=1$.
\end{prop}
\begin{proof}
Lemma \ref{lem_serre} is the case $n=1$ and shows in particular that $(\calQ, \theta) $ is locally quadratic. Since $k$ is characteristic $2$, this means by Lemma \ref{lem_symp_1_symd} that $\theta$ is also symplectic. Therefore $p_i^*(\cQ,\theta)$ have symplectic involutions and so when $n=2$, Proposition \ref{tens_symplectic} shows that $p_1^*(\cQ,\theta)\otimes_{\cO} p_2^*(\cQ,\theta)$ can be extended to a quadratic triple. Then Proposition \ref{tens_triple_invol} handles the cases of $n\geq 3$.
\end{proof}

\subsection{Non-trivial Obstructions Which Become Trivial in Characteristic $2$}\label{sec_weird_example}
Here we provide an example of an Azumaya $\cO$--algebra with locally quadratic involution such that both the strong and weak obstructions are non-trivial, but where base changing to $S' = S\times_{\Spec(\ZZ)}\Spec(\FF_2)$, i.e. reducing modulo $2$, causes both obstructions to be trivial.

To start, we build the base scheme by gluing the two affine schemes $U_1 = \Spec(\ZZ[x]/\langle 4x \rangle)$ and $U_2 = \Spec(\ZZ[\frac{1}{x}]/\langle \frac{4}{x}\rangle)$. Notice that the localization $\left(\ZZ[x]/\langle 4x \rangle\right)_x \cong \ZZ/4\ZZ[x,\frac{1}{x}]$, and since $U_1\cong U_2$, this is also the localization of $\ZZ[\frac{1}{x}]/\langle \frac{4}{x}\rangle$ at $\frac{1}{x}$. Hence we set $U_{12}=U_{21}=\Spec(\ZZ/4\ZZ[x,\frac{1}{x}])$ and then glue $U_1$ and $U_2$ together along $U_{12}$ to produce the base scheme $S$. We note that $S$ is not flat as a $\Spec(\ZZ)$--scheme, since, for example, the map $U_{12} \to \Spec(\ZZ)$ corresponds to the map $\ZZ \to \ZZ/4\ZZ[x,\frac{1}{x}]$ which is not a flat map of rings. The global sections of $S$ can be computed as the equalizer in the diagram
\[
\begin{tikzcd}
\cO(S) \arrow[hookrightarrow]{r} & \ZZ[x]/\langle 4x \rangle \times \ZZ[\frac{1}{x}]/\langle \frac{4}{x}\rangle \arrow[yshift = 0.5ex]{r} \arrow[yshift = -0.5ex]{r} & \ZZ/4\ZZ[x,\frac{1}{x}]
\end{tikzcd}
\]
and this yields that $\cO(S) \cong \{(a,b) \in \ZZ^2 \mid a\equiv b \pmod{4}\}$. Interestingly, $2=(2,2)\in \cO(S)$ is not a zero divisor even though $2$ is a zero divisor on both $U_1$ and $U_2$.

Next, we define a locally free $\cO$--module of constant rank $2$ by twisting the free module $\cO^2$. Considering the natural (Zariski) cover $\{U_i \to S\}_{i=1,2}$, we are in the setting of gluing sheaves on a topological space as in \cite[Tag 00AK]{Stacks}, and so it is sufficient to choose any automorphism
\[
\alpha_{12} \colon (\ZZ/4\ZZ[x,\tfrac{1}{x}])^2 \iso (\ZZ/4\ZZ[x,\tfrac{1}{x}])^2
\]
in order to specify gluing data, i.e., a cocycle, since then $\alpha_{21}=\alpha_{12}^{-1}$ and $\alpha_{ii}=\Id$ are required. We choose the automorphism $\alpha_{12} = \begin{bmatrix} 1 & 2 \\ 0 & 1 \end{bmatrix}$ and denote the twisted module by $\cE$.
\lv{%
The global sections of $\cE$ can be computed as the kernel of the map
\begin{align*}
(\ZZ[x]/\langle 4x \rangle)^2 \times (\ZZ[\tfrac{1}{x}]/\langle \tfrac{4}{x} \rangle)^2 &\to (\ZZ/4\ZZ[x,\tfrac{1}{x}])^2 \\
\left(\begin{bmatrix} a+f \\ b+g \end{bmatrix},\begin{bmatrix} f'+a' \\ g'+b'\end{bmatrix}\right) &\mapsto \begin{bmatrix} \overline{(a+2b)}+(f+2g) \\ \overline{b}+g \end{bmatrix} - \begin{bmatrix} f'+\overline{a'} \\ g'+\overline{b'}\end{bmatrix}
\end{align*}
where $a,b,a',b' \in \ZZ$, $f,g \in \ZZ/4\ZZ[x]$ with no constant term, $f',g' \in \ZZ/4\ZZ[\frac{1}{x}]$ also with no constant term, and overline denotes reduction modulo $4$. Therefore,
\[
\cE(S) = \left\{ \left(\begin{bmatrix} a \\ b \end{bmatrix},\begin{bmatrix} a' \\ b' \end{bmatrix}\right) \mid \begin{array}{l} a,b,a',b'\in \ZZ \\ a'\equiv a+2b \pmod{4} \\ b' \equiv b \pmod{4} \end{array} \right\}.
\]
}
The endomorphism algebra $\cQ = \cEnd_{\cO}(\cE)$ is a neutral Azumaya $\cO$--algebra which is the twist of $\Mat_2(\cO)$ over the same cover and along the inner automorphism
\[
\Inn(\alpha_{12}) \colon \Mat_2(\ZZ/4\ZZ[x,\tfrac{1}{x}]) \iso \Mat_2(\ZZ/4\ZZ[x,\tfrac{1}{x}]).
\]
\lv{%
We may compute the global sections $\cQ(S)$ as the kernel of the map
\begin{align*}
\Mat_2(\ZZ[x]/\langle 4x \rangle) \times \Mat_2(\ZZ[\tfrac{1}{x}]/\langle \tfrac{4}{x} \rangle) &\to \Mat_2(\ZZ/4\ZZ[x,\tfrac{1}{x}])\\
(B_1,B_2)&\mapsto \Inn(\alpha_{12})(\overline{B_1})-\overline{B_2}
\end{align*}
where overline indicates reducing modulo $4$. In detail, this is the following map
\begin{align*}
&\left(\begin{bmatrix} a+f & b+g \\ c+h & d+k \end{bmatrix}, \begin{bmatrix} f'+a' &  g'+b' \\ h'+c' & k'+d' \end{bmatrix}\right)\\
\mapsto &\begin{bmatrix} \overline{(a+2c)}+(f+2h) & \overline{(2a+b+2d)}+(2f+g+2k) \\ \overline{c+h} & \overline{(2c+d)}+(2h+k)\end{bmatrix} - \begin{bmatrix} f'+\overline{a'} &  g'+\overline{b'} \\ h'+\overline{c'} & k'+\overline{d'} \end{bmatrix}
\end{align*}
where $a,b,c,d,a',b',c',d' \in \ZZ$, the terms $f,g,h,k \in \ZZ/4\ZZ[x]$ with no constant term, and $f',g',h',k' \in \ZZ/4\ZZ[\frac{1}{x}]$ also with no constant terms. Hence, in the kernel all polynomial terms must be zero, and so we get
\[
\cQ(S) = \left\{\left(\begin{bmatrix} a & b \\ c & d \end{bmatrix},\begin{bmatrix} a' & b' \\ c' & d' \end{bmatrix}\right) \in \Mat_2(\ZZ)^2 \mid \begin{array}{l}a' \equiv a+2c \pmod{4} \\ b' \equiv 2a+c+2d \pmod{4} \\ c' \equiv c \pmod{4} \\ d'\equiv 2c+d \pmod{4} \end{array} \right\}.
\]
}

If we equip $\Mat_2(\cO)$ with the hyperbolic orthogonal involution
\[
\sigma(\begin{bmatrix} a & b \\ c & d \end{bmatrix}) = \begin{bmatrix} d & b \\ c & a \end{bmatrix},
\]
then $\sigma$ will descend to an involution $\theta$ on $\cQ$ since $\Inn(\alpha_{12})\circ \sigma = \sigma \circ \Inn(\alpha_{12})$. Notice that $\sigma$ is locally quadratic since it can be extended to a quadratic triple using $\ell = \begin{bmatrix} 1 & 0 \\ 0 & 0 \end{bmatrix} \in \Mat_2(\cO(S))$ as in Example \ref{quafoex}\ref{quafoex_b}. Therefore, $(\cQ,\theta)$ is an Azumaya $\cO$--algebra with locally quadratic involution since $\theta$ is isomorphic to $\sigma$ on $U_1$ and $U_2$.
\lv{%
The involution $\theta$ behaves on global sections as $\sigma$ on each component, i.e.,
\[
\theta\left(\begin{bmatrix} a & b \\ c & d \end{bmatrix},\begin{bmatrix} a' & b' \\ c' & d' \end{bmatrix}\right) = \left(\begin{bmatrix} d & b \\ c & a \end{bmatrix},\begin{bmatrix} d' & b' \\ c' & a' \end{bmatrix}\right).
\]
}
\begin{lem}\label{lem_example_3_nontrivial}
Consider the quaternion $\cO$--algebra $(\cQ,\theta)$ with locally quadratic involution constructed above. The weak obstruction $\Weak(\cQ,\theta)$ is non-trivial and hence the strong obstruction $\Strong(\cQ,\theta)$ is non-trivial as well.
\end{lem}
\begin{proof}
Assume there exists a quadratic triple $(\cQ,\theta,f)$. The restrictions of this quadratic triple to $U_1$ and $U_2$ is then a quadratic triple on an affine scheme, and so by Proposition \ref{lem_pair_affine}\ref{lem_pair_affine_i} they are described by elements $\ell_1 \in \cQ(U_1)$ and $\ell_2 \in \cQ(U_2)$ respectively. Since $\cQ(U_1) = \Mat_2(\ZZ[x]/\langle 4x \rangle)$ and $\ell_1+\sigma(\ell_1) = 1$, we can compute that $\ell_1$ must be of the form
\[
\ell_1 = \begin{bmatrix} a+g & 2h \\ 2k & 1-a-g \end{bmatrix}
\]
where $a \in \ZZ$ and $g,h,k \in \ZZ/4\ZZ[x]$ with no constant terms. Since $\ell_1$ is only determined up to an alternating element, we may assume $a+g=1$. Similarly, $\ell_2$ may be taken to be of the form
\[
\ell_2 = \begin{bmatrix} 1 & 2h' \\ 2k' & 0 \end{bmatrix}
\]
where $h',k' \in \ZZ/4\ZZ[\frac{1}{x}]$ with no constant terms. However, the maps $f|_{U_1}$ and $f|_{U_2}$ must agree on the overlap $U_{12}$. Since we glued together $\cQ$ using $\Inn(\alpha_{12})=\Inn(\begin{bmatrix} 1 & 2 \\ 0 & 1 \end{bmatrix})$, this is equivalent to requiring that 
\[
\Inn(\alpha_{12})(\ell_1)-\ell_2 \in \cAlt_{\cQ,\theta}(U_{12}) = \Alt(\Mat_2(\ZZ/4\ZZ[x,\tfrac{1}{x}]),\sigma)).
\]
We compute the following, keeping in mind that $2=-2$ in $\ZZ/4\ZZ[x,\frac{1}{x}]$.
\begin{align*}
\Inn(\alpha_{12})(\ell_1)-\ell_2 &= \begin{bmatrix} 1 & 2 \\ 0 & 1 \end{bmatrix}\begin{bmatrix} 1 & 2h \\ 2k & 0 \end{bmatrix}\begin{bmatrix} 1 & 2 \\ 0 & 1 \end{bmatrix} - \begin{bmatrix} 1 & 2h' \\ 2k' & 0 \end{bmatrix} \\
&= \begin{bmatrix} 1 & 2+2h \\ 2k & 0 \end{bmatrix} - \begin{bmatrix} 1 & 2h' \\ 2k' & 0 \end{bmatrix}\\
&= \begin{bmatrix} 0 & -2h'+2+2h \\ -2k'+2k & 0 \end{bmatrix}.
\end{align*}
However, we have that
\[
\Alt(\Mat_2(\ZZ/4\ZZ[x,\tfrac{1}{x}]),\sigma)) = \left\{\begin{bmatrix} y & 0 \\ 0 & -y \end{bmatrix} \mid y\in \ZZ/4\ZZ[x,\tfrac{1}{x}]\right\}
\]
and therefore $\alpha_{12}(\ell_1)-\ell_2$ cannot be an alternating element since $-2h'+2+2h \neq 0$ because $h'$ and $h$ have no constant terms. This is a contradiction, and therefore no such $f$ making $(\cQ,\theta,f)$ a quadratic triple exists. Thus $\Weak(\cQ,\theta)\neq 0$, and therefore also $\Strong(\cQ,\theta)\neq 0$, as claimed.
\end{proof}

Now we will consider the base change diagram
\[
\begin{tikzcd}
S'=S\times_{\Spec(\ZZ)} \Spec(\FF_2) \arrow{r} \arrow{d} & S \arrow{d} \\
\Spec(\FF_2) \arrow{r} & \Spec(\ZZ)
\end{tikzcd}
\]
\begin{lem}\label{lem_ex_base_change_trivial}
The pullback $S'=S\times_{\Spec(\ZZ)} \Spec(\FF_2)$ is isomorphic to projective space $\PP^1_{\FF_2}$. Furthermore, the restriction $\cE|_{S'}$ is isomorphic to $\cO|_{S'}^2$, meaning that $(\cQ,\theta)|_{S'} \cong (\Mat_2(\cO|_{S'}),\sigma)$ and therefore has trivial strong and weak obstructions.
\end{lem}
\begin{proof}
The scheme $S'$ is glued together from the affine schemes 
\begin{align*}
U_1\times_{\Spec(\ZZ)} \Spec(\FF_2) &= \Spec(\ZZ[x]/\langle 4x \rangle \otimes_\ZZ \FF_2) \cong \Spec(\FF_2[x]),\text{ and}\\
U_2\times_{\Spec(\ZZ)} \Spec(\FF_2) &= \Spec(\ZZ[\tfrac{1}{x}]/\langle \tfrac{4}{x} \rangle \otimes_\ZZ \FF_2) \cong \Spec(\FF_2[\tfrac{1}{x}])
\end{align*}
along the open subset $U_{12}\times_{\Spec(\ZZ)} \Spec(\FF_2) = \Spec(\FF_2[x,\frac{1}{x}])$ in the standard construction of projective space, and therefore $S' \cong \PP^1_{\FF_2}$.

The restriction of our quaternion algebra, $\cQ|_{S'}$ is the twist of $\Mat_2(\cO|_{S'})$ along the automorphism
\[
\Inn(\alpha_{12})\otimes 1 \colon \Mat_2(\FF_2[x,\tfrac{1}{x}]) \iso \Mat_2(\FF_2[x,\tfrac{1}{x}]).
\]
However, since $\Inn(\alpha_{12})=\Inn(\begin{bmatrix} 1 & 2 \\ 0 & 1\end{bmatrix})$, then $\Inn(\alpha_{12})\otimes 1$ will be its reduction modulo $2$ which is simply the identity. Therefore $\cQ|_{S'} = \Mat_2(\cO|_{S'})$ and also $\theta|_{S'} = \sigma$ is the standard hyperbolic orthogonal involution. As noted above, $\sigma$ can be extended to a quadratic triple using $\ell = \begin{bmatrix} 1 & 0 \\ 0 & 0 \end{bmatrix} \in \Mat_2(\cO|_{S'})$ which satisfies $\ell+\sigma(\ell)=1$ and therefore $\Strong(\Mat_2(\cO|_{S'}),\sigma)=0$, hence also $\Weak(\Mat_2(\cO|_{S'}),\sigma)=0$, so both obstructions are trivial as claimed.
\end{proof}

This example demonstrates two interesting facts. First, that obstructions may be non-trivial even for neutral Azumaya algebras. Second, examples of non-trivial obstructions are not confined to the setting of characteristic $2$ and allowing $2$ to be neither invertible nor zero provides examples of non-trivial obstructions of a different nature than those which occur in characteristic $2$.

\begin{remark}
Working over the same scheme $S$ constructed in Section \ref{sec_weird_example}, a neutral Azumaya $\cO$--algebra of any even degree with non-trivial obstructions can be constructed in a similar way. Twisting the $\cO$--module $\cO^{2n}$ along the same open cover using the automorphism
\[
\alpha_{12} = \begin{bmatrix} 1 & & 2 \\ & \ddots & \\ & & 1 \end{bmatrix}
\]
will produce a locally free $\cO$--module of rank $2n$. Twisting $(\Mat_{2n}(\cO),\sigma)$ by $\Inn(\alpha_{12})$, where $\sigma$ is the hyperbolic orthogonal involution which reflects across the second diagonal, will produce $(\cEnd_{\cO}(\cE),\theta)$ which is locally quadratic but has $\Weak(\cEnd_{\cO}(\cE),\theta)\neq 0$ and hence $\Strong(\cEnd_{\cO}(\cE),\theta)\neq 0$.
\end{remark}

Constructing higher degree examples of non-trivial obstructions using tensor products does not work, as the following lemma shows.
\begin{lem}
Let $(\cQ,\theta)$ be the quaternion $\cO$--algebra with locally quadratic involution constructed in Section \ref{sec_weird_example}. Let $(\cA,\theta')=(\cQ,\theta)\otimes_{\cO} \ldots \otimes_{\cO} (\cQ,\theta)$ be the tensor product of $n\geq 2$ copies of $(\cQ,\theta)$. Then $\Strong(\cA,\theta')=0$ and thus $\Weak(\cA,\theta')=0$ also.
\end{lem}
\begin{proof}
First, $(\cA,\theta')$ is indeed locally quadratic by Lemma \ref{tens_triple_invol_obs}\ref{obs1}. We show that $\Strong(\cA,\theta')=0$ when $n=2$, and then the claim for $n\geq 3$ follows from Lemma \ref{tens_triple_invol_obs}\ref{obs2}.

Over the cover $\{U_i \to S\}_{i=1,2}$, the algebra $(\cA,\theta')=(\cQ,\theta)\otimes_{\cO}(\cQ,\theta)$ is the twist of $\Mat_4(\cO)$ by the automorphism $\varphi =\Inn(\begin{bmatrix} 1 & 2 \\ 0 & 1 \end{bmatrix}\otimes \begin{bmatrix} 1 & 2 \\ 0 & 1 \end{bmatrix})$. We choose elements
\begin{align*}
\ell_1 &= \begin{bmatrix} 1 & 0 \\ 0 & 1 \end{bmatrix}\otimes \begin{bmatrix} 1 & 0 \\ 0 & 0 \end{bmatrix} \in \Mat_4(\ZZ[x]/\langle 4x \rangle) \\
\ell_2 &= \begin{bmatrix} 1 & 0 & 2 & 0 \\ 0 & 1 & 0 & -2 \\ 0 & 0 & 0 & 0 \\ 0 & 0 & 0 & 0 \end{bmatrix} \in \Mat_4(\ZZ[\tfrac{1}{x}]/\langle \tfrac{4}{x} \rangle).
\end{align*}
These both satisfy $\ell_i + (\sigma\otimes\sigma)(\ell_i)=1$. Since the involution $\theta' = \theta\otimes \theta$ is isomorphic to $\sigma\otimes \sigma$ over $U_{12}$, and furthermore
\[
\varphi(\ell_1|_{U_{12}})= \ell_2|_{U_{12}},
\]
these $\ell_i$ glue into a global section $\ell \in (\cA,\theta')$ with $\ell+\theta'(\ell)=1$. Hence $\Strong(\cA,\theta')=0$ as claimed.
\end{proof}


\begin{thebibliography}{SGA34}



\bibitem[Bas]{BR} H. Bass, {\it Lectures on topics in algebraic K-theory}, Notes by Amit Roy, Tata Institute of Fundamental Research Lectures on Mathematics, No. 41 (1967), Bombay.


\bibitem[B:A3]{BA8} N.~Bourbaki, {\em Alg\`ebre}, Ch.~8, second revised edition of the 1958 edition, Springer, Berlin 2012.



\bibitem[Br1]{Br1} M. Brion, {\it Homogeneous bundles over abelian
varieties},
J. Ramanujan Math. Soc. {\bf 27} (2012), 91-118.

\bibitem[Br2]{Br2} M. Brion, {\it
Homogeneous projective bundles over abelian varieties}
Algebra Number Theory {\bf 7} (2013), 2475-2510.

\bibitem[Br3]{Br3} M. Brion, {\it Homogeneous vector bundles over
abelian varieties via representation theory},
Representation Theory {\bf 20} (2020), 85-114.

\bibitem[CF]{CF} B. Calm\`es, J. Fasel, {\it Groupes classiques},
     Autour des sch\'emas en groupes, vol II, Panoramas et Synth\`eses {\bf 46} (2015),
      1-133.








\bibitem[EGA-I]{EGA-neu} A.~Grothendieck and J.~Dieudonn\'e, {\em \'El\'ements de g\'eom\'etrie alg\'ebrique: I. Le langage des schémas}, Grundlehren der Mathematischen Wissenschaften  \textbf{166} (2nd ed.). Berlin; New York: Springer-Verlag, 1971.



\bibitem[Fo]{Ford} T.~J. Ford, {\it Separable Algebras}, Graduate
    Studies in Mathematics {\bf 183}, Amer.~Math.~Soc, Provident, RI, (2017).




\bibitem[Gir]{Gir} J.~Giraud, {\it Cohomologie non ab\'elienne}, Die Grundlehren der Mathematischen Wissenschaften, Band 179, Berlin-Heidelberg-New York: Springer-Verlag, 1971.

\bibitem[GNR]{GNR} P.~Gille, E.~Neher, C.~Ruether, \emph{The Norm Functor over Schemes}, preprint: \arXiv{2401.15051}.

\bibitem[GP]{GP} P.~Gille, A.~Pianzola, {\it Galois cohomology and forms of algebras over Laurent polynomial rings}, Mathematische Annalen {\bf 338} (2007), 497-543.  

\bibitem[Gro]{Gro-Brau} A.~Grothendieck, \emph{Le Groupe de Brauer I}, S\'em. Bourbaki, 1964/65, no 290.

\bibitem[GW]{GW} U.~G\"ortz and T.~Wedhorn, {\em Algebraic Geometry I}, second edition, Springer Fachmedien Wiesbaden, 2020.



\bibitem[Kne]{Kn} M.~Knebusch, \emph{Symmetric bilinear forms over algebraic varieties}, Proc. {Conf}. {Quadratic} {Forms}, {Kingston} 1976, {Queen}'s {Pap}. pure appl. {Math}. 46, 103-283 (1977).


\bibitem[Knu]{K} M.-A.~Knus, {\it Quadratic and Hermitian Forms over
    Rings}, Grundlehren der mathematischen Wissenschaften {\bf 294}
    (1991), Springer.

\bibitem[KM]{KM} N.M.~Katz, B.~Mazur, {\it Arithmetic Moduli of Elliptic Curves}, Annals of Mathematics Studies {\bf 108}, Princeton University Press (1985).

\bibitem[KMRT]{KMRT} M.-A.~Knus, A.~Merkurjev, M.~Rost and J.-P.~Tignol,
    {\it   The   Book of Involutions}, American Mathematical Society Colloquium Publications \textbf{44}, American Mathematical Society, Providence, RI (1998).

 \bibitem[KO]{KO} M.-A.~Knus, M.~Ojanguren,
 {\it Th\'eorie de la Descente et Alg\`ebres d'Azumaya},
 Lecture Notes in Mathematics {\bf 389} (1974), Springer.



 \bibitem[KS]{KS} M.~Kashiwara, P.~Schapira, {\it
  Categories and Sheaves}, Grundlehren der mathematischen Wissenschaften {\bf  332} (2006), Springer.

\bibitem[L]{Liu} Q.~Liu, {\it Algebraic geometry and arithmetic curves}, Transl. by Reinie Erné, Oxford: Oxford University Press (2006).


\bibitem[M]{M}  J.-S.~Milne, {\'Etale Cohomology}, Princeton University Press, 1980.

\bibitem[Mu]{Mu} S. Mukai, {\it Semi-homogeneous vector bundles on an
abelian variety}, J. Math. Kyoto Univ. {\bf  18} (1978) 239-272.





\bibitem[RY]{RY} Z.~Reichstein, B.~Youssin, {\it Essential Dimensions of Algebraic Groups and a Resolution Theorem for $G$-Varieties}, Canadian Journal of Mathematics {\bf 52} (2000), 1018-1056.

\bibitem[RZ]{roze} A.~Rosenberg and D.~Zelinsky, {\em Automorphisms of separable algebras}, Pacific J.~Math.~\textbf{11} (1961), 1109--1117.


\bibitem[Sa]{Sah} C.~H.~Sah, {\em Symmetric bilinear forms and quadratic forms},     J. Algebra {\bf 20} (1972), 144-160.


\bibitem[Se]{Se} J.-P. Serre, {\it Sur la topologie des vari\'et\'es
alg\'ebriques en charact\'eristique $p$}, Symposium
Internacional de Topologia Algebraica, Universidad Nacional Autonoma de
Mexico,
1958, 24-53.

\bibitem[SGA3]{SGA3} {\it S\'eminaire de G\'eom\'etrie alg\'ebrique de
    l'I.H.E.S., 1963-1964, sch\'emas en groupes, dirig\'e par M. Demazure et A.
    Grothendieck},  Lecture Notes in Math. 151-153. Springer (1970).





\bibitem[St]{Stacks}  The Stacks Project Authors, {\em Stacks project},
\url{http://stacks.math.columbia.edu/}

\bibitem[Ti]{Tits} J.~Tits, \emph{Formes Quadratiques, Groupes Orthogonaux et alg\`ebres de Clifford}, Invent. Math. 5, 19-41 (1968).


\bibitem[W]{Waterhouse} W.C.~Waterhouse, \emph{A Unified Kummer-Artin-Schreier Sequence}, Math. Ann. 277, no. 3, 447-457 (1987).

\end{thebibliography}
\end{document}